\documentclass[a4paper,english]{amsart}
\usepackage[utf8]{inputenc}
\usepackage{babel}
\usepackage{verbatim}
\usepackage{mathtools}
\usepackage{amsthm}
\usepackage{amstext}
\usepackage{amssymb}
\usepackage{amsfonts}
\usepackage{esint}
\usepackage[unicode=true,
 bookmarks=false,
 breaklinks=false,pdfborder={0 0 1},backref=false,colorlinks=false]
 {hyperref}
\hypersetup{
 colorlinks,linkcolor=blue,anchorcolor=blue,citecolor=blue}
\usepackage{breakurl}
\usepackage{mathrsfs}

\makeatletter

\numberwithin{equation}{section}
\numberwithin{figure}{section}
\usepackage{enumitem}		% customizable list environments
\theoremstyle{plain}
\newtheorem{thm}{\protect\theoremname}[section]
  \theoremstyle{plain}
  \newtheorem{prop}[thm]{\protect\propositionname}
  \theoremstyle{remark}
  \newtheorem{rem}[thm]{\protect\remarkname}
  \newtheorem*{remn}{\protect\remarkname}
  \theoremstyle{plain}
  \newtheorem{cor}[thm]{\protect\corollaryname}
  \theoremstyle{plain}
  
  \theoremstyle{plain}
  \newtheorem{lem}[thm]{\protect\lemmaname}
  \theoremstyle{definition}
  \newtheorem{example}[thm]{\protect\examplename}
  \theoremstyle{definition}
    \newtheorem{defn}[thm]{\protect\definitionname}

\usepackage{babel}
\usepackage{a4wide}

  \providecommand{\propositionname}{Proposition}
\providecommand{\theoremname}{Theorem}

\newcommand{\irr}{\mathrm{Irr}(\mathbb{G})}
\newcommand{\f}{\varphi}

\newcommand{\pol}{\mathrm{Pol}(\mathbb{G})}

\DeclareMathOperator{\tr}{Tr}

\makeatother

  \providecommand{\corollaryname}{Corollary}
  \providecommand{\examplename}{Example}
  \providecommand{\factname}{Fact}
  \providecommand{\lemmaname}{Lemma}
  \providecommand{\propositionname}{Proposition}
  \providecommand{\remarkname}{Remark}
\providecommand{\theoremname}{Theorem}
\providecommand{\definitionname}{Definition}

\begin{document}
\title{Lacunary Fourier series for compact quantum groups}

\author{Simeng Wang}

\address{Laboratoire de Mathématiques, Université de Franche-Comté, 25030
Besançon Cedex, France and Institute of Mathematics, Polish Academy
of Sciences, ul. \'{S}niadeckich 8, 00-956 Warszawa, Poland}

\email{simeng.wang@univ-fcomte.fr}



\begin{abstract}
This paper is devoted to the study of Sidon sets, $\Lambda(p)$-sets and some related notions for compact quantum groups. We establish several different characterizations of Sidon sets, and in particular prove that any Sidon set in a discrete group is a strong Sidon set in the sense of Picardello. We give several relations between Sidon sets, $\Lambda(p)$-sets and lacunarities for $L^p$-Fourier multipliers, generalizing a previous work by Blendek and Michali\u{c}ek. We also prove the existence of $\Lambda(p)$-sets for orthogonal systems in noncommutative $L^p$-spaces, and deduce the corresponding properties for compact quantum groups. Central Sidon sets are also discussed, and it turns out that the compact quantum groups with the same fusion rules and the same dimension functions have identical central Sidon sets. Several examples are also included.
\end{abstract}

\maketitle

\section*{Introduction}
The study of lacunarity and particularly of Sidon sets is a major and fascinating subject of harmonic analysis. Historically, the notion of Sidon sets originated from discussions of special lacunary series on the circle  $\mathbb T$. This has been generalized later by replacing $\mathbb T$ with an arbitrary compact abelian group $G$. 
Recall that for a compact abelian group $G$, a Sidon set $\mathbf E$ for $G$ is a subset of the dual discrete group $\Gamma=\hat G$ such that any continuous function on $G$ with Fourier transform supported on $\mathbf E$ has absolutely convergent Fourier series. The theory of Sidon sets and lacunarity for compact abelian groups has been remarkably developed in the past several decades, and in a series of works of Rudin, Drury, Rider, Marcus, Pisier and others (see  \cite{drury70sidon,rider1975randomsidon,rudin1960sidon,marcuspisier1981randomfourier,pisier1978sidon,pisier78aleatoireseminaire,pisier83entropysidon,pisier83entropy} and references therein); it is shown to be deeply related to the theory of random Fourier series, metric entropy condition, multiplier spaces as well as some other topics in harmonic analysis and Banach space theory.

The development of similar subjects in more general setting goes naturally into two lines. The first one is the theory of Sidon sets in the setting of a  non-abelian compact group $G$; these sets are special subsets of irreducible representations of $G$. This generalized notion of Sidon sets was firstly introduced in \cite{figrider66lacunarycpt,hewittross1970abstract} and then has seen many links to the harmonic analysis and probability theory as in the abelian case  (see \cite{bozejko74drurycptgroup,marcuspisier1981randomfourier}). The other line of development concerns subsets of an arbitrary non-abelian discrete group $\Gamma$ with related \textquotedblleft functions" in its group von Neumann algebra $VN(\Gamma)$, such as in \cite{picardello73lacunary,fig77lacunary,bozejko81newlacunary}. The latter is more complicated and the behavior of lacunarity is closely related to the amenability of the discrete group, which also involves some tools from the operator space theory and leads to interesting topics in the study of completely bounded $L^p$-Fourier multipliers in  abstract harmonic analysis (\cite{harcharras99nclambdap,pisier1995multiplier}).

In the both directions mentioned above, the Sidon sets admit many seemingly non-related characterizations, and enjoy some special relations with $\Lambda(p)$-sets, multipliers, etc.. Moreover, many basic properties of Sidon sets in the two different settings have quite similar formulations. This gives rise to a natural motivation to find a more general framework, say Woronowicz's compact quantum groups, to unify these developments from different viewpoints, and to seek new aspects of the theory for these objects. In fact, based on the Woronowicz-Peter-Weyl theory for compact quantum groups, we may discuss similar questions for Sidon sets in the quantum group setting. Note that a priori in this general view, a discrete group $\Gamma$ is regarded as the set of irreducible representations of the dual compact quantum group $\mathbb G=\hat \Gamma$, and the amenability of $\Gamma$ is often interpreted as the coamenability of $\mathbb G$. Recently, Blendek and Michali\u{c}ek \cite{blendekmichalicek2013sidonl1} have obtained some partial results towards this direction, which we will significantly complete and improve in Section 4.

In this paper we introduce and develop some important notions and properties related to lacunarities in the framework of compact quantum groups. The subjects that we address are as follows:

(1) \emph{Generalizations and characterizations of Sidon sets.} With some preliminary work on Fourier series, we introduce in Section 3 the notion of Sidon sets for a compact quantum group and as an analogue of the classical case, we give some first characterizations of Sidon sets via interpolation of Fourier series of bounded functionals or $L^1$-functions (Theorem \ref{thm:sidon_l1_interpolation state}). Theorem \ref{thm:sidon_l1_interpolation state} answers in particular a basic question in the study on lacunarity of discrete groups, raised in \cite{fig77lacunary,picardello73lacunary}, that is, the equivalence between strong Sidon sets and Sidon sets in discrete groups in the sense of \cite{picardello73lacunary}. In fact, in the latter article Picardello defined two different kinds of lacunary sets that he  called strong Sidon sets and Sidon sets in discrete groups. He proved the equivalence between these two notions for amenable discrete groups, but the non-amenable case has been left open. Here our approach is different from the classical ones in \cite{hewittross1970abstract,picardello73lacunary}. Our  argument is simpler and avoids the use of coamenability as in \cite{picardello73lacunary}. It as well applies to the general quantum setting. We show as well that the Sidon property is stable under Cartesian/free products of compact quantum groups. Apart from the above approach, there are indeed various viewpoints on generalizations of Sidon sets for non-coamenable compact quantum groups, which lead to different types of lacunarities such as weak Sidon sets and unconditional Sidon sets. In Theorem \ref{sidon and u sidon} we will discuss the relations between these various notions and prove the equivalence among them for coamenable compact quantum groups. 

(2) \emph{$\Lambda(p)$-sets and $L^p$-Fourier multipliers.} We discuss the equivalence between $\Lambda(p)$-sets and interpolation sets of bounded $L^p$-Fourier multipliers for a compact quantum group (Theorem \ref{thm:lambda inter}), generalizing the previous work \cite{harcharras99nclambdap}. Our investigation also leads to some facts which are hidden in the classical cases: we show that the restriction of the modular element $Q$ of the dual quantum group on a $\Lambda(p)$-set must be bounded (Proposition \ref{prop: unif bdd of mod gp for sidon}), and that the class of $\Lambda(p)$-sets on a compact quantum group is independent of the interpolation parameter in the construction of noncommutative $L^p$-spaces (Proposition \ref{prop:indep lambda parameter}).

(3) \emph{Relations between Sidon sets and $\Lambda(p)$-sets.} For a compact group or (the dual of) a discrete group, a simple (but non-trivial) argument shows that any Sidon set $\mathbf E$ is of type $\Lambda(p)$ for $1< p<\infty$, which indeed means that any two norms $\|\,\|_p,\|\,\|_{p'}$ for $1< p,p'<\infty$ are equivalent on the subspace of polynomials with Fourier series supported in $\mathbf E$. The case for a general compact quantum group turns out to be more difficult. A first attempt was made by Blendek and Michali\u{c}ek \cite{blendekmichalicek2013sidonl1}, who showed in 2013 that if $\mathbb G$ is a compact quantum group \emph{of Kac type} and if the Sidon set $\mathbf E$ satisfies some other special lacunarity conditions, then the norms $\|\,\|_1$ and $\|\,\|_2$ are equivalent on the subspace of \emph{central} polynomials related to $\mathbf E$. In this paper we give a shorter argument which completely solves the problem by showing that any Sidon set for an arbitrary compact quantum group is a $\Lambda(p)$-set for $1< p <\infty$ (see Theorem \ref{sidon lambda} and Corollary \ref{cor:cor sidon lambda}). 

(4) \emph{Existence of $\Lambda(p)$-sets.}
We prove in Theorem \ref{thm: exist lambda p} that any infinite subset of irreducible representations with uniformly bounded dimensions admits an infinite $\Lambda(p)$-set. The argument is based on the general fact that any infinite uniformly bounded orthogonal system with respect to a normal faithful state in a von Neumann algebra (including type III) admits an infinite $\Lambda(p)$-set. Since this result may be of independent interest, we include a proof in the appendix. 
On the other hand, contrary to the classical case, we may find some nice central $\Lambda(4)$-sets for $\mathrm{SU}_q(2)$ with $0<q<1$. This will be given in Proposition \ref{central lambda four su}.

(5) \emph{Central Sidon sets.} We also investigate in Section 5 some basic notions and facts concerning central Sidon sets for compact quantum groups. Our argument includes some new characterizations of compact quantum groups of Kac type via bounded central functionals or conditional expectations onto central functions. We show that any two compact quantum groups with the same fusion rules and the same dimension functions have identical central Sidon sets. This gives many examples of central Sidon sets for compact quantum groups via monoidal equivalence, Drinfeld-Jimbo deformations, etc.. As a corollary, we see that the Drinfeld-Jimbo $q$-deformation of any compact simply connected semi-simple Lie group (e.g. $\mathrm{SU}_q(2)$) does not admit any infinite Sidon set.

We would also like to remark that many results mentioned above rely on some very elementary properties of bounded $L^p$-Fourier multipliers on compact quantum groups, which have not been discussed so far in literature. The first property concerns the question how to construct a left bounded $L^p$-multiplier from a right bounded $L^p$-multiplier, and the second one concerns the inequality comparing the norms $\|a\|_{\ell^\infty (\hat{\mathbb G})}$ and $\|a\|_{\mathrm M (L^p(\mathbb G))}$. These questions are easy to deal with for completely bounded $L^\infty$-multipliers, but it seems that they are not obvious for other cases. Although the completely bounded $L^\infty$-multipliers, instead of bounded $L^\infty$-multipliers, are more natural objects for studying harmonic analysis on quantum groups, the  questions above are still natural and elementary for studying $L^p$-Fourier analysis on quantum groups, especially for $p<\infty$. As a result we include a detailed argument on these basic facts (Lemma \ref{lem: left to right multipliers}, Proposition \ref{prop:multiplier bdd quantum}) in Section 2, which are also frequently used in other proofs of this paper.

It would also be interesting to study the completely bounded version of Sidon sets or $\Lambda(p)$-sets. We refer to Pisier's work \cite{pisier1995multiplier} for completely bounded lacunarity in discrete groups. We have not studied here this topic; but we will pursue it elsewhere. In fact, our argument on $L^p$-Fourier multipliers in Section 2, and the estimation of modular elements in Section 4, will be helpful for the study towards this direction. 
\section{Preliminaries}
\subsection{Compact quantum groups} 
\subsubsection{Basic notions} 
Let us first recall some well-known definitions and properties concerning compact quantum groups. We refer to  \cite{woronowicz1998note} and \cite{maesvandaele1998note} for more details.

\begin{defn}
Let  $A$ be a unital C{*}-algebra. If there exists a unital $*$-homomorphism
$\Delta:A\to A\otimes A$ such that
$(\Delta\otimes\iota)\Delta=(\iota\otimes\Delta)\Delta$ and 
\[
\{\Delta(a)(1\otimes b):a,b\in A\}\quad\text{and}\quad\{\Delta(a)(b\otimes1):a,b\in A\}
\]
are linearly dense in $A\otimes A$, then $(A,\Delta)$ is called
a \emph{compact quantum group} and $\Delta$ is called the \emph{comultiplication} on $A$. We denote $\mathbb{G}=(A,\Delta)$
and $A=C(\mathbb{G})$.
\end{defn}

Any compact quantum group $\mathbb G$ admits a unique \emph{Haar
state} $h$ on $C(\mathbb{G})$ such that for all $x\in C(\mathbb{G})$,
\[
(h\otimes\mathrm{\iota})\circ\Delta(x)=h(x)1=(\iota\otimes h)\circ\Delta(x).
\]

Consider an element $u\in C(\mathbb{G})\otimes B(H)$, where $H$ is a Hilbert space with $\dim H=n$. We identify
$C(\mathbb{G})\otimes B(H)=\mathbb{M}_{n}(C(\mathbb{G}))$ and write $u=[u_{ij}]_{i,j=1}^{n}$.
The matrix $u$ is called an \emph{$n$-dimensional representation }of $\mathbb{G}$
if for all $j,k=1,...,n$ we have 
\begin{equation*}
\Delta(u_{jk})=\sum_{p=1}^{n}u_{jp}\otimes u_{pk}.\label{eq:comultiplication}
\end{equation*}
Denote by $\mathrm{Irr}(\mathbb{G})$ 
the set of unitary equivalence classes of irreducible finite-dimensional unitary representations
of $\mathbb{G}$. Also denote by $\mathrm{Rep}(\mathbb G)$ the set of unitary equivalence classes of (not necessarily irreducible) finite-dimensional unitary representations
of $\mathbb{G}$. For each $\pi\in\mathrm{Irr}(\mathbb{G})$, we fix a representative
$u^{(\pi)}\in C(\mathbb{G})\otimes B(H_{\pi})$ 
of the class $\pi$ where $H_{\pi}$ is the finite dimensional
Hilbert space on which $u^{(\pi)}$ acts. In the sequel we write $n_\pi=\dim \pi =\dim H_\pi$ for $\pi\in\irr$.

For $\pi,\pi'\in \mathrm{Irr}(\mathbb{G})$, define the tensor product representation $\pi\otimes\pi'$ on $H_{\pi}\otimes H_{\pi'}$ by 
$$u^{(\pi\otimes \pi')}=\sum_{i,j,k,l}u_{ij}^{(\pi)}u_{kl}^{(\pi')}\otimes e_{ij}^{(\pi)}\otimes e_{kl}^{(\pi')},$$
where $e_{ij}^{(\pi)}, e_{kl}^{(\pi')}$ denotes the matrix units of $B(H_{\pi})$ and $B( H_{\pi'})$ respectively.
For each $\pi\in \mathrm{Irr}(\mathbb{G})$, there exists a unique $\bar{\pi }\in \mathrm{Irr}(\mathbb{G})$ such that the trivial representation $1\in \irr$ is a subrepresentation of $\pi\otimes\bar \pi$. We call $\bar \pi$ the adjoint of $\pi$. The notions of tensor product and adjoint then can be extended to all elements in $\mathrm{Rep}(\mathbb G)$ by decomposing representations into irreducible ones.

For $\pi\in\mathrm{Rep}(\mathbb G)$, the character of $\pi$ is the element $\chi_\pi=\sum_{i=1}^{n_\pi}u_{ii}^{(\pi)}\in C(\mathbb G)$. One can show that the definition does not depend on the representative matrix $u^{(\pi)}$. For $\pi,\pi'\in\mathrm{Rep}(\mathbb G)$, we have 
\begin{equation}\label{character of rep}
\chi_{\pi\oplus\pi'}=\chi_{\pi}+\chi_{\pi'},\quad \chi_{\pi\otimes\pi'}=\chi_{\pi}\chi_{\pi'},\quad \chi_{\bar{\pi}}=\chi_\pi^*.
\end{equation}

Denote  $\mathrm{Pol}(\mathbb{G})=\mathrm{span}\{u_{ij}^{(\pi)}:u^{(\pi)}=[u_{ij}^{(\pi)}]_{i,j=1}^{n_{\pi}},\pi\in\mathrm{Irr}(\mathbb{G})\}$. This is a dense subalgebra of $C(\mathbb{G})$. Consider the GNS representation $(\pi_{h},H_{h})$ of the Haar state $h$,
then $\mathrm{Pol}(\mathbb{G})$ can be viewed as a subalgebra of
$B(H_{h})$. Define $C_{r}(\mathbb{G})$ (resp., $L^{\infty}(\mathbb{G})$)
to be the C{*}-algebra (resp., the von Neumann algebra) generated
by $\mathrm{Pol}(\mathbb{G})$ in $B(H_{h})$. Then $h$ extends to
a normal faithful state on $L^{\infty}(\mathbb{G})$. On the other hand, we may equip the following C*-norm on $\mathrm{Pol}(\mathbb{G})$,
$$\|x\|_u=\sup\{p(x):p\text{ is a C*-seminorm on }\mathrm{Pol}(\mathbb{G})\},\quad x\in \mathrm{Pol}(\mathbb{G}).$$
Then the corresponding completion of $\mathrm{Pol}(\mathbb{G})$ is a unital C*-algebra, denoted by $C_u(\mathbb G)$.

It is known that there exists a linear antihomomorphism $S$ on $\mathrm{Pol}(\mathbb{G})$, called the \emph{antipode} of $\mathbb{G}$, 
determined by 
$$S(u_{ij}^{(\pi)})=(u_{ji}^{(\pi)})^*,\quad u^{(\pi)}=[u_{ij}^{(\pi)}]_{i,j=1}^{n_{\pi}},\ \pi\in\mathrm{Irr}(\mathbb{G}).$$ 
Also, let $\epsilon$ be the \emph{counit}  of $\mathrm{Pol}(\mathbb{G})$, i.e., the linear functional defined by 
$$\epsilon(u_{ij}^{(\pi)})=\delta_{ij},\quad u^{(\pi)}=[u_{ij}^{(\pi)}]_{i,j=1}^{n_{\pi}},\ \pi\in\mathrm{Irr}(\mathbb{G}).$$

Let $(\mathbb{G}_{i}:i\in I)$ be a family of compact quantum groups. It is shown in \cite{wang95tensor} that there exists a compact quantum group, denoted by $\prod_{i\in I}\mathbb{G}_{i}$,
such that the algebra $C_r(\prod_{i\in I}\mathbb{G}_{i})$ is the minimal
tensor product of the C{*}-algebras $C_r(\mathbb{G}_{i})$, and the comultiplication on $\prod_{i\in I}\mathbb{G}_{i}$ is induced by those on $\mathbb{G}_{i}$'s. We call $\prod_{i\in I}\mathbb{G}_{i}$ the Cartesian product of $(\mathbb{G}_{i})$. Each irreducible representation $\pi\in\mathrm{Irr}(\mathbb{G}_{i})$
can be naturally viewed as an irreducible representation of $\prod_{i\in I}\mathbb{G}_{i}$,
still denoted by $\pi$. In this way we may write $\mathrm{Irr}(\prod_{i\in I}\mathbb{G}_{i})=\{\otimes_{i\in I}\pi_i:\pi_i\in \mathrm{Irr}(\mathbb G _{i}),i\in I\}$. On the other hand, according to \cite{wang1995freeprod}, we may construct a compact quantum group denoted by $\hat{*}_{i\in I}\mathbb{G}_{i}$ and called the dual free product of $(\mathbb{G}_{i})$, 
such that the algebra $C_r(\hat{*}_{i\in I}\mathbb{G}_{i})$ is the reduced
free product of the C{*}-algebras $C_r(\mathbb{G}_{i})$ associated to the Haar states of $\mathbb{G}_{i}$. Again, each irreducible representation $\pi\in\mathrm{Irr}(\mathbb{G}_{i})$
can be naturally viewed as an irreducible representation of $\hat{*}_{i\in I}\mathbb{G}_{i}$,
still denoted by $\pi$.

\subsubsection{Modular properties of the Haar state}
It is well-known that for each $\pi\in\irr$ there
exists a unique positive invertible operator $Q_{\pi}\in B(H_{\pi})$
with $\mathrm{Tr}(Q_{\pi})=\mathrm{Tr}(Q_{\pi}^{-1})\coloneqq d_{\pi}$ that intertwines $u^{(\pi)} $ and $(S^2\otimes\iota) (u^{(\pi)})$. Then the Haar state can be calculated as follows, 
\begin{equation}
h(u_{ij}^{(\pi)}(u_{lm}^{(\pi')})^{*})=\delta_{\pi\pi'}\delta_{il}\dfrac{(Q_{\pi})_{mj}}{d_{\pi}},\quad h((u_{ij}^{(\pi)})^{*}u_{lm}^{(\pi')})=\delta_{\pi\pi'}\delta_{jm}\dfrac{(Q_{\pi}^{-1})_{li}}{d_{\pi}},\label{eq:haar state def}
\end{equation}
where $\pi'\in\irr$, $1\leq i,j\leq n_{\pi}$, $1\leq l,m\leq n_{\pi'}$. The number $d_\pi$ is called the \emph{quantum dimension} of $\pi$. The quantum group 
$\mathbb G$ is said to be \emph{of Kac type} if $Q_\pi=\mathrm{Id}_\pi$ for all $\pi\in \irr$. The Woronowicz characters on $\mathrm{Pol}(\mathbb G)$ are defined as
$$f_z(u_{ij}^{(\pi)})=(Q_\pi^z)_{ij},\quad z\in \mathbb C, \pi\in\irr, 1\leq i,j\leq n_{\pi}.$$ 
We denote
$$\Delta^{(2)}\coloneqq (\Delta\otimes\iota)\Delta=(\iota\otimes\Delta)\Delta.$$
The modular automorphism group of the Haar state $h$ on $L^\infty (\mathbb{G})$ is determined by the following formula:
$$\sigma_z(x)=(f_{\mathrm i z}\otimes\iota \otimes f_{\mathrm i z})\Delta^{(2)}(x),\quad  x\in \mathrm{Pol}(\mathbb G),z\in\mathbb C,$$
in other words,
\begin{equation}\label{modular group on cqg}
(\sigma_{z}\otimes\iota) (u^{(\pi)})= (1\otimes Q_\pi^{\mathrm i z})u^{(\pi)}(1\otimes Q_\pi^{\mathrm i z}),\quad \pi\in\irr.
\end{equation}
The antipode $S$ has the following polar decomposition 
\begin{equation}\label{eq: polar decomp of s}
S=R\circ \tau_{-\frac{\mathrm i}{2}}=  \tau_{-\frac{\mathrm i}{2}} \circ R,
\end{equation}
where $R$ is a $*$-antiautomorphism of $C_r(\mathbb G)$ and $(\tau_z)_{z\in \mathbb C}$ is the analytic extension of the one-parameter group $(\tau_t)_{t\in \mathbb R}$ of $*$-automorphisms defined as 
$$\tau_z(x)= (f_{\mathrm i z}\otimes\iota \otimes f_{-\mathrm i z})\Delta^{(2)}(x)$$
for $x\in \mathrm{Pol}(\mathbb G)$, or in other words,
\begin{equation}\label{eq: scaling group on cqg}
(\tau_{z}\otimes\iota) (u^{(\pi)})= (1\otimes Q_\pi^{\mathrm i z})u^{(\pi)}(1\otimes Q_\pi^{-\mathrm i z}),\quad \pi\in\irr,
\end{equation}
and moreover
\begin{equation}
\label{antipode2}
 S^2= \tau_{-\mathrm i},\quad \Delta \circ R=\Sigma \circ  (R\otimes R)\circ \Delta.
\end{equation} 
The dual quantum group $\hat{\mathbb G}$ of $\mathbb G$ is defined via its ``algebras
of functions'',
\[c_{0}(\hat{\mathbb{G}})=\oplus_{\pi\in\irr}^{c_{0}}B(H_{\pi}),\quad 
\ell^{\infty}(\hat{\mathbb{G}})=\oplus_{\pi\in\irr}B(H_{\pi}),
\]
where $\oplus_{\pi}B(H_{\pi})$ refers to the direct sum of
$B(H_{\pi})$, i.e. the bounded families $(x_{\pi})_{\pi}$
with each $x_{\pi}$ in $B(H_{\pi})$, and $ \oplus_{\pi\in\irr}^{c_{0}}B(H_{\pi}) $
  corresponds to the subalgebra of bounded families converging to 0
  at infinity. Also set $ c_{c}(\hat{\mathbb{G}}) $
  to be the corresponding algebraic direct sum and denote by $ \prod_{\pi}B(H_{\pi}) $
  the usual Cartesian product. We will not recall the full quantum group structure on $\hat{\mathbb{G}}$ as we do
not need it in the following. We only remark that the (left) Haar
weight $\hat{h}$ on $\hat{\mathbb G}$ can be explicitly given by (see
e.g. \cite[Section 5]{vandaele1996discrete}) 
\[
\hat{h}:\ell^{\infty}(\hat{\mathbb G})\ni x\mapsto\sum_{\pi\in\irr}d_{\pi}\tr(Q_{\pi}p_{\pi}x),
\]
where $p_{\pi}$ is the projection onto $H_{\pi}$ and $\mathrm{Tr}$
denotes the usual trace on $B(H_{\pi})$ for each $\pi$.

\subsubsection{Coamenability} Consider a compact quantum group $\mathbb G$. We say that $\mathbb G$ is \emph{coamenable} if the counit $\epsilon:\mathrm{Pol}(\mathbb G)\to \mathbb C $ extends to a state on $C_r(\mathbb G)$. On the other hand, for linear functionals $\f,\f'$ on $\mathrm{Pol}(\mathbb{G})$, we define the convolution product 
\begin{equation}\label{conv def}
\f\star\f'=(\f\otimes\f')\circ\Delta.
\end{equation}
Then it is easy to see that
\[
\|\varphi_{1}\star\varphi_{2}\|\leq\|\varphi_{1}\|\|\varphi_{2}\|,
\] 
where the norm is induced from $C_r(\mathbb G)^*$ or $L^\infty (\mathbb G)^*$, when the respective functionals admit  bounded extensions.
Equipped with this convolution product, the predual space $L^\infty(\mathbb G)_*$  forms a Banach algebra. The following characterization of coamenability is given in \cite[Theorem 3.1]{bedostuset03amenable}.

\begin{prop}
\label{amenability}
A compact quantum group $\mathbb G$ is coamenable if and only if the Banach algebra $L^\infty(\mathbb G)_*$ has a bounded right approximate unit with norm not more than $1$, and if and only if the identity map on $\mathrm{Pol}(\mathbb G)$ extends to a $*$-isomorphism from $C_u(\mathbb G)$ to $C_r(\mathbb G)$.
\end{prop}

We remark that any compact group $G$ or compact quantum group $\mathbb G$ with an amenable discrete group $\Gamma$ as the dual quantum group, is coamenable.

\subsubsection{Drinfeld-Jimbo deformation, quantum $\mathrm{SU}(N)$ groups}
Let $ G $ be a simply connected semi-simple compact Lie group. It follows from the work of Levendorskii and Soibelman \cite{levsoi91qdeformation,soibelman90qdeformation} that given any $ q > 0 $ one can define
a compact quantum group $G_q$, called the \emph{Drinfeld-Jimbo $ q $-deformation} of  $G$ such that the fusion rules, the classical dimension function and the coamenablity do not depend on $q$. More precisely, we may state the following property (see for example \cite[Theorem 2.4.7]{neshveyevtuset13qgbook}, \cite{banica99fusion} and references therein). 

\begin{prop}
\label{deformation fusion rule}
Let $0<q<1$. 

\emph{(1)} There exists a bijection $\Phi :\mathrm{Rep}(G)\to \mathrm{Rep}(G_q)$ such that 
$$\Phi(\pi\otimes\pi')=\Phi(\pi)\otimes \Phi(\pi'),\quad \Phi(\oplus_{i\in I}\pi_i)=\oplus_{i\in I}\Phi(\pi_i), \quad \pi,\pi',\pi_i\in \mathrm{Rep}(G)$$
and $\dim \Phi(\pi)=\dim \pi$ for all $\pi\in \mathrm{Irr}(G)$.

\emph{(2)} $G_q$ is coamenable.
\end{prop}

Take $G$ to be the special unitary group $\mathrm{SU}(N)$ of degree $N$. We denote by $\mathrm{SU}_q(N)$ the compact quantum group $G_q$ for $0<q<1$. Let us recall some facts of the representation theory of $\mathrm{SU}_q(2)$. The elements of $\mathrm{Irr}(\mathrm{SU}_{q}(2))$
can be indexed by $n\in\mathbb{N}\cup \{0\}$ and each representation
$u^{(n)}$ is of dimension $n+1$. The associated matrix $Q_{n}$ in \eqref{eq:haar state def} can
be represented as a diagonal under some appropriate basis (see for example Theorem 17 in \cite[Sect.4.3.2]{klimyksch97qgrepbook}):
\begin{equation}\label{qn for sutwo}
Q_{n}=\begin{bmatrix}q^{-n}\\
 & q^{-n+2}\\
 &  & \ddots\\
 &  &  & q^{n-2}\\
 &  &  &  & q^{n}
\end{bmatrix}.
\end{equation}
%
%Recall
%\[
%\sigma_{z}(x)=D^{\mathrm{i}z}xD^{-\mathrm{i}z},\quad x\in L_{\infty}(\mathbb{G}),\ z\in\mathbb{C}.
%\]
Write $\chi_{n}=\sum_{i}u_{ii}^{(n)}$ to be the character of 
$u^{(n)}$. We recall the property below. See for example Proposition 6.2.10 in \cite{timmermann08qgbook} for the proof.
\begin{prop}
\emph{(1)} For $m,m'\in\mathbb{N}\cup \{0\}$,  $\chi_{m}\chi_{m'}=\chi_{|m-m'|}+\chi_{|m-m'|+1}+\cdots+\chi_{m+m'}$;

\emph{(2)} For $n\in\mathbb{N}\cup \{0\}$, $\chi_{n}=\chi_{n}^{*}$.
\end{prop}

\subsection{Noncommutative $L^p$-spaces}\label{subs:nc Lp}
In this subsection we recall some basic definitions and facts on noncommutative $L^p$-spaces. We refer to \cite{takesaki2002opeI,takesaki2003oa2}
for the theory of von Neumann algebras and to \cite{pisierxu2003nclp}
for more details on noncommutative $L_{p}$-spaces. In this paper we will mainly use the construction via interpolation \cite{kosaki84interpolation} and we refer to \cite{bergh76interpolation} for all notions and notation from interpolation theory used below. 

Let $\mathcal{M}$ be a von Neumann algebra equipped with a distinguished normal faithful state $\varphi$. Denote by $\mathcal{M}_*$ the predual space of $\mathcal{M}$. Define $L^1(\mathcal{M},\varphi)=\mathcal{M}_*$ and $L^\infty (\mathcal{M},\varphi)=\mathcal{M}$. We identify $\mathcal{M}$ as a subspace
of $\mathcal{M}_{*}$ by the following injection 
\[
j:\mathcal{M}\to \mathcal{M}_{*},\quad j(x)=x\varphi\coloneqq \varphi(\cdot\, x),\quad x\in \mathcal{M}.
\]
It is known that $j$ is a contractive%
\begin{comment}
since h is of norm 1 
\end{comment}
{} injection with dense image%
\begin{comment}
by H-B thm 
\end{comment}
. In this way we may view $(\mathcal{M},\mathcal{M}_*)$ as a compatible pair of Banach spaces and for $1<p<\infty $, we introduce the corresponding noncommutative $L^p$-space as 
$$L^p(\mathcal{M},\varphi)=(\mathcal{M},\mathcal{M}_*)_{1/p},$$
where $(\cdot,\cdot)_{1/p}$ denotes the complex interpolation space. Denote by $\|\cdot\|_p$ the norm on $L^p(\mathcal{M},\varphi)$. Let $H_\varphi$ be the Hilbert space in the GNS construction induced by $\varphi$. Then $$L^2(\mathcal{M},\varphi)=H_\varphi$$ with equal norms.

Another useful and equivalent construction of noncommutative $L^p$-spaces is given by Haagerup. We refer to \cite{terp1981lp,pisierxu2003nclp,haagerupjx10reduction} for details. For the convenience of the reader, we will not formulate the precise definition but rather cite here several basic properties which will be sufficient for later use. 
%but for later use we need to recall the following correspondence between the two different approaches established in \cite[Sect.9]{kosaki84interpolation}.
 
Assume that the von Neumann algebra $\mathcal{M}$ acts on a Hilbert space ${H}$ and denote by $\sigma=\sigma^\varphi$ the modular automorphism group of $\varphi$. 
%The construction of Haagerup uses the crossed product method to obtain a von Neumann subalgebra  $\mathcal{R}\subset B(L^2(\mathbb{R},{H}))$ equipped with a distinguished normal semifinite faithful tracial weight $\tau$, and $M$ can be  naturally identified with a von Neumann subalgebra of $\mathcal{R}$. Let $L^0(\mathcal{R},\tau)$ denote the topological involutive algebra of all operators on $L^2(\mathbb{R},{H})$ measurable with respect to $(\mathcal{R},\tau)$ (cf. \cite[Chap.1]{terp1981lp}).
For $1\leq p\leq \infty $, denote by $L^{p,\mathsf{H}}(\mathcal{M},\varphi)$ the corresponding Haagerup $L^p$-space with norm $\|\cdot\|_{p,\mathsf{H}}$. Recall that each element in $L^{p,\mathsf{H}}(\mathcal{M},\varphi)$ is realized as a densely defined operator on $L^2(\mathbb{R},{H})$ and the usual H\"older inequality also holds for these noncommutative $L^p$-spaces in this sense. Also note that $\mathcal{M}$ can be identified with $L^{\infty,\mathsf H}(\mathcal{M},\varphi)$.

Let $D$ be the density operator associated to $\varphi$ and $\mathrm{tr}$ be the trace on $L^{1,\mathsf{H}}(\mathcal{M},\varphi)$. We recall that $D$ is a distinguished invertible positive selfadjoint operator on $L^2(\mathbb{R},{H})$  and $\mathrm{tr}$ is a distinguished positive functional on $L^{1,\mathsf{H}}(\mathcal{M},\varphi)$, which enjoy the following properties.

\begin{prop}
\label{density op haagerup Lp}
\emph{(1)} For all $x\in \mathcal{M},t\in \mathbb{R}$, $\sigma_t(x)=D^{\mathrm i t}xD^{-\mathrm i t}$;

\emph{(2)} Let $1\leq p,q\leq \infty $ be such that $1/p+1/q=1$. Then for $x\in L^{p,\mathsf{H}}(\mathcal{M},\varphi), y\in L^{q,\mathsf{H}}(\mathcal{M},\varphi)$, we have $xy,yx\in L^{1,\mathsf{H}}(\mathcal{M},\varphi)$ and $\mathrm{tr}(xy)=\mathrm{tr}(yx)$;

\emph{(3)} $D\in L^{1,\mathsf{H}}(\mathcal{M},\varphi)$ and $\varphi(x)=\mathrm{tr}(xD)$ for $x\in \mathcal{M}$;

\emph{(4)} For $x\in L^{1,\mathsf{H}}(\mathcal{M},\varphi)$, $\mathrm{tr}(|x|)=\|x\|_{1,\mathsf{H}}$. For $x\in L^{p,\mathsf{H}}(\mathcal{M},\varphi)$ with $1\leq p<\infty$, we have $$|x|^p\in L^{1,\mathsf{H}}(\mathcal{M},\varphi),\quad \|x\|_{p,\mathsf{H}}=\|x^*\|_{p,\mathsf{H}}=\||x|\|_{p,\mathsf{H}}=\||x|^p\|_{1,\mathsf{H}}^{1/p};$$
and for $x\in \mathcal M$, $1\leq p\leq p'\leq \infty$, we have $$xD^{1/p}\in L^{p,\mathsf{H}}(\mathcal{M},\varphi),\quad
xD^{1/p'}\in L^{p',\mathsf{H}}(\mathcal{M},\varphi),\quad
\|xD^{1/p}\|_{p,\mathsf{H}}\leq \|xD^{1/p'}\|_{p',\mathsf{H}};$$

\emph{(5) (\cite[(1.3)]{junge02doob})}  For $1\leq p<\infty$ and for $0\leq x\leq y\in L^{p,\mathsf H}(\mathcal{M})$, we have $\|x\|_{p,\mathsf H}\leq \|y\|_{p,\mathsf H}$.
\end{prop}

In this paper we will frequently identify the two $L^p$-spaces via the following isomorphism.

\begin{prop}
[{\cite[Sect.9]{kosaki84interpolation}}]
 For all $1\leq p\leq \infty $, the map 
$$j^p: x\mapsto xD^{1/p}, x\in \mathcal{M}$$ 
extends to an isometry from $L^{p}(\mathcal{M},\varphi)$ onto $L^{p,\mathsf{H}}(\mathcal{M},\varphi)$.
 \end{prop}

We will need the noncommutative Khintchine inequality for Rademacher sequences. Let $2\leq p<\infty$. For a finitely supported sequence  $(x_n)_{1\leq n\leq N}\subset  L^{p,\mathsf{H}}(\mathcal{M},\varphi)$, we introduce the notation 
$$\|(x_n)\|_{L^p(\mathcal{M};\ell^2_c)}=\Big\|(\sum_n |x_n|^2)^{1/2}\Big\|_{p,\mathsf{H}},\quad
\|(x_n)\|_{L^p(\mathcal{M};\ell^2_r)}=\Big\|(\sum_n |x_n^*|^2)^{1/2}\Big\|_{p,\mathsf{H}},$$
and write
$$\|(x_n)\|_{CR_p[L^p( \mathcal{M})]}=\max \{\|(x_n)\|_{L^p(\mathcal{M};\ell^2_c)}, \|(x_n)\|_{L^p(\mathcal{M};\ell^2_r)}\}.$$
By Proposition \ref{density op haagerup Lp} and the triangle inequality, we see easily that
$$\|(x_n)\|_{CR_p[L^p( \mathcal{M})]}\leq \Big(\sum_n \|x_n\|_{L^{p,\mathsf H}(\mathcal{M},\varphi)}^2\Big)^{1/2}.$$
Denote by $(\varepsilon_{n})_{n\geq 1}$ a
Rademacher sequence on a probability space $(\Omega,P)$, i.e., an independent sequence of random variables with $ P (\varepsilon_n = 1) = P (\varepsilon_n =  -1) = 1/2 $ for all $ n $.
The following noncommutative Khintchine inequality for Haagerup's $L^p$-spaces is given in \cite[Theorem 3.4]{jungexu03burkholder}. 
\begin{thm}
\label{khintchine}
There exists an absolute constant $C>0$ such that for all $2\leq p<\infty$ and all finitely supported sequences $(x_{n})$  in $L^{p,\mathsf H}(\mathcal{M},\varphi)$, we have
\begin{equation*}
  \|(x_n)\|_{CR_p[L^p( \mathcal{M})]}\leq\Big(\int_{\Omega}\Big\|\sum_{n}\varepsilon_{n}(\omega)x_{n}\Big\|_{L^{p,\mathsf H}(\mathcal{M},\varphi)}^p dP(\omega)\Big)^{1/p}\leq C\sqrt{p}\|(x_n)\|_{CR_p[L^p( \mathcal{M})]}.
\end{equation*}
Consequently $L^{p,\mathsf H}(\mathcal{M},\varphi)$ is of type $2$, i.e., for the above sequences we have 
$$\Big(\int_{\Omega}\Big\|\sum_{n}\varepsilon_{n}(\omega)x_{n}\Big\|_{L^{p,\mathsf H}(\mathcal{M},\varphi)}^p dP(\omega)\Big)^{1/p}\leq
C\sqrt p \Big(\sum_n \|x_n\|_{L^{p,\mathsf H}(\mathcal{M},\varphi)}^2\Big)^{1/2}.$$
\end{thm}

In this paper, we will mainly be interested in the case of $\mathcal{M}=L^\infty (\mathbb G)$ for a compact quantum group $\mathbb G$. Let $h$ be the Haar state on $\mathbb G$. Throughout the paper, for any $1\leq p\leq \infty$ we will use the notation $L^p(\mathbb G)\coloneqq L^p(L^\infty (\mathbb G),h)$ for Kosaki's noncommutative $L^p$-spaces and $L^{p,\mathsf H}(\mathbb G)\coloneqq L^{p,\mathsf H}(L^\infty (\mathbb G),h)$ for Haagerup's noncommutative $L^p$-spaces introduced above. As is seen in the previous construction, we will in the sequel identify $L^\infty(\mathbb G)$ with a subspace of $L^1(\mathbb G)$ via the embedding 
$$j:x\mapsto xh\coloneqq h(\cdot x).$$ In particular we define the convolution for $x,x'\in L^\infty(\mathbb{G})$,
$$x\star x'\coloneqq (xh)\star (x'h) (\in L^1(\mathbb G)).$$ It is easy to see that the algebra of polynomials $\mathrm{Pol}(\mathbb{G})$ is a common dense subspace of all $L^p$-spaces associated to $\mathbb G$ with $1\leq p<\infty $. In fact we have the following property (see \cite[Lemma 2.2]{junge02doob}).

\begin{lem}
\label{prop: poly dense in L_1}Let $A$ be an ultraweakly dense $*$-subalgebra of $\mathcal{M}$. Then for $1\leq p<\infty $,  $A$ is dense
in $L^{p}(\mathcal{M},\varphi)$ with respect to $\|\ \|_p$.\end{lem}

We will not consider the noncommutative $L^p$-spaces associated to discrete quantum groups for a general $p$ in this text, but let us add several words on the special cases $p=1$ and $2$. We consider the dual discrete quantum group $\hat{\mathbb G}$ of a compact quantum group $\mathbb G$. Define the $L^{1}$-space $\ell^{1}(\hat{\mathbb{G}})$ on $\hat{\mathbb{G}}$
associated to $\hat{h}$ as 
\[
\ell^{1}(\hat{\mathbb{G}})=\{x\in c_{0}(\hat{\mathbb{G}}):\|x\|_{1}=\sum_{\pi\in\irr}d_{\pi}\mathrm{Tr}(|p_{\pi}xQ_{\pi}|)<\infty\}.
\]
By the property of $\mathrm{Tr}$ on $B(\oplus_\pi H_\pi)$, it is easy to see that $\ell^{1}(\hat{\mathbb{G}})$ is a Banach space
and the injection 
\[
j':\ell^{1}(\hat{\mathbb{G}})\to\ell^{\infty}(\hat{\mathbb{G}})_{*},\quad j'(x)=x\hat{h}\coloneqq\hat{h}(\cdot\, x),\quad x\in \ell^{1}(\hat{\mathbb{G}})
\]
is an isometric isomorphism. 
%In fact if we let $\tau=\sum_{\pi\in\irr}d_{\pi}\mathrm{Tr}(p_{\pi}\cdot)$
%and consider   $L =\{x\in c_{0}(\hat{\mathbb{G}}):\tau(| x|)<\infty\},$
%then the bijection 
%\[
%j'':\ell^{1}(\hat{\mathbb{G}})\to L ,\quad(x_{\pi})_{\pi\in\irr}\mapsto(x_{\pi}Q_{\pi})_{\pi\in\irr}
%\]
%gives rise to an isometric isomorphism. It is also known that $x\mapsto\tau(\cdot x)$
%gives identification between $L $ and $\ell^{\infty}(\hat{\mathbb{G}})_{*}$,
%so the isometric isomorphism $j'$ is verified. 
One can easily see
that $c_{c}(\hat{\mathbb{G}})$ is a dense subset in $\ell^{1}(\hat{\mathbb{G}})$. Also, we define the space $\ell^2(\hat{\mathbb{G}})$ to be the Hilbert space in the GNS construction induced by the Haar weight $\hat h$ on $\ell^\infty(\hat{\mathbb G})$.

\section{Fourier series and multipliers}

%Let $\mathbb{G}$ be a compact quantum group. 
%We define the $L^{1}$-space $\ell^{1}(\hat{\mathbb{G}})$ on $\hat{\mathbb{G}}$
%associated to $\hat{h}$ as 
%\[
%\ell^{1}(\hat{\mathbb{G}})=\{x\in c_{0}(\hat{\mathbb{G}}):\|x\|_{1}=\sum_{\pi\in\irr}d_{\pi}\mathrm{Tr}(|p_{\pi}xQ_{\pi}|)<\infty\}.
%\]
%It is easy to see that $\ell^{1}(\hat{\mathbb{G}})$ is a Banach space
%and the injection 
%\[
%j':\ell^{1}(\hat{\mathbb{G}})\to\ell^{\infty}(\hat{\mathbb{G}})_{*},\quad j'(x)=x\hat{h}\coloneqq\hat{h}(\cdot\, x),\quad x\in l^{1}(\hat{\mathbb{G}})
%\]
%is an isometric isomorphism. In fact if we let $\tau=\sum_{\pi\in\irr}d_{\pi}\mathrm{Tr}(p_{\pi}\cdot)$
%and consider the $L^{1}$-space $L^{1}(\hat{\mathbb{G}},\tau)$ associated
%to the trace $\tau$, then similar to the Schatten classes one can
%check $L^{1}(\hat{\mathbb{G}},\tau)=\{x\in c_{0}(\hat{\mathbb{G}}):\sum_{\pi\in\irr}d_{\pi}\mathrm{Tr}(|p_{\pi}x|)<\infty\}$
%and hence the bijection 
%\[
%j'':\ell^{1}(\hat{\mathbb{G}})\to L^{1}(\hat{\mathbb{G}},\tau),\quad(x_{\pi})_{\pi\in\irr}\mapsto(x_{\pi}Q_{\pi})_{\pi\in\irr}
%\]
%gives rise to an isometric isomorphism. It is also known that $x\mapsto\tau(\cdot x)$
%gives identification between $L^{1}(\hat{\mathbb{G}},\tau)$ and $\ell^{\infty}(\hat{\mathbb{G}})_{*}$,
%so the isometric isomorphism $j'$ is verified. One can easily see
%that $c_{c}(\hat{\mathbb{G}})$ is a dense subset in $\ell^{1}(\hat{\mathbb{G}})$.

The Fourier transform for locally compact quantum groups has been discussed in \cite{cooney2010hy},  \cite{caspers2013fourier} and \cite{kahng2010fourier}.
In the setting of compact quantum groups, we may give a more explicit
description. Let a compact quantum group $\mathbb{G}$ be fixed. For a linear functional $\f$ on $\mathrm{Pol}(\mathbb{G})$, we define
the \emph{Fourier transform} $\hat{\varphi}=(\hat{\varphi}(\pi))_{\pi\in\irr}\in\oplus_{\pi}B(H_{\pi})$
by 
\[
\hat{\varphi}(\pi)=(\varphi\otimes\iota)((u^{(\pi)})^{*})\in B(H_{\pi}),\quad\pi\in\irr.
\]
In particular, any $x\in L^{\infty}(\mathbb G)$ (or $L^{2}(\mathbb{G})$)
induces a functional $xh\coloneqq h(\cdot\,x)$ on $\mathrm{Pol}(\mathbb{G})$ defined by $y\mapsto h(yx)$,
and the Fourier transform
$\hat{x}=(\hat{x}(\pi))_{\pi\in\irr}$ of $x$ is given by 
\[
\hat{x}(\pi)=(h(\cdot x)\otimes\iota)((u^{(\pi)})^{*})\in B(H_{\pi}),\quad\pi\in\irr.
\]
The above definition is slightly different from that of \cite{caspers2013fourier}
or \cite{kahng2010fourier}. Indeed, we replace the unitary $u^{(\pi)}$
by $(u^{(\pi)})^{*}$ in the above formulas. This is just to be compatible
with standard definitions in classical analysis on compact groups
such as in \cite[Section 5.3]{folland1995harmonic}, and this will not
cause any essential difference. We refer to Section 2.2.2 of \cite{wang14lpimproving} for some explanation on classical examples in this setting. On the other hand,
the notation $\hat{\varphi}$ has a slight conflict with the dual
Haar weight $\hat{h}$ on $\hat{\mathbb{G}}$. One can however
distinguish them by the elements on which it acts, so we hope that
this will not cause any ambiguity for the reader. 

Let $\varphi_{1},\varphi_{2}$ be linear functionals on $\mathrm{Pol}(\mathbb{G})$. Consider their convolution product $ 
\varphi_{1}\star\varphi_{2}=(\varphi_{1}\otimes\varphi_{2})\circ\Delta$. 
We note that for a linear functional $\varphi$ on $\mathrm{Pol}(\mathbb{G})$ and $\pi\in\irr$, 
\begin{align}
(\varphi\circ S^{-1})\,\hat{}\,(\pi)&=((\varphi\circ S^{-1})\otimes\iota)((u^{(\pi)})^{*})=\left[\varphi(u_{ij}^{(\pi)})\right]_{i,j}.\label{eq:fourier series with S}
\end{align}
In particular,
\begin{equation}\label{eq:fourier series for phi}
(\varphi^{*}\circ S^{-1})\,\hat{}\,(\pi)=\hat{\varphi}(\pi)^{*},
\end{equation}
where $\varphi^{*}$ denotes
the usual adjoint of $\varphi$, i.e., $\varphi^*(x)=\overline{\varphi(x^*)}$.
On the other hand, a straightforward calculation shows that 
\begin{equation}\label{conv fourier series}
(\varphi_{1}\star\varphi_{2})\,\hat{}\,(\pi)=\hat{\varphi}_{2}(\pi)\hat{\varphi}_{1}(\pi).
\end{equation}
Using the property of the antipode $S$ we may also show (cf. \cite[Proposition 2.2]{vandaele2007fourier}, \cite[(2.14)]{wang14lpimproving}) that for a linear functional $\varphi$ on $L^\infty(\mathbb{G})$ and $x\in L^\infty (\mathbb G)$, $\pi\in\irr$,
\begin{equation}\label{conv another}
((\iota\otimes\varphi)\Delta( x))\,\hat{}\,(\pi)=(\varphi\circ S^{-1})\,\hat{}\,(\pi)\hat{x}(\pi),\quad
((\varphi\otimes\iota)\Delta( x))\,\hat{}\,(\pi)=\hat{x}(\pi)(\varphi\circ S)\,\hat{}\,(\pi).
\end{equation}
Using \eqref{eq: scaling group on cqg} and \eqref{antipode2} we may rewrite the  second  equality above as 
\begin{equation}\label{conv another with q}
((\varphi\otimes\iota)\Delta( x))\,\hat{}\,(\pi)Q_\pi=\hat{x}(\pi)Q_\pi(\varphi\circ S^{-1})\,\hat{}\,(\pi).
\end{equation}

In the sequel let $\mathcal{F}:f\mapsto\hat{f}$ denote the Fourier transform.
\begin{prop}\label{prop: l1 into c0}
$\mathcal{F}$ is a contraction from $L^{\infty}(\mathbb{G})^{*}$
to $\ell^{\infty}(\hat{\mathbb{G}})$, and moreover $\mathcal{F}$ sends $L^{1}(\mathbb{G})$ injectively into $ c_{0}(\hat{\mathbb{G}})$.\end{prop}
\begin{proof}
For $\varphi\in L^{\infty}(\mathbb{G})^{*}$, recall that $\|\varphi\otimes\iota\|=\|\varphi\|$%
\begin{comment}
see Paulsen's book ex1 of the chap on tensor 
\end{comment}
, so we have 
\[
\|\hat{\varphi}\|=\sup_{\pi\in\irr}\|\hat{\varphi}(\pi)\|=\sup_{\pi\in\irr}\|(\varphi\otimes\iota)((u^{(\pi)})^{*})\|\leq\sup_{\pi\in\irr}\|\varphi\otimes\iota\|\|u^{(\pi)}\|=\|\varphi\|.
\]
So $\mathcal{F}$ is a contraction. Recall that $\mathrm{Pol}(\mathbb{G})$
is dense in $L^{1}(\mathbb{G})=L^\infty(\mathbb G)_{*}$ and note that $\mathcal{F}(\mathrm{Pol}(\mathbb{G}))\subset c_{c}(\hat{\mathbb{G}})$,
so $\mathcal{F}(L^{1}(\mathbb{G}))\subset\overline{\mathcal{F}(\mathrm{Pol}(\mathbb{G}))}\subset\overline{c_{c}(\hat{\mathbb{G}})}=c_{0}(\hat{\mathbb{G}})$. The injectivity of $\mathcal{F}$ follows from the ultraweak density of $\mathrm{Pol}(\mathbb G)$ in $L^\infty(\mathbb G)$.\end{proof}

It is easy to establish the Fourier inversion formula and the Plancherel theorem
for $L^2(\mathbb{G})$. The following result  can be found in Proposition 2.6 in \cite{wang14lpimproving}.
\begin{prop}
\emph{\label{prop:Plancherel}(a) }For all $x\in L^{2}(\mathbb{G})$,
we have 
\begin{equation}
x=\sum_{\pi\in\irr}d_{\pi}(\iota\otimes\mathrm{Tr})[(1\otimes\hat{x}(\pi)Q_{\pi})u^{(\pi)}],\label{fourier series}
\end{equation}
where the convergence of the series is in the $L^{2}$-sense. For any $\pi\in\irr$, if we denote by $\mathcal E _\pi $ the orthogonal projection of $L^2(\mathbb G )$ onto the subspace spanned by the matrix coefficients $(u_{ij}^{(\pi)})_{i,j=1}^{n_{\pi}}$, and write $\mathcal{E}_{\pi}x=\sum_{i,j}x_{ij}^{(\pi)}u_{ij}^{(\pi)}$ with $x_{ij}^{(\pi)}\in\mathbb C$, $X_{\pi}=[x_{ji}^{(\pi)}]_{i,j}$, then
$$\hat x (\pi)=d_\pi^{-1}X_\pi Q_\pi^{-1}.$$

\emph{(b) }$\mathcal{F}$ is a unitary operator from $L^{2}(\mathbb{G})$
onto $\ell^{2}(\hat{\mathbb{G}})$.
\end{prop}

For $a=(a_{\pi})_{\pi}\in\prod_{\pi}B(H_{\pi})$, we
define the left and right \emph{multipliers} $m_{a}^L:\mathrm{Pol}(\mathbb{G})\to\mathrm{Pol}(\mathbb{G}), m_{a}^R:\mathrm{Pol}(\mathbb{G})\to\mathrm{Pol}(\mathbb{G})$
associated to $a$ (cf. \cite{jungeruan2009repquantumgroup,daws2012cpmultiplier}) by 
\begin{equation}
(m_{a}^L\otimes \iota)u^{(\pi)}=(1\otimes a_\pi)u^{(\pi)},\quad
(m_{a}^R\otimes \iota)u^{(\pi)}=u^{(\pi)}(1\otimes a_\pi),\label{left multiplier formula}
\end{equation}
which, according to the above proposition, yields that 
\begin{equation}
(m_{a}^Lx)\,\hat{}\,(\pi)Q_\pi=\hat{x}(\pi)Q_\pi a_{\pi},\quad (m_{a}^Rx)\,\hat{}\,(\pi)Q_\pi=a_{\pi}\hat{x}(\pi)Q_\pi  \quad\pi\in\irr.\label{eq:multiplier, Fourier coefficient}
\end{equation}

Denote $Q=(Q_\pi)_{\pi} \in \prod_{\pi}B(H_{\pi})$ and let $1\leq p\leq \infty$. We say that $a$ is a \emph{bounded left (resp., right) multiplier} on $L^p(\mathbb G)$ if $m_{a}^L$ (resp., $m_a^R$) extends
to a  bounded map on $L^p(\mathbb{G})$, and denote the set of all such multipliers by $\mathrm{M}_L(L^p(\mathbb G))$ (resp., $\mathrm{M}_R(L^p(\mathbb G))$). We define  
$$\mathrm{M}(L^p(\mathbb G))=\Big\{a\in \prod_{\pi}B(H_{\pi}):Q^{-1/p}aQ^{1/p}\in \mathrm{M}_L(L^p(\mathbb G)),\   a\in \mathrm{M}_R(L^p(\mathbb G)) \Big\}$$
equipped with the norm
$$\|a\|_{\mathrm{M}(L^p(\mathbb G))} 
=\max\{\|m_{Q^{-1/p}aQ^{1/p}}^L\|_{B(L^p(\mathbb G))},\,\|m_{a}^R\|_{B(L^p(\mathbb G))} \} .$$
 
\begin{lem}
	\label{lem:multiplier on ltwo}
	For any $a\in \prod_{\pi}B(H_{\pi})$, we have
	$$\|a\|_\infty = \|m_{Q^{-1/2}aQ^{1/2}}^L\|_{B(L^2(\mathbb G))} 
	= \|m_{a}^R\|_{B(L^2(\mathbb G))}.$$
	In particular, we have the following isometric isomorphism
	$$\mathrm{M}(L^2(\mathbb G))=\ell^\infty (\hat{\mathbb G}).$$
\end{lem}
\begin{proof}
Let $a\in \ell^\infty (\hat{\mathbb G})$. By   equality \eqref{eq:multiplier, Fourier coefficient} and Proposition \ref{prop:Plancherel}, we have for all $x\in \mathrm{Pol}(\mathbb G)$,
\begin{align*}
\|m_{Q^{-1/2}aQ^{1/2}}^L x\|_2^2 
& = \sum_{\pi \in \mathrm{Irr} (\mathbb G)} d_\pi \mathrm{Tr} \left(|\hat x (\pi) Q_\pi^{1/2}a_\pi Q_\pi ^{-1/2} |^2 Q_\pi \right) \\
& = \sum_{\pi \in \mathrm{Irr} (\mathbb G)} d_\pi \mathrm{Tr} \left(\hat x (\pi) Q_\pi^{1/2} a_\pi a_\pi^* Q_\pi^{1/2} \hat x (\pi)^*  \right)\\
& \leq \|a\|_\infty ^2 \sum_{\pi \in \mathrm{Irr} (\mathbb G)} d_\pi \mathrm{Tr} \left(\hat x (\pi) Q_\pi \hat x (\pi)^*  \right)
= \|a\|_\infty ^2 \|x\|_2^2.
\end{align*}
Therefore we get
\begin{equation}\label{eq: multiplier two 1}
Q^{-1/2}aQ^{1/2} \in \mathrm{M}_L(L^2(\mathbb G)),\quad  \|m_{Q^{-1/2}aQ^{1/2}}^L\|_{B(L^2(\mathbb G))} \leq \|a\|_\infty.
\end{equation}

Conversely, take $a\in \prod_{\pi}B(H_{\pi})$ with $Q^{-1/2}aQ^{1/2} \in \mathrm{M}_L(L^2(\mathbb G))$. For $\pi\in \irr$ and for each eigenvalue $\lambda$ of the operator $a_\pi a_\pi^*$, we choose an orthogonal projection $E\in B(H_\pi)$ such that $a_\pi a_\pi^*E=\lambda E$. According to Proposition \ref{prop:Plancherel}, we choose an $x\in \pol$ with 
$$\hat x (\pi) = E Q_\pi^{-1/2},\quad \hat x (\pi')=0, \quad \pi' \neq \pi.$$
Then we have
\begin{align*}
\|m_{Q^{-1/2}aQ^{1/2}}^L x\|_2^2
& = d_\pi \mathrm{Tr} \left(|\hat x (\pi) Q_\pi^{1/2}a_\pi Q_\pi ^{-1/2} |^2 Q_\pi \right) \\
& =d_\pi \mathrm{Tr} \left(a_\pi^* Q_\pi^{1/2} \hat x (\pi)^*\hat x (\pi) Q_\pi^{1/2} a_\pi   \right)
= d_\pi \mathrm{Tr} \left(a_\pi^* E a_\pi   \right)= \lambda d_\pi \mathrm{Tr} \left(E   \right)\\
&
 = \lambda d_\pi \mathrm{Tr} \left(E^*E   \right) = \lambda d_\pi \mathrm{Tr} \left( Q_\pi^{1/2} \hat x (\pi)^*\hat x (\pi) Q_\pi^{1/2}   \right)
 = \lambda \|x\|_2^2.
\end{align*}
So we see that $\lambda\leq \|m_{Q^{-1/2}aQ^{1/2}}^L\|_{B(L^2(\mathbb G))}^2$. Taking the supremum over all such eigenvalues $\lambda$, we get
\begin{equation}\label{eq: multiplier two 2}
a\in \ell^\infty(\hat{\mathbb G}),\quad \|a\|_\infty \leq \|m_{Q^{-1/2}aQ^{1/2}}^L\|_{B(L^2(\mathbb G))}.
\end{equation}
Combining the two inequalities \eqref{eq: multiplier two 1} and \eqref{eq: multiplier two 2} we prove that for all $a\in \prod_{\pi}B(H_{\pi})$ we have 
 $$\|a\|_\infty = \|m_{Q^{-1/2}aQ^{1/2}}^L\|_{B(L^2(\mathbb G))}.$$
 A similar argument gives 
 $$\|a\|_\infty = \|m_{a}^R\|_{B(L^2(\mathbb G))}.$$
So we establish the lemma. 
\end{proof}

It is well known that if $\mathbb G$ is a compact group $G$ or the dual compact quantum group $\hat \Gamma$ of a discrete group $\Gamma$, then for all $1\leq p\leq \infty$ we have
\begin{equation}\label{eq:multiplier bdd classical}
\|a\|_\infty \leq \|a\|_{\mathrm M (L^p(\mathbb G))},\quad a\in \mathrm M (L^p(\mathbb G)).
\end{equation}
We refer to \cite{harcharras99nclambdap,hewittross1970abstract} for related discussions. This is however not clear for an arbitrary compact quantum group.
Note that if the operators $m_a^L$ and $m_a^R$ are completely bounded on $L^\infty(\mathbb G)$, then from \eqref{left multiplier formula} we see that 
\begin{equation*}
\|a\|_\infty =\| ((1\otimes a_\pi)u^{(\pi)})\|_{\oplus (C(\mathbb G)\otimes B(H_\pi))}\leq \|m_a^L\|_{CB(L^\infty(\mathbb G))},\quad 
\|a\|_\infty \leq \|m_a^R\|_{CB(L^\infty(\mathbb G))},
\end{equation*}
where $\|\|_{CB(L^\infty(\mathbb G))}$ denote the completely bounded norm in the sense of \cite{effrosruan00op}. If the multiplier $a\in \mathrm{M}(L^\infty(\mathbb G))$ is not necessarily completely bounded, it is even not clear to see whether the sequence $a\in \prod_{\pi }B(H_\pi) $ is bounded or not.
We refer to \cite[Section 8.2]{daws10multiplier} for some related discussions. In the following Proposition \ref{prop:multiplier bdd quantum} we will give an affirmative answer in the case that $\mathbb G$ is of Kac type, and indeed give a similar estimate for the general case, which also improves an early result in \cite[Proposition 8.8]{daws10multiplier} for compact quantum groups. Let us first establish some easy but useful lemmas.
\begin{lem}
	\label{lem: R isometry Lp}
Let $1\leq p\leq \infty$. We have the following assertions: 

\emph{(a)} For all $t\in \mathbb R$, the operator $\sigma_t$ extends to an isometry on $L^p(\mathbb G)$.

\emph{(b)} We have $\sigma_{-\mathrm i /p}\circ R= R \circ \sigma_{\mathrm i/p}$, and they extend to an isometry on $L^p(\mathbb G)$.

\emph{(c)} The map $x\mapsto \sigma_{-\mathrm i /p}(x^*)$ extends to an isometry on $L^p(\mathbb G)$.
\end{lem}
\begin{proof}
	(a) Since $h\circ \sigma =h$, it is easy to see that the operators $\sigma_t$ and $\sigma_{-t}$ extend to isometries on $L^1(\mathbb G)$ and $L^\infty(\mathbb G)$ for all $t\in \mathbb R$. So by interpolation $\sigma_t$ also extends to an isometry on $L^p(\mathbb G)$ for all $t\in\mathbb R$.
	
(b) It is well known and easy to see from \eqref{modular group on cqg}, \eqref{eq: polar decomp of s} and \eqref{eq: scaling group on cqg} that
$$S\circ \sigma_z= \sigma_{-z}\circ S,\quad R\circ \sigma_z= \sigma_{-z}\circ R, \quad z\in \mathbb C.$$
Also recall that $R$ is a $*$-antiautomorphism on $L^\infty(\mathbb G)$ with $h\circ R=h$, and $h(xy)=h(y\sigma_{-\mathrm i}(x))$ for $x,y\in \mathrm{Pol}(\mathbb G)$. So for $x\in \mathrm{Pol}(\mathbb G)$, we have
\begin{align*}
	\|x\|_1 & =\sup_{y\in \mathrm{Pol}(\mathbb G),\|y\|_\infty=1}|h(yx)|
	= \sup_{y\in \mathrm{Pol}(\mathbb G),\|y\|_\infty=1}|h(R(x)R(y))|\\
	&=	\sup_{y\in \mathrm{Pol}(\mathbb G),\|y\|_\infty=1}|h(R(y)\sigma_{-\mathrm i}(R(x)))| =\|\sigma_{-\mathrm i}(R(x))\|_1.
\end{align*}
Thus we see that $\sigma_{-\mathrm i}\circ R$ extends to an isometry on $L^1(\mathbb G)$. Therefore by (a), $\sigma_{-\mathrm i-t}\circ R$ and $R \circ \sigma_{\mathrm i+t}$ also extend to isometries on $L^1(\mathbb G)$ for $t\in\mathbb R$. Then the assertion follows directly from the Stein interpolation theorem (see e.g. \cite[Theorem 2.7]{lunardi09interpolationbook}).

(c) This follows directly from Proposition \ref{density op haagerup Lp} (1)(4).
\end{proof}

\begin{lem}
	\label{lem: left to right multipliers}
Let $1\leq p\leq \infty$. Consider $\mathbf{E}\subset\irr$ and $$X=\mathrm{span}\left\lbrace u_{ij}^{(\pi)},(u_{ij}^{(\pi)})^*: \pi\in \mathbf{E},1\leq i,j\leq n_\pi \right\rbrace.
$$
We equip the space $X$ with the norm $\|\|_p$. Then we have for all  $a\in \prod_{\pi}B(H_{\pi})$,
	$$\|m_a^R|_X\|_{B(X)}=\|m_{Q^{-1/2}a^*Q^{1/2}}^L|_X\|_{B(X)}.$$
\end{lem}

\begin{proof}
Note that the subspace $X$ is invariant under $R$, $\sigma$ and $\tau$, and that the maps $m_a^R$ and $m_{Q^{-1/2}a^*Q^{1/2}}^L$ send $X$ into $X$ itself. We define a linear functional $\varphi$ on $\pol$ by $\varphi(u_{ij}^{(\pi)})=\overline{(a_\pi)_{ji}}$ for $\pi\in\irr$ and $1\leq i,j\leq n_\pi$. Then we see that $\widehat{\varphi^*}=a$ where $\varphi^*=\overline{\varphi(\cdot ^*)}$. Take $x\in X$. By \eqref{conv another}, \eqref{conv another with q}, we have
\begin{align*}
m_a^R(\sigma_{-\mathrm i/p}(x^*))= (\iota\otimes (\varphi^*\circ S))\Delta(\sigma_{-\mathrm i /p}(x^*)).
\end{align*}
Recall from \eqref{eq: polar decomp of s} and \eqref{antipode2} that
$$S=\tau_{-\mathrm i/2}\circ R,\quad R^2=\mathrm{id},\quad \Sigma \circ \Delta \circ R= (R\otimes R)\circ \Delta.$$
So we get 
\begin{align*}
(\iota\otimes (\varphi^*\circ S))\Delta(\sigma_{-\mathrm i /p}(x^*))
&=(R\otimes (\varphi^*\circ \tau_{-\mathrm i/2}))(R\otimes R)\Delta(\sigma_{-\mathrm i /p}(x^*))\\
&= ( (\varphi^*\circ \tau_{-\mathrm i/2})\otimes R)\Delta (R\circ \sigma_{-\mathrm i /p}(x^*)) .
\end{align*}
Writing $R=(R\circ \sigma_{\mathrm i /p} ) \circ \sigma_{-\mathrm i /p}$ and using the previous lemma, we then have
\begin{align*}
\|m_a^R(\sigma_{-\mathrm i/p}(x^*))\|_p
&=\big\|\sigma_{-\mathrm i /p}( (\varphi^*\circ \tau_{-\mathrm i/2})\otimes \iota)\Delta (R\circ \sigma_{-\mathrm i /p}(x^*))\big\|_p\\
&=\left\|\sigma_{-\mathrm i /p}\left(\left[( (\varphi\circ \tau_{\mathrm i/2})\otimes \iota)\Delta (R\circ \sigma_{\mathrm i /p}(x))\right]^* \right) \right\|_p\\
&=\big\|( (\varphi\circ \tau_{\mathrm i/2})\otimes \iota)\Delta (R\circ \sigma_{\mathrm i /p}(x))\big\|_p.
\end{align*}
Note that by \eqref{eq:fourier series with S} and \eqref{eq: scaling group on cqg}, we have
$$(\varphi\circ \tau_{\mathrm i/2}\circ S^{-1})\,\hat{}\,=Q^{-1/2}a^*Q^{1/2}.$$
So together with \eqref{conv another with q} and \eqref{eq:multiplier, Fourier coefficient}, we have 
$$( (\varphi\circ \tau_{\mathrm i/2})\otimes \iota)\Delta (R\circ \sigma_{\mathrm i /p}(x))=m_{Q^{-1/2}a^*Q^{1/2}}^L(R\circ \sigma_{-\mathrm i /p}(x))$$
and hence the above computations yield
$$\|m_a^R(\sigma_{-\mathrm i/p}(x^*))\|_p=
\|m_{Q^{-1/2}a^*Q^{1/2}}^L(R\circ \sigma_{\mathrm i /p}(x))\|_p.$$
Since the maps $x\mapsto \sigma_{-\mathrm i/p}(x^*)$ and $x\mapsto R\circ \sigma_{\mathrm i /p}(x)$ are isometries on $X$ according to the previous lemma, we obtain the desired conclusion.
\end{proof}

\begin{lem}
	\label{lem:fourier series with sigma}
	For $z\in \mathbb C$ and $x\in \pol$, we have
	$$\mathcal F (\sigma_z(x)) = Q^{\mathrm i  z} \hat x Q^{\mathrm i z}.$$
	In particular for $a\in\prod_{\pi\in\irr} B(H_\pi)$ and $t\in \mathbb R$,
	$$\sigma_t \circ m^L_a \circ \sigma_{-t} = m^L_{Q^{-\mathrm i t}a Q^{\mathrm i t}},\quad \sigma_t \circ m^R_a \circ \sigma_{-t} = m^R_{Q^{\mathrm i t}a Q^{-\mathrm i t}}.$$
\end{lem}
\begin{proof}
	Recall that
	$$\sigma_z(x)=(f_{\mathrm i z}\otimes\iota \otimes f_{\mathrm i z})\Delta^{(2)}(x),\quad  x\in \mathrm{Pol}(\mathbb G),z\in\mathbb C.$$
	And by the construction of $(f_z)$ and \eqref{eq:fourier series with S} we have $(f_{\mathrm i z} \circ S^{-1})\,\hat{}\,=Q^{\mathrm i z}$. So together with \eqref{conv another} and \eqref{conv another with q} we get 
	$$\mathcal F (\sigma_z(x)) = Q^{\mathrm i  z} \hat x Q^{\mathrm i z}.$$
	Consequently, for $x\in\pol$,
	$$\mathcal{F}(\sigma_t \circ m^L_a \circ \sigma_{-t}(x))Q= Q^{\mathrm i t} \mathcal{F}( m^L_a \circ \sigma_{-t}(x)) Q^{1+\mathrm i t}
	=Q^{\mathrm i t}  \mathcal{F}(    \sigma_{-t}(x)) Qa Q^{\mathrm i t} =\hat x Q Q^{\mathrm i t} a Q^{-\mathrm i t}  .$$
	That is, 
		$$\sigma_t \circ m^L_a \circ \sigma_{-t} = m^L_{Q^{-\mathrm i t}a Q^{\mathrm i t}},$$
	as desired. The equality for $m^R_a$ follows similarly.
\end{proof}
Now we are able to give an analogue of inequality \eqref{eq:multiplier bdd classical} for general compact quantum groups. The proof below is based on a personal communication by Marius Junge.
\begin{prop}
	\label{prop:multiplier bdd quantum}
Let $1\leq p\leq \infty$. For all $a\in\prod_{\pi}B(H_{\pi})$ and all $0\leq \theta\leq 1$, we have
$$\|Q^{\frac{1}{4}-\frac{\theta}{2}}aQ^{ -\frac{1}{4}+\frac{\theta}{2}}\|_\infty \leq \|m_a^R\|^{1/2}_{B(L^p(\mathbb G))} \|m_{Q^{-\theta}aQ^{\theta}}^L\|^{1/2}_{B(L^p(\mathbb G))} .$$
\end{prop}
\begin{proof}
Assume that $m_a^R$ is bounded on $L^p(\mathbb G)$. Consider the map
$$T:\pol\to\pol,\quad y\mapsto \sigma_{-\mathrm i /p}\big([m_a^R(\sigma_{-\mathrm i /p}(y^*))]^*\big).$$
By Lemma \ref{lem: R isometry Lp}, $T $ extends to a bounded operator on $L^p(\mathbb G)$ with $\|T\|_{B(L^p(\mathbb G))}=\|m_a^R\|_{B(L^p(\mathbb G))}$. We let $1\leq p'\leq \infty$ with $1/p+1/p'=1$ and denote $\langle \cdot,\cdot \rangle_{p',p}$ the duality bracket between $L^{p'}(\mathbb G)$ and $L^p(\mathbb G)$ defined by 
$$\langle x,y \rangle_{p',p} = \mathrm{tr}(xD^{1/p'}yD^{1/p}),\quad x,y\in \mathrm{Pol}(\mathbb G).$$
First let $1\leq p<\infty$. Then consider the adjoint map $T^*:L^p(\mathbb G)^*\to L^p(\mathbb G)^*$ of $T$ and recall the formula $\sigma_t(x)=D^{\mathrm i t}xD^{-\mathrm i t}$ in Proposition \ref{density op haagerup Lp}. We have for $x,y\in \pol$,
\begin{align*}
 \langle T^*x,y \rangle_{p',p}
 &= \langle x,Ty \rangle_{p',p} = \mathrm{tr}(xD^{1/p'}TyD^{1/p})
 =\mathrm{tr}\big(xD[m_a^R(\sigma_{-\mathrm i /p}(y^*))]^*\big)\\
 &= h\big([m_a^R(\sigma_{-\mathrm i /p}(y^*))]^*x\big).
\end{align*}
According to Proposition \ref{prop:Plancherel} and   \eqref{eq:multiplier, Fourier coefficient}, we note that
\begin{align*}
h\big([m_a^R(\sigma_{-\mathrm i /p}(y^*))]^*x\big) 
&= \hat h \Big( \big([m_a^R(\sigma_{-\mathrm i /p}(y^*))]\,\hat{ } \,\big) ^*\hat x\Big)
=\hat h \Big( \big([a(\sigma_{-\mathrm i /p}(y^*))\,\hat{ } \,]\big) ^*\hat x\Big)\\
&=\hat h \Big( [(\sigma_{-\mathrm i /p}(y^*))\,\hat{ } \,]^*a^*\hat x\Big) 
= h\big( (\sigma_{-\mathrm i /p}(y^*))^* m_{a^*}^Rx \big).
\end{align*}
Observe that $(\sigma_{-\mathrm i /p}(y^*))^*=D^{-1/p}yD^{1/p}$. Therefore the above equalities give
$$\langle T^*x,y \rangle_{p',p}=h\big( (\sigma_{-\mathrm i /p}(y^*))^* m_{a^*}^Rx \big)=\mathrm{tr}(m_{a^*}^Rx D^{1/p'}yD^{1/p})=\langle m_{a^*}^Rx,y \rangle_{p',p}.$$
So $m_{a^*}^R$ extends to an isometry on $L^{p'}(\mathbb G)$ with 
\begin{equation}\label{eq:a bdd pf bis}
\|m_{a^*}^R\|_{B(L^{p'}(\mathbb G))}=\|T\|_{B(L^p(\mathbb G))}=\|m_a^R\|_{B(L^p(\mathbb G))}.
\end{equation}
If $p=\infty$, we consider the restriction of operator $T^*|_{L^1(\mathbb G)}:L^1(\mathbb G)\to L^1(\mathbb G)$ instead of $T^*:L^\infty(\mathbb G)^*\to L^\infty(\mathbb G)^*$, and repeat the above argument. Then \eqref{eq:a bdd pf bis} also holds for $p=\infty$.

On the other hand, we assume that $m_{Q^{-\theta}aQ^{\theta}}^L$ is bounded on $L^p(\mathbb G)$. By Lemma \ref{lem: left to right multipliers}, we have 
\begin{equation}\label{eq:a bdd pf}
\|m_{Q^{-\frac{1}{2}+\theta}a^*Q^{ \frac{1}{2}-\theta}}^R\|_{B(L^p(\mathbb G))}=\|m_{Q^{-\theta}aQ^{\theta}}^L\|_{B(L^p(\mathbb G))}.
\end{equation}
Note also that by Lemma \ref{lem: R isometry Lp} and Lemma \ref{lem:fourier series with sigma}, for all $t\in\mathbb R$,
\begin{align*}
\|m_{Q^{(-\frac{1}{2}+\theta)(1+\mathrm i t)}a^*Q^{ (\frac{1}{2}-\theta)(1+\mathrm i t)}}^R\|_{B(L^p(\mathbb G))}
	&=\|\sigma_{(-\frac{1}{2}+\theta)t}\circ m_{Q^{-\frac{1}{2}+\theta}a^*Q^{ \frac{1}{2}-\theta}}^R \circ \sigma_{(\frac{1}{2}-\theta)t}\|_{B(L^p(\mathbb G))}\\
&	=\|m_{Q^{-\frac{1}{2}+\theta}a^*Q^{ \frac{1}{2}-\theta}}^R\|_{B(L^p(\mathbb G))}.
\end{align*}
Hence, applying the Stein interpolation theorem,   \eqref{eq:a bdd pf bis} and \eqref{eq:a bdd pf}  yields
\begin{align*}
\|m_{Q^{-\frac{1}{4}+\frac{\theta}{2}}a^*Q^{ \frac{1}{4}-\frac{\theta}{2 }}}^R\|_{B(L^2(\mathbb G))}
& \leq \|m_{a^*}^R\|_{B(L^{p'}(\mathbb G))}^{\frac{1}{2}}
\|m_{Q^{-\frac{1}{2}+\theta}a^*Q^{ \frac{1}{2}-\theta}}^R\|_{B(L^p(\mathbb G))}^{\frac{1}{2}} \\
&=\|m_a^R\|_{B(L^p(\mathbb G))}^{\frac{1}{2}}
\|m_{Q^{-\theta}aQ^{\theta}}^L\|_{B(L^p(\mathbb G))}^{\frac{1}{2}}.
\end{align*}
Now applying Lemma \ref{lem:multiplier on ltwo}, we get
$$\|Q^{\frac{1}{4}-\frac{\theta}{2}}aQ^{ -\frac{1}{4}+\frac{\theta}{2 }}\|_\infty \leq 
\|m_a^R\|_{B(L^p(\mathbb G))}^{\frac{1}{2}}
\|m_{Q^{-\theta}aQ^{ \theta}}^L\|_{B(L^p(\mathbb G))}^{\frac{1}{2}},$$
as desired.
\end{proof}

In particular, we may state the following corollary. 
\begin{cor}
	Assume that $\mathbb G$ is of Kac type. Let $1\leq p\leq \infty$. For all $a\in\prod_{\pi}B(H_{\pi})$, we have
	$$\|a\|_\infty \leq \|a\|_{\mathrm{M}(L^p(\mathbb G))}.$$
\end{cor}
In the following paragraph we present a special subspace of $C_r(\mathbb{G})$,
which, in the classical case, is known as the \emph{Fourier algebra
}introduced by Eymard \cite{eymard1964fourieralg}, as explained in
\cite[p.367]{hewittross1970abstract}. Recall that we always identify $C_r(\mathbb{G})\subset L^\infty (\mathbb G)$ as a subspace of $L^1(\mathbb G)$.
\begin{prop}\label{prop: fourier alg}
Let $A(\mathbb{G})=\{x\in L^{1}(\mathbb{G}):\hat{x}\in\ell^{1}(\hat{\mathbb{G}})\}$.
Then $A(\mathbb{G})\subset C_r(\mathbb{G})$ and 
\begin{equation}
\|x\|_{\infty}\leq\|\hat{x}\|_{1},\quad x\in A(\mathbb{G}).\label{eq:L_infty and l_1 absolute cv series}
\end{equation}
Moreover, if we let $\|x\|_{A}=\|\hat{x}\|_{1}$ for $x\in A(\mathbb{G})$,
then $(A(\mathbb{G}),\|\cdot\|_{A})$ is a Banach space isometrically
isomorphic to $\ell^{1}(\hat{\mathbb{G}})$.\end{prop}
\begin{proof}
Firstly we show that 
\begin{equation}
\|x\|_{\infty}\leq\|\hat{x}\|_{1},\quad x\in\mathrm{Pol}(\mathbb{G}).\label{eq:L_infty l_1 finite absolute cv series}
\end{equation}
Choose an $x\in\mathrm{Pol}(\mathbb{G})$. For each $\pi\in\irr$,
by the identification $L^{\infty}(\mathbb{G})=(L^{\infty}(\mathbb{G})_{*})^{*}$
and the Hölder inequality on $\mathrm{Tr}$ we have 
\begin{align*}
\|(\iota\otimes\mathrm{Tr})\left[(1\otimes\hat{x}(\pi)Q_{\pi})u^{(\pi)}\right]\|_{\infty} & =\sup_{\omega\in L^{\infty}(\mathbb{G})_{*},\|\omega\|=1}\left|\omega\left((\iota\otimes\mathrm{Tr})\left[(1\otimes\hat{x}(\pi)Q_{\pi})u^{(\pi)}\right]\right)\right|\\
 & =\sup_{\omega\in L^{\infty}(\mathbb{G})_{*},\|\omega\|=1}\left|\mathrm{Tr}\left(\hat{x}(\pi)Q_{\pi}\left[(\omega\otimes \iota)u^{(\pi)}\right]\right)\right|\\
 & \leq\sup_{\omega\in L^{\infty}(\mathbb{G})_{*},\|\omega\|=1}\mathrm{Tr}(|\hat{x}(\pi)Q_{\pi}|)\|(\omega\otimes \iota)u^{(\pi)}\|\\
 & \leq\mathrm{Tr}(|\hat{x}(\pi)Q_{\pi}|).
\end{align*}
Therefore by (\ref{fourier series}), 
\[
\|x\|_{\infty}\leq\sum_{\pi\in\irr}d_{\pi}\|(\iota\otimes\mathrm{Tr})\left[(1\otimes\hat{x}(\pi)Q_{\pi})u^{(\pi)}\right]\|_{\infty}\leq\sum_{\pi\in\irr}d_{\pi}\mathrm{Tr}(|\hat{x}(\pi)Q_{\pi}|) = \|\hat{x}\|_{1}.
\]
Now given $y\in A(\mathbb{G})$, we note that $c_{c}(\hat{\mathbb{G}})$
is dense in $\ell^{1}(\hat{\mathbb{G}})$ and $\mathcal{F}(\mathrm{Pol}(\mathbb{G}))=c_{c}(\hat{\mathbb{G}})$,
so we may take $x_{n}\in\mathcal{\mathrm{Pol}}(\mathbb{G})$ such
that $\|\hat{x}_{n}-\hat{y}\|_{1}\to0$. So $(\hat{x}_{n})_{n}$ is
$\|\cdot\|_{1}$-Cauchy, and hence by (\ref{eq:L_infty l_1 finite absolute cv series})
the sequence $(x_{n})_{n}$ is $\|\cdot\|_{\infty}$-Cauchy in $\mathrm{Pol}(\mathbb{G})$.
Since $\mathrm{Pol}(\mathbb{G})$ is dense in $C_r(\mathbb{G})$, the
sequence $(x_{n})_{n}$ in $\mathrm{Pol}(\mathbb{G})$ converges to
some $x\in C_r(\mathbb{G})$. Note that $\|\cdot\|_{1}\leq\|\cdot\|_{\infty}$
on $L^{\infty}(\mathbb{G})$, we have also $\|x_{n}-x\|_{1}\to0$.
Then according to the contractive property of $\mathcal{F}$, it holds
that $\hat{x}(\pi)=\lim_{n}\hat{x}_{n}(\pi)=\hat{y}(\pi)$
for all $\pi\in\irr$. So $y=x\in C_r(\mathbb{G})$ according to the injectivity of $\mathcal F$, and $\|y\|_{\infty}=\lim_{n}\|x_{n}\|_{\infty}\leq\lim_{n}\|\hat{x}_{n}\|_{1}=\|\hat{y}\|_{1}$.
As a result $A(\mathbb{G})\subset C_r(\mathbb{G})$ and (\ref{eq:L_infty and l_1 absolute cv series})
is proved.

Now we set $\|x\|_{A}=\|\hat{x}\|_{1}$ for all $x\in A(\mathbb{G})$,
then $(A(\mathbb{G}),\|\cdot\|_{A})$ is clearly a normed space and
the map $x\mapsto\hat{x}$ sends $(A(\mathbb{G}),\|\cdot\|_{A})$
isometrically into $\ell^{1}(\hat{\mathbb{G}})$. To see that it is surjective,
we shall show that for any $a=(a_{\pi})_{\pi\in\irr}\in \ell^{1}(\hat{\mathbb{G}})$
there exists $x\in L^{1}(\mathbb{G})$ with $\hat{x}=a$. In fact,
if we consider a sequence $x_{n}\in\mathcal{\mathrm{Pol}}(\mathbb{G})$
such that $\|\hat{x}_{n}-a\|_{1}\to0$, then the existence of $x\in L^{1}(\mathbb{G})$
with $\hat{x}=a$ follows simply from the same argument as in the preceding
paragraph. 
\end{proof}

In some literature the Fourier algebra $A(\mathbb G)$ is also defined simply to be the space $\ell^1(\hat{\mathbb G})$. Note that the identification that we use yields also the product on $A(\mathbb G)$ (convolution product of $\ell^1(\hat{\mathbb G})\cong\ell^\infty(\hat{\mathbb G})_*$, well-known in the theory of locally compact quantum groups), which makes $A(\mathbb G)$ a Banach algebra. 

\section{Sidon sets}

Let $\mathbb{G}$ be a compact quantum group and $\mathbf{E}$ be
a subset of $\irr$. Define 
$$
L_{\mathbf{E}}^{\infty}(\mathbb{G})  =  \{x\in L^\infty (\mathbb{G}):\hat{x}(\pi)=0\textrm{ if }\pi\in\irr\backslash\mathbf{E}\},
$$
\[
C_{\mathbf{E}}(\mathbb{G})=C_r(\mathbb{G})\cap L_{\mathbf{E}}^{\infty}(\mathbb{G}),\quad\mathrm{Pol}_{\mathbf{E}}(\mathbb{G})=\mathrm{Pol}(\mathbb{G})\cap L_{\mathbf{E}}^{\infty}(\mathbb{G}),
\]
and

\[
\ell^{\infty}(\mathbf{E})=\{(a_{\pi})_{\pi\in\irr}\in\ell^{\infty}(\hat{\mathbb{G}}):a_{\pi}=0\textrm{ if }\pi\in\irr\backslash\mathbf{E}\},
\]
\[
\ell^{1}(\mathbf{E})=\ell^{1}(\hat{\mathbb{G}})\cap\ell^{\infty}(\mathbf{E}),\ c_{0}(\mathbf{E})=c_{0}(\hat{\mathbb{G}})\cap\ell^{\infty}(\mathbf{E}),\ c_{c}(\mathbf{E})=c_{c}(\hat{\mathbb{G}})\cap\ell^{\infty}(\mathbf{E}).
\]
Then the subspaces $L_{\mathbf{E}}^{\infty}(\mathbb{G}), C_{\mathbf{E}}(\mathbb{G}), \ell^{\infty}(\mathbf{E}), \ell^{1}(\mathbf{E})$ are all closed subspaces. Note that we may identify the duality between spaces $$c_0(\mathbf E)^*=\ell^1(\mathbf E),\quad \ell^1(\mathbf E)^*=\ell^\infty (\mathbf E)$$ 
via the bracket $\langle a,b \rangle=\hat h(ba)=\sum_{\pi}d_\pi\mathrm{Tr}(a_\pi Q_\pi b_\pi)$ for $a\in \ell^1(\mathbf E),b\in \ell^\infty (\mathbf E)$. 
\begin{defn}\label{defn: sidon}
We say that a subset $\mathbf E \subset \irr $ is a \emph{Sidon set (with constant $K$)} if there exists $K>0$ such that for any $x\in\mathrm{Pol_{\mathbf{E}}}(\mathbb{G})$, we have
$$\|\hat{x}\|_{1}\leq K\|x\|_{\infty}.$$
\end{defn}
\begin{rem}\label{rem:sidon union finite}
	Any finite subset of $\irr $ is clearly a Sidon set. Let $\mathbf{E}\subset\mathrm{Irr}(\mathbb{G})$
	be a Sidon set and $\mathbf{F}\subset\mathrm{Irr}(\mathbb{G})$ be
	finite. Then $\mathbf{E}\cup\mathbf{F}$ is also a Sidon set. In fact, assume without loss of generality that $\mathbf{E}\cup\mathbf{F}=\emptyset$, and take $x\in \mathrm{Pol}_{\mathbf E}(\mathbb G)$ and $y\in \mathrm{Pol}_{\mathbf F}(\mathbb G)$. Since $\mathrm{Pol}_{\mathbf F}(\mathbb G)$ is finite-dimensional, it is complemented in $L^\infty(\mathbb G)$. Therefore, there exist two constants $K_1,K_2>0$ such that
	\begin{align*}
		\|\hat x +\hat y \|_1 &\leq \|\hat x\|_1 +\|\hat y\|_1 \leq K_1\|x\|_\infty +K_1 \|y\|_\infty \leq K_1 \|x+y\|_\infty +2K_1 \|y\|_\infty\\
		&  \leq K_1 \|x+y\|_\infty +2K_1K_2 \|x+y\|_\infty  =(K_1+2K_1K_2)\|x+y\|_\infty.
	\end{align*}
\end{rem}

We first  give the following fundamental characterizations of Sidon set, extending the classical result of \cite[(37.2)]{hewittross1970abstract} for compact groups. On the other hand, this result establishes the equivalence between the so-called strong Sidon sets (i.e. sets with condition (3) in the theorem below) and Sidon sets in non-amenable discrete groups, which had been open since the   work of Picardello \cite{picardello73lacunary} in 1970s. Our approach is different from the idea of \cite{hewittross1970abstract,picardello73lacunary}. In hindsight, the proof in \cite{hewittross1970abstract} depends essentially on the coamenability of the compact group, which does not apply to the more general cases in the quantum setting. Instead, we use a simpler argument via duality.
\begin{thm}
\label{thm:sidon_l1_interpolation state}Let $\mathbb{G}$ be a compact
quantum group and $\mathbf{E}$ be a subset of $\irr$. The following
assertions are equivalent:
\begin{enumerate}[label=\emph{(\arabic{enumi})}]
\item  $\mathbf{E}$ is a Sidon set; 
\item for any $a\in\ell^{\infty}(\mathbf{E})$, there exists $\varphi\in C_r(\mathbb{G})^{*}$
such that $\hat{\varphi}(\pi)=a_{\pi}$ for all $\pi\in\mathbf{E}$;
\item for any $a\in c_{0}(\mathbf{E})$, there exists $x\in L^{1}(\mathbb{G})$
such that $\hat{x}(\pi)=a_{\pi}$ for all $\pi\in\mathbf{E}$;
\item $L_{\mathbf{E}}^{\infty}(\mathbb{G})\subset A(\mathbb{G})$;
\item $C_{\mathbf{E}}(\mathbb{G})\subset A(\mathbb{G})$;
\item there exists $K>0$ such that for any $x\in C_{\mathbf{E}}(\mathbb{G})$,
$\|\hat{x}\|_{1}\leq K\|x\|_{\infty}$; 
\item there exists $K>0$ such that for any $x\in L_{\mathbf{E}}^{\infty}(\mathbb{G})$,
$\|\hat{x}\|_{1}\leq K\|x\|_{\infty}$.
\end{enumerate}
\end{thm}
\begin{proof}
From Proposition \ref{prop: fourier alg}, it is easy to see that the condition (4) implies the surjectivity of the inverse Fourier transform $\mathcal F ^{-1}:\ell^1(\mathbf E)\to L^\infty_{\mathbf E}(\mathbb G)$. 
Then the equivalence $(4)\Leftrightarrow(7)$
follows from the open mapping theorem. Similarly, we may obtain $(5)\Leftrightarrow(6)$. Also, the implications $(7) \Rightarrow (6) \Rightarrow (1)$ are trivial. Let us show  $(1)\Rightarrow(2)\Rightarrow (3) \Rightarrow (1)$ and $(3) \Rightarrow(7)$. In the following set $\mathrm{Pol}_{\mathbf{E}}(\mathbb{G})^c=\{x\in\mathrm{Pol}(\mathbb{G}):x^{*}\in\mathrm{Pol}_{\mathbf{E}}(\mathbb{G})\}$.

$(1)\Rightarrow(2)$. 
Take $a\in\ell^{\infty}
(\mathbf{E})$. We consider
the functional
\[
\varphi:\ \mathrm{Pol}_{\mathbf{E}}(\mathbb{G})^c\to\mathbb{C},\quad x\mapsto\overline{\hat{h}(a^{*}\widehat{x^{*}})}.
\]
According to (1), we have $$|\varphi(x)|=|\hat{h}(a^{*}\widehat{x^{*}})|\leq\|a^{*}\|_{\infty}\|\widehat{x^{*}}\|_{1}\leq K\|a^{*}\|_{\infty}\|x^{*}\|_{\infty}=K\|a\|_{\infty}\|x\|_{\infty},$$
so $\varphi$ is continuous on $\mathrm{Pol}_{\mathbf{E}}(\mathbb{G})^c$ and it has a Hahn-Banach extension
to $C_r(\mathbb{G})$. We still denote the extension by $\varphi$.
Recall that by Proposition \ref{prop:Plancherel}, $$(u_{ji}^{(\pi)})\,\hat{}\,(\pi)=d_\pi^{-1}e_{ij}^{(\pi)}Q_\pi^{-1}.$$
We have for $\pi\in\mathbf{E}$,
\begin{align*}
\hat{\varphi}(\pi) & =(\varphi\otimes\iota)((u^{(\pi)})^{*})=\sum_{1\leq i,j\leq n(\pi)}\varphi((u_{ji}^{(\pi)})^{*})e_{ij}^{(\pi)}\\
 & =\sum_{1\leq i,j\leq n(\pi)}\overline{\hat{h}(a^{*}(u_{ji}^{(\pi)})\,\hat{}\,(\pi))}e_{ij}^{(\pi)}=d_{\pi}^{-1}\sum_{1\leq i,j\leq n(\pi)}\overline{\hat{h}(a^{*}e_{ij}^{(\pi)}Q_{\pi}^{-1})}e_{ij}^{(\pi)}\\
 & =\sum_{1\leq i,j\leq n(\pi)}\overline{\mathrm{Tr}(a_{\pi}^{*}e_{ij}^{(\pi)})}e_{ij}^{(\pi)}=\sum_{1\leq i,j\leq n(\pi)}\overline{(a_{\pi}^{*})_{ji}}e_{ij}^{(\pi)}=a_{\pi},
\end{align*}
as desired.

$(2)\Rightarrow (3)$. We  consider the map
$$\sigma:L^{1}(\mathbb{G}) \to c_{0}(\mathbf{E}),\quad \psi\mapsto\hat{\psi}|_{\mathbf{E}}.$$
Then the second adjoint map $(\sigma^*)^*$ is given by
$$\sigma^{**}:L^\infty (\mathbb{G})^* \to \ell^\infty(\mathbf{E}),\quad \psi\mapsto\hat{\psi}|_{\mathbf{E}}.$$
Note that the condition (2) means nothing but the surjectivity of $\sigma^{**}$. Recall the general fact that for a bounded map $T$ between two Banach spaces, $T$ is surjective if and only if $T^{**}$ is surjective (cf. for example \cite[3.1.22]{megginson98banachbook}). So $\sigma$ is also surjective, whence the condition (3).

$(3)\Rightarrow(1)$. The assertion (3) implies that the bounded map $\sigma:\psi\mapsto\hat{\psi}|_{\mathbf{E}}$
from $L^{1}(\mathbb{G})$ to $c_{0}(\mathbf{E})$ is surjective. 
By the open mapping theorem, we may find a constant $K>0$ such that for all $a\in c_0(\mathbf E)$, there exists $x\in L^1(\mathbb G)$ satisfying $\hat{x}(\pi)=a_{\pi}$ for all $\pi\in\mathbf{E}$ and $\|x\|_1\leq K\|a\|_\infty$.

Now consider $y\in \mathrm{Pol}_{\mathbf E}(\mathbb G)$ and let us show that $\|\hat{y}\|_{1}\leq K\|y\|_{\infty}$. Equivalently, let us prove that for all $a\in c_c(\mathbf E)$ with $\|a\|_\infty\leq 1$, we have 
\begin{equation}\label{l infty in fourier by dual}
|\hat h (a^*\hat y)|\leq K\|y\|_\infty.
\end{equation}
In the following let $y$ and $a$ be fixed as above. Choose a $\psi\in L^1(\mathbb G)$ satisfying $\hat{\psi}(\pi)=a_{\pi}$ for all $\pi\in\mathbf{E}$ and $\|\psi\|_1\leq K$. Note also that  $\psi|_{\mathrm{Pol}_{\mathbf E}(\mathbb G)^c}=(\mathcal{F}^{-1}(a)h)|_{\mathrm{Pol}_{\mathbf E}(\mathbb G)^c}$. Hence together with Proposition \ref{prop:Plancherel} and the choice of $y$, we get 
\begin{align*}
|\hat h(a^*\hat y)| & =|\hat h(\hat y^*a)|=|h(y^*\mathcal{F}^{-1}(a))+(\psi-\mathcal{F}^{-1}(a)h)(y^*)|
=|\psi(y^*)| \leq K \|y\|_\infty.
\end{align*}
Therefore the desired inequality \eqref{l infty in fourier by dual} follows.

$(3)\Rightarrow(7)$. We have proved that (3) implies (1) and on the other hand, for any set $\mathbf E \subset \irr$ satisfying (1), we have the following observation: if $\psi\in L^1(\mathbb G)$ satisfies $\hat \psi |_{\mathbf E}=0$, then for all $x\in L^\infty_{\mathbf E}(\mathbb G)$, we have $\psi(x^*)=0$. In fact, let $(y_n)\subset\mathrm{Pol}(\mathbb G)$ be a sequence such that $y_nh$ converges to $\psi$ in $L^1(\mathbb G)$ and let $x\in L^\infty_{\mathbf E}(\mathbb G)$. We then note that $\hat \psi ^*\hat x=0, (\hat y_n-\hat \psi) ^*\hat x\in c_c(\mathbf E)$ and by \eqref{eq:fourier series for phi} and \eqref{conv another} we have
\begin{align*}
|\hat h (\hat y_n ^*\hat x)|
& = |\hat h ((\hat y_n-\hat \psi) ^*\hat x)|\leq \|(\hat y_n-\hat \psi) ^*\hat x\|_1
= \|\mathcal F((\iota\otimes (y_nh-\psi)^*)\Delta(x))\|_1 \\
& \leq K\|(\iota\otimes (y_nh-\psi)^*)\Delta(x)\|_\infty 
\leq K\|y_nh-\psi\|_1\|x\|_\infty \to 0,
\end{align*} 
where $K$ is the Sidon constant for $\mathbf E$. Consequently by the choice of $y_n$ and Proposition \ref{prop:Plancherel} we have 
$$\psi(x^*)=\lim_n h(x^*y_n)=\lim_n \overline{\hat h (\hat y_n ^* \hat x)}=0,$$
as claimed. Then the implication $(3) \Rightarrow (7)$ follows from the same argument as in $(3) \Rightarrow (1)$ above.
\end{proof}
\begin{rem}
	Note that for a Sidon set $\mathbf E$, by the assertion (4) and Proposition \ref{prop: fourier alg}, we have 
	$$L^\infty_{\mathbf E}(\mathbb G)=C_{\mathbf E}(\mathbb G).$$
	By the assertion (7), $C_{\mathbf E}(\mathbb G)$ is isomorphic to a Banach subspace of $\ell^1(\mathbf E)$ via the Fourier transform. Since $c_c(\mathbf E)$ is dense in $\ell^1(\mathbf E)$, $C_{\mathbf E}(\mathbb G)$ and $\ell^1(\mathbf E)$ are indeed isomorphic. In particular, the subspace $\mathrm{Pol}_{\mathbf E}(\mathbb G)$ is dense in $C_{\mathbf E}(\mathbb G)$. 
\end{rem}
\begin{rem}\label{adaption sidon two}
Since by \eqref{eq:fourier series for phi} we know that
$(\varphi^{*}\circ S^{-1})\,\hat{}\,(\pi)=\hat{\varphi}(\pi)^{*}$
for $\pi\in\mathrm{Irr}(\mathbb{G})$, the assertion
(2) in the above theorem can be replaced by the following one:

($ 2' $) for any $a\in\ell^{\infty}(\mathbf{E})$, there exists $\varphi\in C_r(\mathbb{G})^{*}$
such that $(\varphi\circ S^{-1})\,\hat{}\,(\pi)=a_{\pi}$ for
all $\pi\in\mathbf{E}$.
\end{rem}

\begin{rem}\label{rem: sidon const}
Note that we have shown in the proof that, if $\mathbf E$ is a Sidon set of constant $K$, then the obtained elements $\varphi$ and $x$ in (2) and (3) can be chosen to have the norms not more than $K\|a\|_\infty $, respectively; conversely arguing as in the beginning of the proof $(3)\Rightarrow (1)$, if (2) or (3) holds, we may find a constant $K>0$ such that the norms of $\varphi$ or $x$ is not more than $K\|a\|_\infty$ respectively, and the obtained Sidon constant in (1) is exactly $K$.
\end{rem}

\begin{rem}
In view of some technical tricks concerning the non-traciality of Haar states, we would like to present a second proof of the implication $(1)\Rightarrow (7)$ in the above theorem under an additional assumption that $\mathbb G$ is coamenable, roughly following the idea in  \cite[(37.2)]{hewittross1970abstract}. The subtle point, which is trivially hidden in the commutative and cocommutative cases, is the fact that to directly deduce the convergence of Fourier series of $x$ from \eqref{sum diff fourier pf with coamen} below as in \cite[(37.2)]{hewittross1970abstract}, one needs to know that the convolution \eqref{conv def} defines a bounded map from $L^1(\mathbb G)\times L^\infty(\mathbb G)$ into $L^\infty(\mathbb G)$ by restriction. This is generally not clear for a compact quantum group which is not of Kac type, caused by the unboundedness of the antipode $S$. Now let us assume the coamenability of $\mathbb G$, and show the implication $(1) \Rightarrow (7)$:

Assume (1). Take $x\in L_{\mathbf{E}}^{\infty}(\mathbb{G})$. Since
$\mathbb{G}$ is coamenable and 
the subspace $\mathrm{Pol}(\mathbb{G})$ is dense in $L^{1}(\mathbb{G})$, we may choose a net $(y_{i})_{i\in I}\subset \mathrm{Pol}(\mathbb{G})$
with $\|y_{i}\|_{1}\leq 1$ such that $\lim_{i}\|x\star y_{i}-x\|_{1}=0$.
Hence for each $\pi\in\mathbf{E}$,
\[
\|\hat{y}_{i}(\pi)\hat{x}(\pi)-\hat{x}(\pi)\|_{B(H_{\pi})}\leq\|\hat{y}_{i}\hat{x}-\hat{x}\|_{\infty}\leq\|x\star y_{i}-x\|_{1}\to0.
\]
Since the norms on a finite dimensional space are equivalent, we have
for all $\pi\in\mathbf{E}$,
\[
\lim_{i}\mathrm{Tr}(|(\hat{y}_{i}(\pi)\hat{x}(\pi)-\hat{x}(\pi))Q_{\pi}|)=0.
\]
Therefore for any finite subset $\mathbf{F}\subset\mathbf{E}$ and
for any $\varepsilon>0$, we may find some $a\in c_c(\mathbf E)$
with $\|a\|_{\infty}\leq 1$ and \begin{equation}\label{sum diff fourier pf with coamen}
\sum_{\pi\in\mathbf{F}}d_{\pi}\mathrm{Tr}(|(a_\pi\hat{x}(\pi)-\hat{x}(\pi))Q_{\pi}|)<\varepsilon.
\end{equation}
Since $\mathbf E$ is a Sidon set, we may find by previous arguments (proof of $(1)\Rightarrow (2)\Rightarrow (3)$) a functional $\varphi\in L^\infty (\mathbb G)_*=L^1(\mathbb G)$ with $\|\varphi\|_1\leq K$ and $\hat{\varphi}(\pi)=a_{\pi}^*$ for
all $\pi\in\mathbf{E}$. Take $y\in \mathrm{Pol}(\mathbb G)$ such that $\|\varphi-y\|_1<\varepsilon$. Then $\|a_\pi - ((yh)^*\circ S^{-1})\,\hat{}\,(\pi)\|_{B(H_\pi)}= \|\hat{\varphi}(\pi)^*-\hat y (\pi)^*\|_{B(H_\pi)}<\varepsilon$ for $\pi\in\mathbf{E}$. Further together with \eqref{conv another}, 
\begin{align*}
  \sum_{\pi\in\mathbf{F}}d_{\pi}\mathrm{Tr}(|\hat{x}(\pi)Q_{\pi}|)
& \leq  \sum_{\pi\in\mathbf{F}}d_{\pi}\mathrm{Tr}(|a_\pi\hat{x}(\pi) Q_{\pi}|)+\sum_{\pi\in\mathbf{F}}d_{\pi}\mathrm{Tr}(|(a_\pi\hat{x}(\pi)-\hat{x}(\pi))Q_{\pi}|)\\
& \leq  \sum_{\pi\in\mathbf{F}}d_{\pi}\mathrm{Tr}(|((yh)^*\circ S^{-1})\,\hat{}\,(\pi)\hat{x}(\pi) Q_{\pi}|)+\varepsilon \sum_{\pi\in\mathbf{F}}d_{\pi}\mathrm{Tr}(|(\hat{x}(\pi)Q_{\pi}|) +\varepsilon\\
& \leq   \|\mathcal F((\iota\otimes(yh)^*)\Delta(x))\|_{1}+\varepsilon \sum_{\pi\in\mathbf{F}}d_{\pi}\mathrm{Tr}(|(\hat{x}(\pi)Q_{\pi}|)+\varepsilon \\
& \leq K\|(\iota\otimes(yh)^*)\Delta(x)\|_{\infty}+\varepsilon \sum_{\pi\in\mathbf{F}}d_{\pi}\mathrm{Tr}(|(\hat{x}(\pi)Q_{\pi}|)+\varepsilon\\
& \leq  K(K+\varepsilon)\|x\|_{\infty}+\varepsilon \sum_{\pi\in\mathbf{F}}d_{\pi}\mathrm{Tr}(|(\hat{x}(\pi)Q_{\pi}|)+\varepsilon
\end{align*}
where we have applied the property of the Sidon set to the element $(\iota\otimes(yh)^*)\Delta(x)\in\mathrm{Pol}_{\mathbf{E}}(\mathbb{G})$.
Since $\mathbf{F}$ and $\varepsilon$ are arbitrarily chosen, we
get $\|\hat{x}\|_{1}\leq K^2\|x\|_{\infty}$, as desired. Note that the constant $K^2 $ obtained here is worse than that in the previous proof.
\end{rem}

As a corollary we may give a quick proof of the non-surjectivity of the Fourier transform $\mathcal F:L^1(\mathbb G)\to c_0(\hat{\mathbb G})$ for infinite compact quantum group $\mathbb G$.

\begin{cor}
Let $\mathbb G$ be a compact quantum group. The following conditions are equivalent:

\emph{(1)} $\mathbb G$ is finite, i.e., $L^\infty (\mathbb G)$ is a finite-dimensional space;

\emph{(2)} $\irr$ is a Sidon set;

\emph{(3)} $\mathcal F:L^1(\mathbb G)\to c_0(\hat{\mathbb G})$ is surjective;

\emph{(4)} $\mathcal F:L^\infty(\mathbb G)^*\to \ell^\infty(\hat{\mathbb G})$ is surjective.
\end{cor}
\begin{proof}
The equivalence $(2)\Leftrightarrow (3) \Leftrightarrow (4)$ has been already given in Theorem \ref{thm:sidon_l1_interpolation state}. The implication $(1)\Rightarrow (2)$ is trivial. Assume $(2)$ holds, then the previous theorem yields that there exists a constant $K>0$ such that
\begin{equation}
\label{eq: sidon not whole}
\forall \, x\in L^\infty (\mathbb G),\quad \|x\|_\infty \leq \|\hat x\|_1 \leq K \|x\|_\infty.
\end{equation}
Suppose by contradiction with (1) that $\mathbb G$ is not finite. Then we may choose an infinite countable subset $\mathbf E \subset \irr$ and let $A$ be the $*$-subalgebra generated  by $\{\chi_\pi:\pi\in\mathbf E\}$ in $\mathrm{Pol}(\mathbb G)$. Write $$\mathbf F=\{\pi\in\irr: \exists \pi_1,\ldots,\pi_n\in \mathbf E, \pi \text{ is a subrepresentation of } \sigma_1\otimes\cdots \otimes \sigma_n, \sigma_i=\pi_i\text{ or } \bar \pi _i \},$$ 
then $\mathbf F$ is countable and $A\subset \mathrm{Pol}_{\mathbf F}(\mathbb G)$. Consider the von Neumann subalgebra $\mathcal{M}$ generated by $A$ in $L^\infty (\mathbb G)$. Then by the weak density of $A$ in $\mathcal{M}$, for each $x\in \mathcal{M}$ and each $\pi\in \irr\setminus \mathbf F$ we have $\hat x (\pi)=0$. So by the above inequality \eqref{eq: sidon not whole} each $x\in \mathcal{M}$ can be approximated in $\|\|_\infty$ by elements in $\mathrm{Pol}_{\mathbf F}(\mathbb G)$, and in particular $\mathcal{M}$ is separable, which gives a contradiction since the von Neumann algebra $\mathcal{M}$ is infinite-dimensional as so is $A$.
\end{proof}
\begin{rem}
Together with the condition (5) in Theorem \ref{thm:sidon_l1_interpolation state}, the above argument also shows that for any infinite discrete quantum group $\mathbb{H}$, the Fourier transform $\mathcal{F}:\ell^1(\mathbb{H})\to C(\hat{\mathbb{H}})$ in the sense of \cite{kahng2010fourier,caspers2013fourier} is not surjective. The above result is a particular case of the general fact that the predual of an infinite-dimensional von Neumann algebra cannot be equipped with an equivalent C*-norm.
\end{rem}
%Recall that the counit $\epsilon$ on a compact quantum group $\mathbb{G}$
%is a linear functional on $\mathrm{Pol}(\mathbb{G})$ defined by $\epsilon(u_{ij}^{(\pi)})=\delta_{ij}$
%for all $u^{(\pi)}=[u_{ij}^{(\pi)}]\in\mathrm{Irr}(\mathbb{G})$.

The following properties give some general methods of constructing infinite Sidon sets for compact quantum groups.

\begin{prop}\label{prop:direct sum Sidon}
Let $(\mathbb{G}_{i})_{i\in I}$ be a family of compact quantum groups
and assume that $\mathbf{E}_{i}\subset\mathrm{Irr}(\mathbb{G}_{i})$
is a Sidon set with constant $C_{i}$ for each $i\in I$ and that
$C\coloneqq \sup_{i\in I}C_{i}<\infty$. Then $\cup_{i\in I}\mathbf{E}_{i}\subset\mathrm{Irr}(\prod_{i\in I}\mathbb{G}_{i})$
is a Sidon set for $\prod_{i\in I}\mathbb{G}_{i}$.\end{prop}
\begin{proof}
	This follows directly from the assertion (2) in Theorem \ref{thm:sidon_l1_interpolation state} and Remark \ref{rem: sidon const}. Let $a\in \ell^\infty(\cup_{i\in I}\mathbf{E}_{i})$. Without loss of generality we assume $\|a\|_\infty =1/C$. For each $i\in I$, we may find $\varphi_i\in C_r(\mathbb G _i)^*$ such that $\hat \varphi_i$ coincides with $a$ on $\mathbf{E}_{i}$ and $\|\varphi_i\|\leq 1$. Take $\varphi=\otimes_{i\in I} \varphi_i$, then $\varphi$ extends to a bounded functional on $C_r(\prod_{i\in I}\mathbb{G}_{i})$.
	Hence $\hat \varphi(\pi)=a_\pi$ for all $\pi\in \cup_{i\in I}\mathbf{E}_{i}$. So $\cup_{i\in I}\mathbf{E}_{i}$ is a Sidon set.
\end{proof}

\begin{prop}
Let $\mathbb G_1,\mathbb G_2$ be compact quantum groups. Assume that $\mathbf E\subset \mathrm{Irr}(\mathbb G_1)$ and $\mathbf F\subset \mathrm{Irr}(\mathbb G_2)$ are Sidon sets. Then $\mathbf E \cup \mathbf F$ is a Sidon set for $\mathbb G_1 \hat{*}\mathbb G_2$.
\end{prop}
\begin{proof}
Denote $h_1$ and $h_2$ the Haar states for $\mathbb G_1$ and $\mathbb G_2$ respectively. Let $K_1$ be the Sidon constant for $\mathbf E $ and $K_2$ that for $\mathbf F$. By Remark \ref{rem:sidon union finite}, we may assume that $1\notin \mathbf E$ and $1\notin \mathbf F$. Now for any $x\in \mathrm{Pol}_{\mathbf E}(\mathbb G_1 \hat{*}\mathbb G_2)$ and $y\in \mathrm{Pol}_{\mathbf F}(\mathbb G_1 \hat{*}\mathbb G_2)$, we have $h_1(x)=0$, $h_2(y)=0$, and  $ x $ and $  y $ are free. Then it is well known and easy to see from the construction of reduced free products that $\max \{\|x\|_\infty,\|y\|_\infty\}\leq \|x+y\|_\infty$ (see \cite{voiculescu98free,junge05oh,ricardxu06free} for more information on the norm estimates related to freeness). Hence 
$$\|\hat x+\hat y\|_1\leq \|\hat x\|_1+\|\hat y\|_1\leq K_1\|x\|_\infty +K_2\|y\|_\infty\leq (K_1+K_2) \|x+y\|_\infty.$$
This proves that $\mathbf E \cup \mathbf F$ is a Sidon set of constant $K_1+K_2$.
\end{proof}
\begin{rem}
Note that one cannot expect to extend the above proposition to an infinite family of compact quantum groups as in Proposition \ref{prop:direct sum Sidon}. An easy example is the set of infinitely many free generators of the free group $\mathbb F_\infty$, which is not a Sidon set. More details of this example will be presented in Remark \ref{rem:sidon neq usidon}.
\end{rem}
%In the following property we take the notation
%\[
%(L^{p}(\mathbb{G}),\|\|_{p})=(L^{\infty}(\mathbb{G}),L^{1}(\mathbb{G}))_{1/p},\quad1\leq p\leq\infty
%\]
%to be the Izumi $L^{p}$-spaces defined by interpolation and $(L^{p}(\mathbb{G})_{H},\|\|_{p,H})$
%with $1\leq p\leq\infty$ to be the Haagerup $L^{p}$-spaces. Denote
%by $D$ the associated Radon-Nikodym derivate used in Haagerup's construction
%for which $x\mapsto xD^{1/p}$ gives an isometry from $L^{p}(\mathbb{G})$
%to $L^{p}(\mathbb{G})_{H}$.

As is seen in \cite{picardello73lacunary,bozejko81newlacunary}, there are several alternative possible ways to generalize the notion of Sidon sets for non-amenable cases. Let us briefly discuss them in the quantum group setting. We follow the terminologies in \cite{bozejko81newlacunary,harcharras99nclambdap}.

\begin{defn}\label{def:interpolation set infty}
(1) We say that a subset $\mathbf E \subset \irr $ is a \emph{weak Sidon set (with constant $K$)} if there exists $K>0$ such that for any $x\in\mathrm{Pol_{\mathbf{E}}}(\mathbb{G})$, we have
$$\|\hat{x}\|_{1}\leq K\|x\|_{C_u(\mathbb G)}.$$

(2) We say that $\mathbf E \subset \irr $ is an \emph{interpolation set of $\mathrm M(L^\infty({\mathbb{G}}))$ (resp., of $\mathrm M_L(L^\infty({\mathbb{G}}))$, of $\mathrm M_R(L^\infty({\mathbb{G}}))$)} with constant $K$ if for any $a\in\ell^{\infty}(\mathbf{E})$, there exists a bounded multiplier
$\tilde{a}\in \mathrm M(L^\infty({\mathbb{G}}))$ (resp., $\tilde{a}\in \mathrm M_L(L^\infty({\mathbb{G}}))$, $\tilde{a}\in \mathrm M_R(L^\infty({\mathbb{G}}))$) with $\|\tilde a \|_{\mathrm M(L^\infty({\mathbb{G}}))} \leq K\|a\|_\infty$ (resp.,  $\|m_{\tilde{a}}^L\|\leq K\|a\|_\infty$, $\|m_{\tilde{a}}^R\|\leq K\|a\|_\infty$) such that $\tilde{a}_{\pi}=a_{\pi}$
for all $\pi\in\mathbf{E}$.

(3) We say that a subset $\mathbf E \subset \irr $ is a left (resp., right) \emph{unconditional Sidon set (with constant $K$)}  if there exists $K>0$ such that for any unitary $a\in\ell^{\infty}(\mathbf{E})$
and for any $x\in\mathrm{Pol_{\mathbf{E}}}(\mathbb{G})$, $\|m_{a}^Lx\|_{\infty}\leq K\|x\|_{\infty}$ (resp., $\|m_{a}^Rx\|_{\infty}\leq K\|x\|_{\infty}$).
\end{defn}

\begin{rem}
(1) Following the same idea as in the proof $(1) \Rightarrow (2)$ and $(3)\Rightarrow (1)$ in Theorem \ref{thm:sidon_l1_interpolation state}, one can see easily that a subset $\mathbf E \subset \mathrm{Irr}(\mathbb G)$ is a weak Sidon set of constant $K$ if and only if for all $a\in \ell^\infty(\mathbf E)$, there exists $\varphi\in C_u(\mathbb G)^*$ such that $\|\varphi\|\leq K\|a\|_\infty$ and  $\hat{\varphi}(\pi)=a_\pi$ for all $\pi\in \mathbf E$.
Evidently, a Sidon set for $\mathbb G$ is necessarily a weak Sidon set.

(2) If $\mathbb G$ is of Kac type, by Lemma \ref{lem: left to right multipliers} we see that the classes of  interpolation sets of $\mathrm M_L(L^\infty({\mathbb{G}}))$ and of $\mathrm M_R(L^\infty({\mathbb{G}}))$ coincide. Also we note that if $\mathbb G$ is coamenable, these two classes coincide as well, which can be seen from the following theorem. Some more properties of these classes of interpolation sets will be discussed in the next section.
%
%(3) The sequence $\tilde a$ given in the above definition (2) satisfies 
%$$\|\tilde a\|_\infty \leq \|m_{\tilde{a}}^L\|\quad
%(\|\tilde a\|_\infty \leq \|m_{\tilde{a}}^L\|,\ \text{resp.}).$$
%In fact by \eqref{eq:multiplier, Fourier coefficient} and via unitary equivalence we may assume that each $\tilde{a}_\pi$ for $\pi\in \irr$ is a diagonal matrix with entries $a_1,\ldots,a_{n_\pi}$, then we see that for $1\leq j\leq n_\pi$
%$$|a_j|=\|a_j u_{jj}^{(\pi)}\|_\infty=\|m_{\tilde{a}} u_{jj}^{(\pi)}\|_\infty \leq \|m_{\tilde{a}}^L\|.$$
\end{rem}

\begin{thm}\label{sidon and u sidon}
Let $\mathbb{G}$ be a compact quantum group and $\mathbf{E}\subset\mathrm{Irr}(\mathbb{G})$
be a subset. Let $K>0$. If 
\begin{enumerate}\renewcommand{\labelenumi}{\emph{(\theenumi)}}
\item $ \mathbf{E} $ is a weak Sidon set of constant $K$;
\end{enumerate}
then 
\begin{enumerate}
\item[\emph{(2)}] $\mathbf{E}$ is an interpolation set of $\mathrm M(L^\infty({\mathbb{G}}))$ of constant $K$;
\end{enumerate}
and in particular,
\begin{enumerate}
\item[\emph{(3)}] $ \mathbf{E} $ is a left unconditional Sidon set of constant $K$.
\end{enumerate}
Moreover, if additionally $\mathbb G$ is coamenable, then the conditions \emph{(1), (2)} and \emph{(3)} are all equivalent to:
\begin{enumerate}
\item[\emph{(4)}] $ \mathbf{E} $ is a Sidon set of constant $K$.
\end{enumerate}
\end{thm}
\begin{proof}
The proof follows the same line as in \cite{picardello73lacunary,bozejko81newlacunary} and we only present the sketch. Assume that $\mathbf E \subset\irr$ is a weak Sidon set. Then by the above remark and Remarks \ref{adaption sidon two}-\ref{rem: sidon const} any $a\in \ell^\infty (\mathbf E)$ is a restriction of $\hat\varphi $ for some $\varphi\in C_u(\mathbb G)^*$ of norm no more than $K\|a\|_\infty$. Note that $C_r(\mathbb G)$ is a quotient space of $C_u(\mathbb G)$ and hence the dual space $C_r(\mathbb G)^*$ embeds isometrically into $C_u(\mathbb G)^*$. Therefore by the density of $\mathrm{Pol}(\mathbb G)$ in $L^1(\mathbb G)=L^\infty(\mathbb G)_*$, one can  easily see from \eqref{conv fourier series} that $\psi\mapsto \varphi \star \psi$ and $\psi\mapsto  \psi\star \varphi$ give two bounded multipliers on $L^1(\mathbb G)$. Then by duality we obtain the desired multipliers for (2). Thus the implication $(1)\Rightarrow (2) \Rightarrow (3)$ is established.

Now assume additionally $\mathbb{G}$ is coamenable and show that $(3)\Rightarrow (4)$.
%, any bounded left multiplier $\tilde{a}\in\ell^{\infty}(\hat{\mathbb{G}})$
%corresponds to a continuous functional $\varphi\in C_r(\mathbb{G})^{*}$
%given by $\varphi(x)=\epsilon(m_{\tilde{a}}x)$, such that $(\varphi\circ S^{-1})\,\hat{}=\tilde{a}$.
%So the equivalence $(1)\Leftrightarrow(2)$ follows from Theorem \ref{thm:sidon_l1_interpolation state}
%and its remark. $(1)\Rightarrow(3)$ follows from \eqref{eq:L_infty and l_1 absolute cv series}.
%Now assume (3).
 Take $x\in\mathrm{Pol}_{\mathbf{E}}(\mathbb{G})$
and  for each $\pi\in\mathbf{E}$ let $a_{\pi}$ be a unitary matrix such that $|\hat{x}(\pi)Q_{\pi}|=a_{\pi}\hat{x}(\pi)Q_{\pi}$. Then 
\begin{eqnarray*}
\|\hat{x}\|_{1} & = & \sum_{\pi\in\mathbf{E}}d_{\pi}\mathrm{Tr}(|\hat{x}(\pi)Q_{\pi}|)=\sum_{\pi\in\mathbf{E}}d_{\pi}\mathrm{Tr}(a_{\pi}\hat{x}(\pi)Q_{\pi})\\
 & = & \epsilon\Big(\sum_{\pi\in\mathbf{E}}d_{\pi}(\iota\otimes\mathrm{Tr})[(1\otimes a_{\pi}\hat{x}(\pi)Q_{\pi})u^{(\pi)}]\Big)=\epsilon(m_{a}x)\\
 & \leq & \|m_{a}x\|_{\infty}\leq K\|x\|_{\infty},
\end{eqnarray*}
as desired.\end{proof}
\begin{rem}\label{rem:sidon neq usidon}
	(1) The left unconditional Sidon set in the assertion (3) above can be obviously replaced by the right one.
	
	(2) The coamenability is crucial in the above proposition. In fact, denote by $\mathbb F_\infty$ the free group with infinitely many generators and let $\mathbb{G}$ be the quantum group with dual $\hat{\mathbb{G}}={\mathbb{F}}_{\infty}$.
Take $\mathbf{E}$ to be the generators of $\mathbb{F}_{\infty}$, and recall the Haagerup inequality \cite{leinert74freekhintchine,bozejko75optimal}: for finitely many elements $\gamma_1,\ldots,\gamma_n\in \mathbf E$ and $\alpha_1,\ldots,\alpha_n\in \mathbb C$,
$$\Big(\sum_{k=1}^{n}|\alpha_k|^2\Big)^{1/2}\leq \Big\|\sum_{k=1}^{n}\alpha_k\lambda(\gamma_k)\Big\|_{VN(\mathbb F_\infty)}\leq 2\Big(\sum_{k=1}^{n}|\alpha_k|^2\Big)^{1/2}.$$
So
$\mathbf{E}$ satisfies (2) and (3) in the proposition, but in this
case obviously (1) fails to hold.\end{rem}

\begin{example}
\label{free unitary weak sidon}
Consider the compact quantum group $\mathbb G=\prod_{k\geq 1}\mathbb G_k$, where for each $k\geq 1$ and $N_k\geq 1$, $\mathbb G_k=U_{N_k}^+$ denotes the free unitary group of Wang \cite{vandalewang96universal}. Recall that $C_u(U_{N_k}^+)$ is the universal C*-algebra generated by $N^2_k$ elements $\{u_{ij}^{(k)}:1\leq i,j\leq N_k\}$ such that the matrix $u^{(k)}=[u_{ij}^{(k)}]$ is unitary. Take $\mathbf E=\{u^{(k)}:k\geq 1\}\subset \mathrm{Irr}(\mathbb G)$. Then $\mathbf E$ is a weak Sidon set, and hence is an interpolation set of $\mathrm M(L^\infty(\mathbb G))$ and an unconditional Sidon set. In fact, let $U_{N_k}$ be the $N_k\times N_k$ unitary matrix group and $w^{(k)}:U_{N_k}\to \mathbb M_{N_k}(\mathbb C),w\mapsto w$ be its fundamental representation. It is easy to see that $\{w^{(k)}:k\geq 1\}$ is a Sidon set of constant $1$ for $G=\prod_{k\geq 1}U_{N_k}$ (\cite[(37.5)]{hewittross1970abstract}). Thus by the universal property of $U_{N_k}^+$ we have for all finitely supported sequences $(A_k)\in \prod_k \mathbb M_{N_k}$,
$$\sum_{k\geq 1}\mathrm{Tr}(|A_k|)\leq \|\sum_{k\geq 1}\mathrm{Tr}(A_kw^{(k)})\|_{C(G)}\leq 
\|\sum_{k\geq 1}(\iota\otimes\mathrm{Tr})[(1\otimes A_k)u^{(k)}]\|_{C_u(\mathbb G)}.$$
Therefore $\mathbf E$ is a weak Sidon set. However, $\mathbf E$ is \emph{not} a Sidon set. Indeed, write $x=\sum_kx_k\in \mathrm{Pol}_{\mathbf E}(\mathbb G)$ with $x_k\in \mathrm{Pol}(U_{N_k}^+)$. Vergnioux \cite{vergnioux07decay} and Brannan \cite[Theorem 6.3]{brannan12haagerup} showed that there exists $C>0$ such that $\|x_k\|_\infty\leq C\|\hat x_k\|_2$ for all $k$. Hence $\|x\|_\infty \leq C \sum_k \|\hat x_k\|_2$ and the inequality in Definition \ref{defn: sidon} cannot hold.
\end{example}

\begin{example}
\label{ex: sutwo sidon set}
Consider the sequence $(q_n)_{n\geq 1}\subset [q,1]$ with $q\coloneqq\inf_n q_n>0$ and the associated quantum group $\mathbb G =\prod_{n\geq 1}\mathrm{SU}_{q_n}(2)$. Recall that $\mathrm{SU}_{q_n}(2)$ is coamenable and for each $n$, $C(\mathrm{SU}_{q_n}(2))$ is the universal C*-algebra generated by elements $\alpha_n$ and $\gamma_n$ such that the matrix 
$$u_n=\begin{bmatrix}
\alpha_n & -q_n\gamma_n^* \\ 
\gamma_n & \alpha_n^*
\end{bmatrix} $$
is unitary. The matrix $u_n\in \mathbb M_2(C(\mathrm{SU}_{q_n}(2)))$ defines a unitary representation of $\mathrm{SU}_{q_n}(2)$, and the matrix $Q_{u_n}$ associated to $u_n$, simply written as $Q_n$, has the eigenvalues $q_n,q_n^{-1}$. Let $d_n=\mathrm{Tr}(Q_n)=q_n+q_n^{-1}$. Then $\mathbf{E}=\{u_n:n\geq 1\} \subset\irr $ is a Sidon set for $\mathbb G$. To see this, by Proposition \ref{prop:direct sum Sidon} and Theorem \ref{sidon and u sidon}, it suffices to show that for each $n\geq 1$, the singleton $\{u_n\}$ has a uniform right unconditional Sidon constant $1+q^{-1}$, which means that for all $A\in \mathbb M_2(\mathbb C)$ and all unitaries $V\in \mathbb M_2(\mathbb C)$, we have
\begin{equation}\label{eq: sidon sutwo formula}
\|d_n(\iota\otimes\mathrm{Tr})[(1\otimes VAQ_n)u_n]\|_\infty \leq (1+q^{-1}) \|d_n(\iota\otimes \mathrm{Tr})[(1\otimes AQ_n)u_n]\|_\infty.
\end{equation}
Indeed, since the map $x\mapsto d_n^{-1}(\iota\otimes\mathrm{Tr})(x(1\otimes Q_n))$ is unital completely positive and the functional $\mathrm{Tr}$ is tracial, we may use the Cauchy-Schwarz inequality and Proposition \ref{prop:Plancherel} to get 
\begin{align*}
&\,\quad \|d_n(\iota\otimes\mathrm{Tr})[(1\otimes VAQ_n)u_n]\|_\infty^2
 = d_n^4\||d_n^{-1}(\iota\otimes\mathrm{Tr})[u_n(1\otimes VAQ_n)]|^2\|_\infty \\
&\leq d_n^4\|d_n^{-1}(\iota\otimes\mathrm{Tr})[(u_n(1\otimes VA))^*(u_n(1\otimes VA))(1\otimes Q_n)]\|_\infty \\
&= d_n^3 \mathrm{Tr}(A^*AQ_n) 
=d_n^2 \|d_n(\iota\otimes\mathrm{Tr})[(1\otimes AQ_n)u_n]\|_2^2  \\
& \leq (q_n+q_n^{-1})^2 \|d_n(\iota\otimes\mathrm{Tr})[(1\otimes AQ_n)u_n]\|_\infty^2
\end{align*}
which establishes \eqref{eq: sidon sutwo formula} if we note that $q\leq q_n\leq 1$. The order $o(q^{-1})$ of the constant obtained in the above inequality is optimal when $q\to 0$: we see that 
$$\mathrm{Tr}(|e_{21}|)=1,\quad 
\|(\iota\otimes\mathrm{Tr})[(1\otimes e_{21})u_n]\|_\infty =\|-q_n\gamma_n^*\|\leq q_n,$$
which means the Sidon constant $K$ with $\|\hat x \|_\infty \leq K \|x\|_\infty$ for $x\in \mathrm{Pol}_{\mathbf E}(\mathbb G)$ cannot be less than $q_n^{-1}$. Also as a result, if $q_n\to 0$, the subset $\mathbf E$ given above is not a Sidon set.
\end{example}
\section{Relations with $\Lambda(p)$-sets}

In this section we aim to investigate $\Lambda (p)$-sets, and in particular we will establish the relations between Sidon sets and $\Lambda (p)$-sets. 

\subsection{$\Lambda(p)$-sets and Sidon sets}
In the following we define the $\Lambda(p)$-sets for compact quantum groups, which follows from a direct quantum adaptation of that of classical $\Lambda (p)$-sets for compact groups.

\begin{defn}\label{def:lambda p}
Let $\mathbb{G}$ be a compact quantum group and $\mathbf{E}\subset\mathrm{Irr}(\mathbb{G})$
be a subset. Let $\chi_\pi=\sum_i u_{ii}^{(\pi)}$ be the character of $\pi\in\irr$. For $1< p<\infty$, we say that
$\mathbf{E}$ is a \emph{$\Lambda(p)$-set} with constant $K$ if there exists $K>0$ such that for all $x\in\mathrm{Pol}_{\mathbf{E}}(\mathbb{G})$,
\[
\|x\|_{p}\leq K \|x\|_{1},
\]
and we say that $\mathbf{E}$ is a \emph{central $\Lambda(p)$-set} with constant $K$ if there exists $K>0$ such that for all finitely supported sequences $(c_\pi)_{\pi\in\mathbf E}\subset\mathbb C$ and  $x=\sum_{\pi\in\mathbf E}c_\pi\chi_\pi\in\mathrm{Pol}_{\mathbf{E}}(\mathbb{G})$,
\[
\|x\|_{p}\leq K \|x\|_{1}.
\]
\end{defn}
\begin{rem}
Let $1< p <\infty$ and $1< p_0< p$. Notice that in order to see a subset $\mathbf E \subset \irr$ is a $\Lambda(p)$-set, it suffices to check the existence of a constant $K>0$ with
$$\|x\|_{p }\leq K \|x\|_{p_0},\quad
x \in\mathrm{Pol}_{\mathbf{E}}(\mathbb{G}).$$
This is due to the fact that $(L^p(\mathbb G))_{1\leq p\leq \infty}$ is a complex interpolation scale so that  $\|x\|_{p_0}\leq \|x\|_1^\theta \|x\|_{p}^{1-\theta}$ for some $0<\theta<1$. On the other hand, we see that any $\Lambda(p)$-set must be a $\Lambda(p')$-set for $1< p' <p<\infty$. Similar observations are also valid for central $\Lambda(p)$-sets. And as in the classical case, we will be mainly interested in the $\Lambda(p)$-sets for $2<p<\infty$.
\end{rem}
It is well known and not difficult to see that when $\mathbb G$ is a compact group $G$ or the dual quantum group of a discrete group $\Gamma$, any Sidon set $\mathbf E \subset \irr$ (or more generally, an interpolation set of $\mathrm M(L^\infty(\mathbb G))$) is a $\Lambda(p)$-set for $1< p<\infty$ (\cite{hewittross1970abstract,harcharras99nclambdap}). The same question for an arbitrary compact quantum group is however more delicate. To the best knowledge of the author, the only effort towards this direction before our work is the following property recently given by Blendek and Michali\u{c}ek, as a main result in \cite{blendekmichalicek2013sidonl1}: if $\mathbb G$ is a compact quantum group \emph{of Kac type} and if $\mathbf E \subset \irr$ is a Sidon set satisfying the \emph{Helgason-Sidon} condition, 
then there exists $K>0$ such that for all finitely supported sequences $(c_\pi)_{\pi\in\mathbf E}\subset\mathbb C$ and  $x=\sum_{\pi\in\mathbf E}c_\pi\chi_\pi\in\mathrm{Pol}_{\mathbf{E}}(\mathbb{G})$,
\[
\|x\|_{2}\leq K \|x\|_{1}.
\] 
Observe that this result requires many more restrictions on the subset $\mathbf E$ than in the classical cases while the obtained inequality is much weaker. However, compared to the classical one, its proof utilizes quite nontrivial tools such as \textquotedblleft modified Rademacher functions". In the following paragraph, we provide an alternative and more concise argument, which completely removes the non-expected restrictions in \cite{blendekmichalicek2013sidonl1} and moreover
directly establishes the nice relation between Sidon sets and $\Lambda(p)$-sets for all $1< p<\infty$.

In order to characterize the $\Lambda(p)$-sets, let us consider the following notions of interpolation sets of bounded multipliers on $L^p(\mathbb G)$, which also generalize Definition \ref{def:interpolation set infty}.
\begin{defn}
	\label{def:inter sets p}
	Let $\mathbb G$ be a compact quantum group and let $1\leq p\leq \infty$. We say that $\mathbf E \subset \irr $ is an \emph{interpolation set of $\mathrm M(L^p({\mathbb{G}}))$ (resp., of $\mathrm M_L(L^p({\mathbb{G}}))$, of $\mathrm M_R(L^\infty({\mathbb{G}}))$)} with constant $K$ if for any $a\in\ell^{\infty}(\mathbf{E})$, there exists a bounded multiplier
	$\tilde{a}\in \mathrm M(L^p({\mathbb{G}}))$ (resp., $\tilde{a}\in \mathrm M_L(L^p({\mathbb{G}}))$, $\tilde{a}\in \mathrm M_R(L^p({\mathbb{G}}))$) with $\|\tilde a \|_{\mathrm M(L^p({\mathbb{G}}))} \leq K\|a\|_\infty$ (resp.,  $\|m_{\tilde{a}}^L\|_{B(L^p({\mathbb{G}}))} \leq K\|a\|_\infty$, $\|m_{\tilde{a}}^R\|_{B(L^p({\mathbb{G}}))}\leq K\|a\|_\infty$) such that $\tilde{a}_{\pi}=a_{\pi}$
	for all $\pi\in\mathbf{E}$.
\end{defn}

\begin{rem}
	Let $1\leq p\leq \infty$ and $\mathbf E \subset \irr $. We remark that, if for any $a\in\ell^{\infty}(\mathbf{E})$, there exists a bounded multiplier
	$\tilde{a}\in \mathrm M(L^p({\mathbb{G}}))$, then automatically there exists a constant $K>0$ with $\|\tilde a \|_{\mathrm M(L^p({\mathbb{G}}))} \leq K\|a\|_\infty$, and $\mathbf E$ is an interpolation set of $\mathrm M(L^p({\mathbb{G}}))$. In fact, if for any
	$a\in\ell^{\infty}
	(\mathbf{E})$, there exists a
	bounded multiplier
	$\tilde{a}\in \mathrm
	M(L^p({\mathbb{G}}))$
	with
	$\tilde{a}_{\pi}=a_{\pi}$ for
	all $\pi\in\mathbf{E}$, then
	by Proposition \ref{prop:multiplier bdd quantum}, we have
	$$\|Q^{\frac{1}{4}-\frac{1}
		{2p}}aQ^{ -\frac{1}
		{4}+\frac{1}{2p}}\|_\infty
	<\infty.$$ In particular, we
	choose an appropriate
	basis of $\oplus_\pi H_\pi$
	so that $Q_\pi$ is diagonal
	under this basis for $\pi\in
	\irr$, and take
	$a_\pi=e_{ij}$ with $1\leq
	i,j\leq n_\pi$, then the
	above inequality yields that
	$$\sup_{\pi\in \mathbf
		E}\|Q_\pi\|<\infty,\quad
	\sup_{\pi\in \mathbf E}\|Q_
	\pi^{-1}\|<\infty.$$ Also by
	Proposition \ref{prop:multiplier bdd quantum}, it is easy to
	see that $\mathrm
	M(L^p({\mathbb{G}}))$ is a
	Banach subspace of
	$B(L^p({\mathbb{G}}))$. So
	by the open mapping
	theorem and Proposition
	\ref{prop:multiplier bdd quantum}, we may always find a
	constant $K>0$ such that
	the inequality $$\|\tilde a
	\|_{\mathrm
		M(L^p({\mathbb{G}}))} \leq
	K\|a\|_\infty$$ is
	automatically satisfied.
	Thus $\mathbf{E}$ is
	automatically an interpolation set of $\mathrm M(L^p({\mathbb{G}}))$.  But we do not know whether the similar observation can be made for interpolation sets of $\mathrm M_L(L^p({\mathbb{G}}))$ and that of $\mathrm M_R(L^p({\mathbb{G}}))$.
\end{rem}
These kinds of lacunarities have some special  restrictive properties in the non-Kac case. The following result will be of use later.

\begin{prop}
	\label{prop: unif bdd of mod gp for sidon}
	Let $\mathbb G$ be a compact quantum group and $1\leq p\leq \infty$. Assume that $\mathbf E \subset \irr$ satisfies one of the following four conditions:
	
	\emph{(1)} $p\neq \infty$ and $\mathbf E \subset \irr$ is an interpolation set of $\mathrm M_L(L^p({\mathbb{G}}))$ with constant $K$;
	
	\emph{(2)} $p\neq \infty$ and $\mathbf E \subset \irr$ is an interpolation set of $\mathrm M_R(L^p({\mathbb{G}}))$ with constant $K$;
	
	\emph{(3)} $\mathbf E \subset \irr$ is an interpolation set of $\mathrm M(L^p({\mathbb{G}}))$ with constant $K$;
	
	\emph{(4)} there exist constants $K>0$ and $\theta,\theta'\in\mathbb R$ with $\theta+\theta'-\frac{1}{p}-\frac{1}{2}\neq 0$ such that for all $a\in c_c(\mathbf E)$,
	$$\|m_a^L x\|_p\leq K\|Q^{\theta}aQ^{-\theta}\|_\infty \|x\|_p,\quad 
	\|m_a^R x\|_p\leq K\|Q^{\theta'}aQ^{-\theta'}\|_\infty \|x\|_p,\quad
	x\in \mathrm{Pol}_{\mathbf E}(\mathbb G).$$ 
	Then there exist constants $K_1,K_2>0$ such that for all $\pi\in\mathbf E$ and $1\leq i,j,k,l \leq n_\pi$, we have
	\begin{equation}\label{eq:unif bdd mod ij kl}
	\|u_{ij}^{(\pi)}\|_p \leq K_1
	\|u_{kl}^{(\pi)}\|_p,
	\end{equation}
	and moreover 
	\begin{equation}\label{eq:unif bdd mod sidon}
\max\{\|Q_\pi\|,\|Q_\pi^{-1}\|\}\leq K_2.
	\end{equation}
\end{prop}
\begin{proof}
	In the proof we always choose an appropriate basis of $\oplus_\pi H_\pi$ so that $Q_\pi$ is diagonal under this basis for $\pi\in\irr$.
	We first prove \eqref{eq:unif bdd mod ij kl} and \eqref{eq:unif bdd mod sidon} under the assumption (4).
	
	Assume that $K>0$ and $\theta,\theta'\in\mathbb R$ with $\theta+\theta'-\frac{1}{p}-\frac{1}{2}\neq 0$ such that for all $a\in c_c(\mathbf E)$,
	$$\|m_a^L x\|_p\leq K\|Q^{\theta}aQ^{-\theta}\|_\infty \|x\|_p,\quad 
	\|m_a^R x\|_p\leq K\|Q^{\theta'}aQ^{-\theta'}\|_\infty \|x\|_p,\quad
	x\in \mathrm{Pol}_{\mathbf E}(\mathbb G).$$ 
	Let $\pi\in\mathbf E$ and $1\leq i,j,k,l \leq n_\pi$. Then in particular we have 
	$$\|m_{e_{ki}}^L x\|_p\leq K (Q_\pi)_{kk}^{\theta} (Q_\pi)_{ii}^{-\theta} \|x\|_p,\quad 
	\|m_{e_{jl}}^R x\|_p\leq K (Q_\pi)_{jj}^{\theta'} (Q_\pi)_{ll}^{-\theta'}  \|x\|_p,\quad
	x\in \mathrm{Pol}_{\mathbf E}(\mathbb G).$$
	So by Proposition \ref{prop:Plancherel} we have
	\begin{align}\label{eq:mod bdd pf}
	\|u_{ij}^{(\pi)}\|_p
	& = \|(\iota\otimes \mathrm{Tr})[(1\otimes e_{ji})u^{(\pi)}]\|_p 
	= \|m_{e_{jl}}^R m_{e_{ki}}^L (\iota\otimes \mathrm{Tr})[(1\otimes e_{lk})u^{(\pi)}]\|_p\\
	& \leq K^2 (Q_\pi)_{kk}^{\theta} (Q_\pi)_{ii}^{-\theta}  (Q_\pi)_{jj}^{\theta'} (Q_\pi)_{ll}^{-\theta'}  \|(\iota\otimes \mathrm{Tr})[(1\otimes e_{lk})u^{(\pi)}]\|_p \nonumber \\
	& = K^2 (Q_\pi)_{kk}^{\theta} (Q_\pi)_{ii}^{-\theta}  (Q_\pi)_{jj}^{\theta'} (Q_\pi)_{ll}^{-\theta'}  \|u_{kl}^{(\pi)}\|_p. \nonumber
	\end{align}  
	In particular we get for any $1\leq i,j\leq n_\pi$,
	\begin{equation}\label{eq:unif bdd pf ij ji}
	K^{-2} (Q_\pi)_{jj}^{\theta+\theta'} (Q_\pi)_{ii}^{-\theta-\theta'} \|u_{ji}^{(\pi)}\|_p
	\leq 
	\|u_{ij}^{(\pi)}\|_p
	\leq K^2 (Q_\pi)_{jj}^{\theta+\theta'} (Q_\pi)_{ii}^{-\theta-\theta'}  \|u_{ji}^{(\pi)}\|_p.
	\end{equation}
	Recall that $Q_\pi$ is chosen diagonal, and note that by Lemma \ref{lem: R isometry Lp} and the formula \eqref{modular group on cqg},
	$$\|u_{ij}^{(\pi)}\|_p=
	\|\sigma_{-\mathrm i/p}((u_{ij}^{(\pi)})^*)\|_p=(Q_\pi)_{ii}^{-\frac{1}{p}}(Q_\pi)_{jj}^{-\frac{1}{p}}\|(u_{ij}^{(\pi)})^*\|_p
	=(Q_\pi)_{ii}^{-\frac{1}{p}}(Q_\pi)_{jj}^{-\frac{1}{p}}\|S(u_{ji}^{(\pi)})\|_p.$$
	But using the polar decomposition of $S$ in \eqref{eq: polar decomp of s} and \eqref{eq: scaling group on cqg} and recalling that $Q_\pi$ is chosen diagonal, we have
	$$\|S(u_{ji}^{(\pi)})\|_p
	=\|R(\tau_{-\frac{\mathrm i}{2}}(u_{ji}^{(\pi)}))\|_p 
	=\|\tau_{-\frac{\mathrm i}{2}}(u_{ji}^{(\pi)})\|_p 
	=(Q_\pi)_{jj}^{\frac{1}{2}} (Q_\pi)_{ii}^{-\frac{1}{2}}\|u_{ji}^{(\pi)}\|_p.$$
	The above three inequalities yield that
	$$(Q_\pi)_{ii}^{\theta+\theta'-\frac{1}{p}-\frac{1}{2}} (Q_\pi)_{jj}^{-\theta-\theta'-\frac{1}{p}+\frac{1}{2}} \leq K^2,\quad 
	(Q_\pi)_{ii}^{ -\theta-\theta'-\frac{1}{p}+\frac{1}{2}} (Q_\pi)_{jj}^{\theta+\theta'-\frac{1}{p}-\frac{1}{2}} \leq K^2$$	
	Note that $i$ and $j$ are arbitrarily chosen and that $\|Q_\pi\|\geq 1,\|Q_\pi^{-1}\|\geq 1$, so the above inequalities yield
	$$\max\{\|Q_\pi\|,\|Q_\pi^{-1}\|\}\leq K^{2/|\theta+\theta'-\frac{1}{p}-\frac{1}{2}|}.$$
	Combining this with \eqref{eq:mod bdd pf}, we also get
	$$\|u_{ij}^{(\pi)}\|_p \leq K^{2+2/|\theta+\theta'-\frac{1}{p}-\frac{1}{2}|}
	\|u_{kl}^{(\pi)}\|_p,$$
	as desired.
	
	Now we assume that (1) holds, that is, $\mathbf E \subset \irr$ is an interpolation set of $\mathrm M_L(L^p({\mathbb{G}}))$. 
	Let  $a\in c_c(\mathbf E)$. Then $a$ extends to bounded left multipliers on $L^p(\mathbb G)$, and hence by Lemma \ref{lem: left to right multipliers} we have 
	$$\|m_a^L x\|_p\leq K\|a\|_\infty \|x\|_p,\quad 
	\|m_a^R x\|_p\leq K\|Q^{\frac{1}{2}}aQ^{-\frac{1}{2}}\|_\infty \|x\|_p,\quad
	x\in \mathrm{Pol}_{\mathbf E}(\mathbb G).$$ 
	Taking $\theta=0,\theta'=1/2$ in (4), then we obtain the desired inequalities \eqref{eq:unif bdd mod ij kl} and \eqref{eq:unif bdd mod sidon}.
	And the proof under the assumption (2) follows from a similar argument.
	
	The assumption (3) can be viewed as a particular case of (4) by taking $\theta=1/p,\theta'=0$. So we establish the proposition.
\end{proof}

Combining this proposition with Lemma \ref{lem: left to right multipliers}, we deduce the following observation.
\begin{lem}\label{lem:inter left eq right}
	Let $\mathbf E \subset \irr$ be a subset. Let $1\leq p<\infty$. Then $\mathbf E$ is an interpolation set of $\mathrm M_L(L^p({\mathbb{G}}))$ if and only if it is an interpolation set of $\mathrm M_R(L^p({\mathbb{G}}))$.
\end{lem}

Now we are able to characterize the $\Lambda(p)$-sets via the interpolation sets of bounded multipliers for $2<p<\infty$.
\begin{thm}
	\label{thm:lambda inter}
	Let $\mathbf{E}\subset\mathrm{Irr}(\mathbb{G})$ be a subset. Assume $2<p<\infty$. The following assertions are equivalent:
	
	\emph{(1)} $\mathbf{E}$ is a $\Lambda(p)$-set;
	
	\emph{(2)} there exists a constant $K>0$ such that for all $a\in \ell^\infty(\mathbf E)$,
	$$\|m_a^L x\|_p\leq K\|a\|_\infty \|x\|_p,\quad 
	\|m_a^R x\|_p\leq K\|a\|_\infty \|x\|_p,\quad
	x\in \mathrm{Pol}_{\mathbf E}(\mathbb G)\text{ ;}$$
	
	\emph{(3)} $\mathbf E$ is an interpolation set of $\mathrm M_L(L^p({\mathbb{G}}))$;
	
	\emph{(4)} $\mathbf E$ is an interpolation set of $\mathrm M_R(L^p({\mathbb{G}}))$;
	
	\emph{(5)} $\mathbf E$ is an interpolation set of $\mathrm M(L^p({\mathbb{G}}))$.	\smallskip\\
	If additionally the subset $\mathbf{E}$ is symmetric in the sense that $\pi\in \mathbf{E}$ if and only if $\bar \pi\in \mathbf{E}$, then the above assertions are also equivalent to:
	
	\emph{(6)} there exists a constant $K>0$ such that for all $a\in \ell^\infty(\mathbf E)$,
	$$\|m_a^L x\|_p\leq K\|a\|_\infty \|x\|_p,\quad
	x\in \mathrm{Pol}_{\mathbf E}(\mathbb G)\text{ ;}$$
	
	\emph{(7)} there exists a constant $K>0$ such that for all $a\in \ell^\infty(\mathbf E)$,
	$$
	\|m_a^R x\|_p\leq K\|a\|_\infty \|x\|_p,\quad
	x\in \mathrm{Pol}_{\mathbf E}(\mathbb G).$$
\end{thm}
\begin{proof}
	Note that the equivalence $(3)\Leftrightarrow (4) $ has been already given in the previous lemma and the implication $(5)\Rightarrow (4)$ is trivial. Since (3) and (4) are equivalent, the implication $(3)\Rightarrow (2)$ is also obvious.
	
	$(1)\Rightarrow (5)$. Assume that $\mathbf{E}$ is a $\Lambda(p)$-set with constant $K$. Then together with Lemma \ref{lem:multiplier on ltwo}, we see that for all $a\in \ell ^\infty (\mathbf E)$, 
\begin{equation}\label{eq:lambda inter pf mr}
	\|m_a^Rx\|_p \leq K \|m_a^Rx\|_2 \leq K\|a\|_\infty\|x\|_2 \leq K\|a\|_\infty\|x\|_p,\quad x\in \mathrm{Pol}(\mathbb G),
\end{equation}
	which yields that $m_a^R$ extends to a bounded operator on $L^p(\mathbb G)$. So $\mathbf E$ is an interpolation set of $\mathrm M_R(L^p({\mathbb{G}}))$, and in particular by Proposition \ref{prop: unif bdd of mod gp for sidon}, we may find a constant $K'$ such that 
	$$\max\{\|Q_\pi\|,\|Q_\pi^{-1}\|\}\leq K',\quad \pi\in \mathbf E.$$
	Hence by Lemma \ref{lem: left to right multipliers}, \eqref{eq:lambda inter pf mr} also yields that for all $a\in \ell ^\infty (\mathbf E)$, 
	$$\|m_a^Lx\|_p \leq K \|m_a^Lx\|_2 \leq K K'\|a\|_\infty\|x\|_2 \leq K K'\|a\|_\infty\|x\|_p,\quad x\in \mathrm{Pol}(\mathbb G).
	$$
	As a result, both maps $m_a^L$ and $m_a^R$ for $a\in \ell ^\infty (\mathbf E)$ extend to bounded operators on $L^p(\mathbb G)$, which means that $\mathbf E$ is an interpolation set of $\mathrm M(L^p({\mathbb{G}}))$.
	
$(2)\Rightarrow (1)$. Assume that (2) holds and  we take $K>0$ to be the constant satisfying
\begin{equation}\label{eq: p bdd multiplier for m sidon}
\|m_a^Lx\|_p \leq K\|a\|_\infty \|x\|_p,\quad
\|m_a^Rx\|_p \leq K\|a\|_\infty \|x\|_p,\quad 
x\in \mathrm{Pol}_{\mathbf E}(\mathbb G),\ a\in\ell^\infty(\mathbf E).
\end{equation}
For each $\pi\in\mathrm{Irr}(\mathbb G)$, since the operator $Q_\pi$ is positive, we may fix a basis in $H_\pi$ such that the associated matrix of $Q_\pi$ is diagonal, and denote by $u^{(\pi)}\in \mathbb M_{n_\pi}(C(\mathbb G))$ the representation matrix under this basis. 

Let us first  show that for  $\pi\in\mathrm{Irr}(\mathbb G)$ and $1\leq i\leq n_\pi$, we have  some constant $C_0>0$ such that
\begin{equation}\label{eq: p estimate for u11}
\|u_{ii}^{(\pi)}\|_p^2\leq C_0 p \|u_{ii}^{(\pi)}\|_2^2 = C_0 p d_\pi^{-1}(Q_\pi^{-1})_{ii}.
\end{equation} 
Take the $n_\pi\times n_\pi$ matrices
\begin{equation*}
a=\begin{bmatrix}0& 0 & \cdots & 0\\
\vdots& \vdots & \ddots & \vdots\\
0& 0 & \cdots & 0\\
(Q_\pi^{-\frac{1}{2}})_{11} & (Q_\pi^{-\frac{1}{2}})_{22} & \cdots & (Q_\pi^{-\frac{1}{2}})_{n_\pi n_\pi}\\
0& 0 & \cdots & 0\\
\vdots& \vdots & \ddots & \vdots\\
0& 0 & \cdots & 0
\end{bmatrix},\ 
b=\begin{bmatrix}0 & \cdots & 0 & (Q_\pi^{-\frac{1}{2}})_{11} & 0 & \cdots & 0 \\
0 & \cdots & 0 & (Q_\pi^{-\frac{1}{2}})_{22}& 0 & \cdots & 0\\
\vdots & \ddots & \vdots & \vdots& \vdots & \ddots & \vdots\\
0 & \cdots & 0 & (Q_\pi^{-\frac{1}{2}})_{n_\pi n_\pi}& 0 & \cdots & 0
\end{bmatrix}
\end{equation*}
where the nonzero coefficients are in the $i$-th row of $a$, and in the $i$-th column of $b$.
Note that $\|a\|_\infty =d_\pi ^{1/2}$. Let $y=\sum_{j=1}^{n_\pi}(Q_\pi^{-\frac{1}{2}})_{jj}u_{ij}^{(\pi)}$. Then by Proposition \ref{prop:Plancherel},
\begin{equation}\label{eq:lambda inter pf y}
y=(\iota\otimes \mathrm{Tr})[(1\otimes b)u^{(\pi)}],\quad u_{ii}^{(\pi)}=d_\pi^{-1}m_a^Ry.
\end{equation}
According to \eqref{eq: p bdd multiplier for m sidon} we have
\begin{equation}
\label{eq: p estimate for y}
\|u_{ii}^{(\pi)}\|_p\leq K d_\pi^{-\frac{1}{2}}\|y\|_p.
\end{equation}
Let $\varepsilon_j^{(\pi)} (1\leq j\leq n_\pi)$ be a sequence of independent Rademacher variables on a probability space $(\Omega,P)$ and write the $n_\pi \times n_\pi$ matrix
$$\mathfrak{e}=\begin{bmatrix} \varepsilon_1^{(\pi)} & 0 & \cdots & 0 \\
0& \varepsilon_2^{(\pi)} & \cdots & 0\\
\vdots & \vdots & \ddots & \vdots\\
0& 0 & \cdots & \varepsilon_{n_\pi}^{(\pi)}
\end{bmatrix}.$$
Again by \eqref{eq: p bdd multiplier for m sidon} and by Theorem \ref{khintchine} we have with some universal constant $C>0$,
\begin{align}\label{khintchine estimate for y}
\|y\|_p
& = \int_{\Omega} \|m_{\mathfrak e \mathfrak e}^Ry\|_p dP
\leq K \int_{\Omega} \|m_{\mathfrak e}^Ry\|_p dP
= K \int_{\Omega} \|\sum_{j=1}^{n_\pi}(Q_\pi^{-\frac{1}{2}})_{jj}\varepsilon_j^{(\pi)}u_{ij}^{(\pi)}\|_p dP \\
& \leq CK\sqrt p \|((Q_\pi^{-\frac{1}{2}})_{jj}u_{ij}^{(\pi)}D^{\frac{1}{p}})\|_{CR_p[L^p(\mathbb G)]}, \nonumber
\end{align}
where 
\begin{align*}
& \quad\,\|((Q_\pi^{-\frac{1}{2}})_{jj}u_{ij}^{(\pi)}D^{\frac{1}{p}})\|_{CR_p[L^p(\mathbb G)]}\\
&= \max \{\|(D^{\frac{1}{p}}\sum_{j=1}^{n_\pi}(Q_\pi^{-1})_{jj}(u_{ij}^{(\pi)})^*u_{ij}^{(\pi)}D^{\frac{1}{p}})^{\frac{1}{2}}\|_{p,\mathsf H},
\|(\sum_{j=1}^{n_\pi}(Q_\pi^{-1})_{jj}u_{ij}^{(\pi)}D^{\frac{2}{p}}(u_{ij}^{(\pi)})^*)^{\frac{1}{2}}\|_{p,\mathsf H}
\}.
\end{align*}
Recall \eqref{antipode2} and that the matrix $Q_\pi$ under the chosen basis is diagonal. We have 
\begin{align}
\label{eq: p column norm estimate}
\sum_{j=1}^{n_\pi}(Q_\pi^{-1})_{jj}(u_{ij}^{(\pi)})^*u_{ij}^{(\pi)}
& =\sum_{j=1}^{n_\pi}(Q_\pi^{-1})_{jj}S(u_{ji}^{(\pi)})u_{ij}^{(\pi)}
=\sum_{j=1}^{n_\pi}(Q_\pi^{-1})_{jj}(S^{-1}\circ S^2)(u_{ji}^{(\pi)})u_{ij}^{(\pi)}\\\nonumber
&
=\sum_{j=1}^{n_\pi}(Q_\pi^{-1})_{jj}S^{-1}((Q_\pi)_{jj}u_{ji}^{(\pi)}(Q_\pi^{-1})_{ii})u_{ij}^{(\pi)}\\\nonumber
&
=S^{-1}\Big(\sum_{j=1}^{n_\pi}(Q_\pi^{-1})_{ii}(u_{ji}^{(\pi)})^*u_{ji}^{(\pi)}\Big)= (Q_\pi^{-1})_{ii} S^{-1}(1)= (Q_\pi^{-1})_{ii}
\nonumber
\end{align}
where the last line above follows from the fact that $u^{(\pi)}$ is unitary. 
Recall Proposition \ref{density op haagerup Lp} (1) and \eqref{modular group on cqg}. We then have 
\begin{align*}
u_{ij}^{(\pi)}D^{\frac{1}{p}}
=D^{\frac{1}{p}}\sigma_{\frac{\mathrm{i}}{p}}(u_{ij}^{(\pi)})
=D^{\frac{1}{p}}(Q_\pi^{-\frac{1}{p}})_{ii}u_{ij}^{(\pi)}(Q_\pi^{-\frac{1}{p}})_{jj},
\end{align*}
and hence 
$$D^{\frac{1}{p}}(u_{ij}^{(\pi)})^*
=(Q_\pi^{-\frac{1}{p}})_{ii}(Q_\pi^{-\frac{1}{p}})_{jj}(u_{ij}^{(\pi)})^*D^{\frac{1}{p}}.$$
Therefore
$$\sum_{j=1}^{n_\pi}(Q_\pi^{-1})_{jj}u_{ij}^{(\pi)}D^{\frac{2}{p}}(u_{ij}^{(\pi)})^*
=D^{\frac{1}{p}}\sum_{j=1}^{n_\pi}
(Q_\pi^{-\frac{2}{p}})_{ii}(Q_\pi^{-\frac{2}{p}-1})_{jj}u_{ij}^{(\pi)}(u_{ij}^{(\pi)})^*
D^{\frac{1}{p}}.$$
Recall that by Proposition \ref{prop: unif bdd of mod gp for sidon} we have $K_1= \sup_{\pi\in \mathbf{E}}\|Q_\pi\|<\infty,
K_2=\sup_{\pi\in \mathbf{E}}\|Q_\pi^{-1}\|<\infty$, so the above expression can be estimated as 
$$\sum_{j=1}^{n_\pi}(Q_\pi^{-1})_{jj}u_{ij}^{(\pi)}D^{\frac{2}{p}}(u_{ij}^{(\pi)})^*
\leq K_2^{\frac{4}{p}+1}K_1 (Q_\pi^{-1})_{ii}  D^{\frac{1}{p}}\sum_{j=1}^{n_\pi}
u_{ij}^{(\pi)}(u_{ij}^{(\pi)})^*
D^{\frac{1}{p}}=K_2^{\frac{4}{p}+1}K_1 (Q_\pi^{-1})_{ii}  D^{\frac{2}{p}}.$$
Together with \eqref{eq: p column norm estimate} we deduce 
$$\|((Q_\pi^{-\frac{1}{2}})_{jj}u_{ij}^{(\pi)})\|_{CR_p[L^p(\mathbb G)]}^2
\leq K_2^{\frac{4}{p}+1}K_1 (Q_\pi^{-1})_{ii} .$$
Back to \eqref{khintchine estimate for y} we get $\|y\|_p^2\leq C^2K^2K_2^{\frac{4}{p}+1}K_1  p (Q_\pi^{-1})_{ii}$ and by \eqref{eq: p estimate for y} we obtain
$$\|u_{ii}^{(\pi)}\|_p^2\leq C^2K^4K_2^{\frac{4}{p}+1}K_1 p d_\pi^{-1} (Q_\pi^{-1})_{ii},$$
whence \eqref{eq: p estimate for u11}, as desired.

Now take $x\in\mathrm{Pol}_{\mathbf{E}}(\mathbb{G})$. Choose by polar decomposition and diagonalization two sequences  of unitary matrices $v=(v_\pi)_{\pi\in \mathbf E},v'=(v_\pi')_{\pi\in \mathbf E}\in\ell^\infty (\mathbf E)$ such that $c_\pi\coloneqq v_\pi v_\pi'\hat x(\pi)Q_\pi v_\pi^*$ is a diagonal matrix for each $\pi$.
Write 
$$c_\pi=\begin{bmatrix} c_{11}^{(\pi)} & 0 & \cdots & 0 \\
0& c_{22}^{(\pi)} & \cdots & 0\\
\vdots & \ddots & \vdots & \vdots\\
0& 0 & \cdots & c_{n_\pi n_\pi}^{(\pi)}
\end{bmatrix},\quad c_{11}^{(\pi)},\ldots,c_{n_\pi n_\pi}^{(\pi)}\in\mathbb C$$ 
and denote $\mathfrak e=(\mathfrak{e}_\pi)_{\pi\in \mathbf E}$.
By \eqref{eq: p bdd multiplier for m sidon}-\eqref{eq: p estimate for u11} and Theorem \ref{khintchine} we have with some universal constant $C>0$,
\begin{align*}
\|x\|_{p}&=\|m^R_{v'^*v^*}m^L_{v}
\Big(\sum_{\pi\in\mathbf{E}}d_{\pi}(\iota\otimes\mathrm{Tr})[(1\otimes c_\pi)u^{(\pi)}]\Big)\|_p 
\leq K^2\|\sum_{\pi\in\mathbf{E}}d_{\pi}(\iota\otimes\mathrm{Tr})[(1\otimes c_\pi)u^{(\pi)}]\|_p\\
&\leq K^3\int_{\Omega}\|m^R_{\mathfrak e}\Big(\sum_{\pi\in\mathbf{E}}d_{\pi}(\iota\otimes\mathrm{Tr})[(1\otimes c_\pi)u^{(\pi)}]\Big)\|_p dP
=K^3\int_{\Omega}\|\sum_{\pi\in\mathbf{E}}d_{\pi}\sum_{j=1}^{n_\pi}\varepsilon_j^{(\pi)}c_{jj}^{(\pi)}u_{jj}^{(\pi)}\|_{p} dP\\
& \leq  K^3 C\sqrt p \Big(\sum_{\pi\in\mathbf E}\sum_{j=1}^{n_\pi}d_\pi^2 (c_{jj}^{(\pi)})^2\|u_{jj}^{(\pi)}\|_p^2\Big)^{\frac{1}{2}} 
\leq C_0^{\frac{1}{2}}C p K^3 \Big(\sum_{\pi\in\mathbf E}\sum_{j=1}^{n_\pi}d_\pi^2 (c_{jj}^{(\pi)})^2\|u_{jj}^{(\pi)}\|_2^2\Big)^{\frac{1}{2}} \\
& = C_0^{\frac{1}{2}}C p K^3 \|\sum_{\pi\in\mathbf E}\sum_{j=1}^{n_\pi}d_\pi c_{jj}^{(\pi)}u_{jj}^{(\pi)}\|_2 
\end{align*}
where the last equality follows from the fact that $\{u_{ii}^{(\pi)}:1\leq i\leq n_\pi\}$ are orthogonal with respect to $h$ according to \eqref{eq:haar state def}. Using Lemma \ref{lem:multiplier on ltwo} we have
$$\|\sum_{\pi\in\mathbf E}\sum_{j=1}^{n_\pi}d_\pi c_{jj}^{(\pi)}u_{jj}^{(\pi)}\|_2
=\|m^R_{vv'}m^L_{v^*}x\|_2\leq K_1 K_2\|x\|_2.$$
Hence the above two inequalities together yield
$$\|x\|_p\leq C_0^{\frac{1}{2}}CpK^3 K_1 K_2\|x\|_2.$$
So we prove that $\mathbf{E}$ is a $\Lambda(p)$-set.

Finally let us discuss (6) and (7). In fact, if $\mathbf E$ is symmetric, we note that $X=\mathrm{Pol}_{\mathbf E}(\mathbb G)$ satisfies the assumption of Lemma \ref{lem: left to right multipliers}. So if (6) holds, by Lemma \ref{lem: left to right multipliers} we have for all $a\in \ell^\infty(\mathbf E)$,
\begin{equation}\label{eq:lambda inter pf sym}
\|m_a^R x\|_p\leq K\|Q^{1/2}a Q^{-1/2}\|_\infty \|x\|_p,\quad
x\in \mathrm{Pol}_{\mathbf E}(\mathbb G). 
\end{equation}
and hence using Proposition \ref{prop: unif bdd of mod gp for sidon} and the assumption in (6), we obtain the inequality
$$\max\{\|Q_\pi\|,\|Q_\pi^{-1}\| \}\leq K',\quad \pi\in \mathbf E$$
for some constant $K'>0$. Thus coming back to \eqref{eq:lambda inter pf sym} again, we see that the condition (7) holds as well. Conversely, we may also show in the same way that (7) implies (6). So in other words the assertions (6) and (7) are equivalent, and in particular they are equivalent to the assertion (2).
%On the other hand, if the compact quantum group $\mathbb G$ is of Kac type, then we have $Q_\pi=\mathrm{Id}_{\pi}$ in the above proof of $(2)\Rightarrow (1)$ and we do not need to use Proposition \ref{prop: unif bdd of mod gp for sidon}. Then one may easily check that a repetition of the argument in $(2)\Rightarrow (1)$ gives the implication $(7)\Rightarrow (1)$. Also a slight adaptation gives the proof of $(6)\Rightarrow (1)$: we take $y=\sum_{j=1}^{n_\pi}u_{ji}^{(\pi)}$ and replace the equalities \eqref{eq:lambda inter pf y} by  
%$$
%y=(\iota\otimes \mathrm{Tr})[(1\otimes a)u^{(\pi)}],\quad u_{ii}^{(\pi)}=d_\pi^{-1}m_b^Ly.$$
%Then following the same idea as in $(2)\Rightarrow (1)$ we may easily prove the implication $(6)\Rightarrow (1)$.
\end{proof}

As mentioned previously, we conclude in the following that any Sidon set for a compact quantum group $\mathbb G$ is a $\Lambda(p)$-set for all $1<p<\infty$, thereby considerably improving the earlier work \cite{blendekmichalicek2013sidonl1}. In fact as in the classical case, we have the stronger result below. Recall that in Theorem \ref{thm:sidon_l1_interpolation state} we have some generalized notions of the Sidon set. 
\begin{thm}\label{sidon lambda}
Assume that $\mathbf{E}\subset\mathrm{Irr}(\mathbb{G})$
is an interpolation set of $\mathrm M(L^\infty(\mathbb G))$.
Then $\mathbf{E}$ is a $\Lambda(p)$-set for all $1< p<\infty$.
\end{thm}
\begin{proof}
Assume that $\mathbf E$ is an interpolation set of $\mathrm M(L^\infty(\mathbb G))$ with constant $K$. 
Let $a\in\ell^\infty(\mathbf E)$. Then $a$ extends to a bounded multiplier $\tilde a\in\mathrm M (L^\infty(\mathbb G))$, and by Proposition \ref{prop:multiplier bdd quantum},
$$\|Q^{1/4}\tilde aQ^{-1/4}\|_\infty\leq \| \tilde a \|_{\mathrm M(L^\infty(\mathbb G))} \leq K\|a\|_\infty.$$ 
Consider
$$T_z=m^L_{a^{(z)}}, \quad
a^{(z)}_\pi =Q_\pi^{-\frac{z}{2}+\frac{1}{4}}\tilde{a}_\pi Q_\pi^{\frac{z}{2}-\frac{1}{4}},\quad
\pi\in\irr,\ 
z=t_1+\mathrm i t_2,\  0\leq t_1\leq 1, t_2\in\mathbb R.$$
Observe that by Lemma \ref{lem:multiplier on ltwo}, the operators $m^L_{a^{(z)}}$ for $z=1+\mathrm i t$ are bounded on $L^2(\mathbb G)$  with norm $$\|m^L_{a^{(z)}}\|_{B(L^2(\mathbb G))}=\|Q^{1/4}\tilde aQ^{-1/4}\|_\infty\leq K\|a\|_\infty.$$
So by the Stein interpolation theorem (see e.g. \cite[Theorem 2.7]{lunardi09interpolationbook}) for $2\leq p <\infty$ we have $\|T_{2/p}x\|_p\leq K_0\|a\|_\infty\|x\|_p$ for $x\in L^p(\mathbb G)$. Let $K_1=\sup_{\pi\in \mathbf{E}}\|Q_\pi\|$, $ K_2=\sup_{\pi\in \mathbf{E}}\|Q_\pi^{-1}\|$. Both $K_1$ and $K_2$ are finite by Proposition \ref{prop: unif bdd of mod gp for sidon}. 
Rewrite $b_\pi=Q_\pi^{-\frac{1}{p}+\frac{1}{4}}a_\pi Q_\pi^{\frac{1}{p}-\frac{1}{4}}$ and the above argument yields that for $b\in\ell^\infty(\mathbf E)$ and $x\in \mathrm{Pol}_{\mathbf E}(\mathbb G)$, 
\begin{equation*}\label{eq: p bdd left multiplier for m sidon}
\|m_b^Lx\|_p \leq K \|a\|_\infty\|x\|_p=K_0 
\sup_{\pi\in \mathbf E}\|Q_\pi^{\frac{1}{p}-\frac{1}{4}}b_\pi Q_\pi^{-\frac{1}{p}+\frac{1}{4}}\|_\infty
\|x\|_p
\leq K_0K_1^{\frac{1}{p}-\frac{1}{4}}K_2^{\frac{1}{p}-\frac{1}{4}}\|b\|_\infty \|x\|_p.
\end{equation*}
Similar argument also applies to $m_a^R$ (where the above $T_z$ should be simply taken as $m^R_{Q^{1/4}\tilde aQ^{-1/4}}$ identically). Now we may take $K>0$ to be the constant satisfying
\begin{equation*}
\|m_a^Lx\|_p \leq K\|a\|_\infty \|x\|_p,\quad
\|m_a^Rx\|_p \leq K\|a\|_\infty \|x\|_p,\quad 
x\in \mathrm{Pol}_{\mathbf E}(\mathbb G),\ a\in\ell^\infty(\mathbf E).
\end{equation*}
Thus $\mathbf{E}$ is a $\Lambda(p)$-set for all $1< p<\infty$ according to Theorem \ref{thm:lambda inter}.
\end{proof}
By Theorem \ref{sidon and u sidon}, we have the following corollary, as desired.
\begin{cor}\label{cor:cor sidon lambda}
If $\mathbf{E}\subset\mathrm{Irr}(\mathbb{G})$
is a Sidon set, then $\mathbf{E}$ is a $\Lambda(p)$-set for all $1< p<\infty$.
\end{cor}

Below we give two more generalizations of the main result in \cite{blendekmichalicek2013sidonl1}. For convenience we introduce the following definitions.
\begin{defn}
	We say that $\mathbf E \subset \irr$ is a \emph{central interpolation set} of  $\mathrm M(L^\infty(\mathbb G))$ if for each bounded scalar sequence $(c_\pi)_{\pi\in \mathbf E}\subset \mathbb C$, there exists a bounded sequence $(\tilde{c}_\pi)_{\pi\in \irr}\subset \mathbb C$ with $\tilde{c}_\pi=c_\pi$ for $\pi\in \mathbf E$ such that $\tilde c=(\tilde c_\pi \mathrm{Id}_{\pi})_{\pi\in \irr}$ is a bounded multiplier on $L^\infty(\mathbb G)$.
\end{defn}  We remark that if $\mathbf E$ is a Sidon set or an interpolation set of $\mathrm M(L^\infty(\mathbb G))$, or if $\mathbb G$ is of Kac type and $\mathbf E$ is a central Sidon set which will be introduced in Section \ref{sect: central Sidon sets with examples}, then $\mathbf E$ is such a central interpolation set of  $\mathrm M(L^\infty(\mathbb G))$.

\begin{thm}\label{thm: central sidon lambda}
Let $\mathbb G$ be a compact quantum group. Assume that  $\mathbf E$ is a central interpolation set of  $\mathrm M(L^\infty(\mathbb G))$.

\emph{(1)} If $1< p<\infty$ and $\sup_{\pi\in\mathbf E}\|\chi_\pi\|_p<\infty$, then $\mathbf E$ is a central $\Lambda(p)$-set;

\emph{(2)} If $\sup_{\pi\in\mathbf E} d_\pi<\infty$, then $\mathbf E$ is a $\Lambda(p)$-set for $1< p<\infty$.
\end{thm}

\begin{proof} 
As is in the previous proof, for each $\pi\in\mathrm{Irr}(\mathbb G)$, we fix a basis in $H_\pi$ such that the associated matrix of $Q_\pi$ is diagonal, and denote by $u^{(\pi)}\in \mathbb M_{n_\pi}(C(\mathbb G))$ the representation matrix under this basis. 

Denote by $(\varepsilon_{\pi})_{\pi\in\mathbf{E}}$ a
Rademacher sequence on a probability space $(\Omega,P)$ and write $\mathfrak e=(\varepsilon_\pi \mathrm{Id}_\pi)_{\pi\in\mathbf{E}}\in \ell^\infty (\mathbf E)$. The same argument as in the proof of Theorem \ref{sidon lambda} yields that we may find $K>0$ such that for $x\in\mathrm{Pol}_{\mathbf{E}}(\mathbb{G})$, $\omega\in \Omega$ and $2< p< \infty$, 
\[
\|x\|_{p}=\|m_{\mathfrak e (\omega)^2}^Rx\|_{p}\leq K\|m_{\mathfrak e (\omega)}^Rx\|_{p}.
\]
Integrating the inequality over $\omega\in \Omega$, we get 
\begin{equation*}
\|x\|_{p}\leq K\int_{\Omega}\|\sum_{\pi\in\mathbf{E}}d_{\pi}\varepsilon_{\pi}(\iota\otimes\mathrm{Tr})[(1\otimes \hat{x}(\pi)Q_{\pi})u^{(\pi)}]\|_{p}dP.\label{eq:Lambda p Fubini}
\end{equation*}
By Theorem \ref{khintchine}, we have with some $K_{0}>0$,
\begin{align*}
 &\quad\,   \int_{\Omega}\|\sum_{\pi\in\mathbf{E}}d_{\pi}\varepsilon_{\pi}(\iota\otimes\mathrm{Tr})[(1\otimes \hat{x}(\pi)Q_{\pi})u^{(\pi)}]\|_{p}dP\\
 &\leq K_{0}\sqrt{p}
 \Big(\sum_{\pi\in\mathbf{E}}\|d_{\pi}(\iota\otimes\mathrm{Tr})[(1\otimes \hat{x}(\pi)Q_{\pi})u^{(\pi)}]\|_{p}^2\Big)^{1/2}.\nonumber
\end{align*}
Therefore
\begin{equation}
\label{eq: c sidon lambda type two estimate}
\|x\|_p \leq KK_0\sqrt p \Big(\sum_{\pi\in\mathbf{E}}\|d_{\pi}(\iota\otimes\mathrm{Tr})[(1\otimes \hat{x}(\pi)Q_{\pi})u^{(\pi)}]\|_{p}^2\Big)^{1/2}.
\end{equation}

(1) If $2< p<\infty$ and $x=\sum_{\pi\in\mathbf E}c_\pi\chi_\pi\in\mathrm{Pol}_{\mathbf{E}}(\mathbb{G})$ with $(c_\pi)_{\pi\in\mathbf E}\subset\mathbb C$, then the above inequality \eqref{eq: c sidon lambda type two estimate} reads
$$\|x\|_p\leq KK_0\sqrt p (\sum_{\pi\in\mathbf E} |c_\pi|^2\|\chi_\pi\|_p^2)^{1/2}.$$
Note that by \eqref{eq:haar state def} and the choice of basis in $H_\pi$, we have $\|\chi_\pi\|_2^2=h(\chi_\pi^*\chi_\pi)=1$. Also recall that $\chi_\pi$ are orthogonal with respect to $h$. Thus  the condition $K_1=\sup_{\pi\in\mathbf E}\|\chi_\pi\|_p<\infty$ implies
$$\|x\|_p\leq KK_1K_0\sqrt p (\sum_{\pi\in\mathbf E} |c_\pi|^2\|\chi_\pi\|_2^2)^{1/2}=\|x\|_2.$$
Therefore $\mathbf E$ is a central $\Lambda(p)$-set.

(2) Assume $K_2=\sup_{\pi\in\mathbf E} d_\pi<\infty$.
We need to show
that the right term of the inequality \eqref{eq: c sidon lambda type two estimate} is not more than $\|x\|_{2}$, up to a constant independent of $x$.
Note that by traciality of $\mathrm{Tr}$, we have
\[
(\iota\otimes\mathrm{Tr})[(1\otimes \hat{x}(\pi)Q_{\pi})u^{(\pi)}]=(\iota\otimes\mathrm{Tr})[u^{(\pi)}(1\otimes \hat{x}(\pi)Q_{\pi})]
\]
Also, we note that the map $d_{\pi}^{-1}(\iota\otimes\mathrm{Tr}(\cdot\, (1\otimes Q_{\pi})))$
is unital completely positive on $C(\mathbb G)\otimes B(H_{\pi})$ for each $\pi\in\mathrm{Irr}(\mathbb{G})$,
so by the Cauchy-Schwarz inequality,
\begin{align*}
 &\quad\, \left|d_{\pi}^{-1}(\iota\otimes\mathrm{Tr})[u^{(\pi)}(1\otimes \hat{x}(\pi)Q_{\pi})]D^{1/p}\right|^{2}\\
& =  D^{1/p}\left|d_{\pi}^{-1}(\iota\otimes\mathrm{Tr})[u^{(\pi)}(1\otimes \hat{x}(\pi)Q_{\pi})]\right|^{2}D^{1/p}\\
& \leq  D^{1/p}d_{\pi}^{-1}(\iota\otimes\mathrm{Tr}(\cdot\, Q_{\pi}))[(1\otimes\hat{x}(\pi)^{*})(u^{(\pi)})^{*}u^{(\pi)}(1\otimes \hat{x}(\pi))]D^{1/p}\\
& =  d_{\pi}^{-1}D^{1/p}\mathrm{Tr}(|\hat{x}(\pi)|^{2}Q_{\pi})D^{1/p}=d_{\pi}^{-1}\mathrm{Tr}(|\hat{x}(\pi)|^{2}Q_{\pi})D^{2/p}.
\end{align*}
Hence together with Proposition \ref{density op haagerup Lp} (5),
\begin{align*}
 & \quad\, \sum_{\pi\in\mathbf{E}}\|d_{\pi}(\iota\otimes\mathrm{Tr})[(1\otimes  \hat{x}(\pi)Q_{\pi})u^{(\pi)}]\|_{p}^2
 =\sum_{\pi\in\mathbf{E}}\||d_{\pi}(\iota\otimes\mathrm{Tr})[(1\otimes  \hat{x}(\pi)Q_{\pi})u^{(\pi)}]D^{1/p}|^{2}\|_{p/2,\mathsf{H}}\\
& \leq K_2^2 \sum_{\pi\in\mathbf{E}}d_\pi^2\||d_{\pi}^{-1}(\iota\otimes\mathrm{Tr})[(1\otimes  \hat{x}(\pi)Q_{\pi})u^{(\pi)}]D^{1/p}|^{2}\|_{p/2,\mathsf{H}}\\
&
\leq K_2^2 \sum_{\pi\in\mathbf{E}} d_\pi^2 \| d_{\pi}^{-1}\mathrm{Tr}(|\hat{x}(\pi)|^{2}Q_{\pi})D^{2/p}\|_{p/2,\mathsf{H}}
=  \sum_{\pi\in\mathbf{E}}d_{\pi}\mathrm{Tr}(|\hat{x}(\pi)|^{2}Q_{\pi})=\|\hat{x}\|_{2}^{2}=\|x\|_{2}^{2}.
\end{align*}
Now back to \eqref{eq: c sidon lambda type two estimate} we get $\|x\|_{p}\leq KK_{0}K_2\sqrt{p}\|x\|_{2}$, as desired.
\end{proof}

\begin{rem}
	Let us make a few remarks on the constant of $\Lambda(p)$-sets. In the proof of Theorem \ref{thm:lambda inter} and Theorem \ref{sidon lambda}, we have shown that, if $\mathbf E \subset \irr$ is an interpolation set of $\mathrm M(L^p(\mathbb G))$ with constant $K$ for some $1< p<\infty$, then $\mathbf E$ is a $\Lambda(p)$-set with constant $c_1 K^{c_2} p$ with two universal constants $c_1,c_2>0$; and if $\mathbf E \subset \irr$ is an interpolation set of $\mathrm M(L^\infty(\mathbb G))$, then $\mathbf E$ is a $\Lambda(p)$-set with constant $c'_1 K^{c'_2} p$ with two universal constants $c'_1,c'_2>0$ for all $1< p<\infty$. It seems that these constants should be improved. Indeed, it is well-known that if $\mathbb G$ is a compact group or the dual quantum group of a discrete group, the constants above can be improved to $c_1 K^{c_2} \sqrt p$ and $c'_1 K^{c'_2} \sqrt p$ respectively (see \cite{hewittross1970abstract,harcharras99nclambdap}). On the other hand, we see that, for some subclass of Sidon sets as in Theorem \ref{thm: central sidon lambda}, we may also obtain the constant in the form
 $c_1 K^{c_2} \sqrt p$.
\end{rem}

\begin{example}\label{ex: lambda p}
(1) Following the notation in Example \ref{free unitary weak sidon}, we consider $\mathbb G=\prod_{k\geq 1}\mathbb G_{k}$, $\mathbb G_{k}=U_{N_k}^+$ and $\mathbf E=\{u^{(k)}:k\geq 1\}\subset \mathrm{Irr}(\mathbb G)$. We saw in Example \ref{free unitary weak sidon} that $\mathbf E$ is an interpolation set of $\mathrm M (L^\infty (\mathbb G))$, and hence by Theorem \ref{sidon lambda} $\mathbf E$ is a $\Lambda(p)$-set for all $1< p<\infty$. Alternatively, recall the Haagerup type inequality shown in Brannan \cite[Theorem 6.3]{brannan12haagerup}: for $k\geq 1$, $x_{ij}\in \mathbb C$, $1\leq i,j\leq N_k$
$$\|\sum_{i,j=1}^{N_k}x_{ij}u_{ij}^{(k)}\|_p\leq C\|\sum_{i,j=1}^{N_k}x_{ij}u_{ij}^{(k)}\|_2$$ 
for a universal constant $C$. So by Theorem \ref{khintchine} and the standard argument as in the beginning of the proof of Theorem \ref{thm: central sidon lambda}, we have the following Khintchine type inequality: for a universal constant $K$ and for all $2\leq p<\infty$ and all finitely supported sequences $(A_k)\in \prod_k \mathbb M_{N_k}$,
\begin{align*}
\|\sum_{k\geq 1}N_k(\iota\otimes\mathrm{Tr})[(1\otimes A_k)u^{(k)}]\|_{L^p(\mathbb G)}
& \leq K\sqrt p \Big(\sum_k\|N_k(\iota\otimes\mathrm{Tr})[(1\otimes A_k)u^{(k)}]\|_{L^p(\mathbb G)}^2\Big)^{1/2}\\
& \leq CK\sqrt p\Big(\sum_kN_k\mathrm{Tr}(|A_k|^2)\Big)^{1/2}.
\end{align*}
If $n_k=1$ for all $k$, then the above inequality reduces to the classical Khintchine inequalities.

(2) Follow the notation in Example \ref{ex: sutwo sidon set}. We consider the sequence $(q_n)_{n\geq 1}\subset [q,1]$ with $q\coloneqq\inf_n q_n>0$ and the associated quantum group $\mathbb G =\prod_{n\geq 1}\mathrm{SU}_{q_n}(2)$. In Example \ref{ex: sutwo sidon set} we proved that $\mathbf{E}=\{u_n:n\geq 1\} \subset\irr $ is a Sidon set for $\mathbb G$. Note that the associated matrix $Q_n\coloneqq Q_{u_n}$ is diagonal with entries $\{q_n^{-1},q_n\}$ under the standard basis and hence $\sup_n\dim_q(u_n)=\sup_n(q_n+q_n^{-1})\leq 1+q^{-1}<\infty$. So by Theorem \ref{thm: central sidon lambda}, $\mathbf{E}$ is also a $\Lambda(p)$-set for $2\leq p <\infty$ and we obtain the following Khintchine type inequalities: there exists a constant $K>0$ (depending on $q$) so that for all finitely supported sequences $(A_n)\in \prod_{n\geq 1} \mathbb M_{2}$,
$$ \Big(\sum_{n\geq 1}d_n\mathrm{Tr}(|A_n|^2Q_n)\Big)^{1/2}
\leq \|\sum_{n\geq 1}d_n(\iota\otimes\mathrm{Tr})[(1\otimes A_nQ_n)u_n]\|_{L^p(\mathbb G)} \leq K \sqrt p \Big(\sum_{n\geq 1}d_n\mathrm{Tr}(|A_n|^2Q_n)\Big)^{1/2}$$
where $d_n=q_n+q_n^{-1}$. Note that by Proposition \ref{prop: unif bdd of mod gp for sidon} and Theorem \ref{thm:lambda inter}, if $q=0$ and $q_n\to 0$, the subset $\mathbf E$ defined as above cannot be a $\Lambda(p)$-set for any $2<p<\infty$.
%
%However, if $q=0$ and $q_n\to 0$, the subset $\mathbf E$ defined as above is not a $\Lambda(4)$-set. To see this, we use the following formula in \cite[Theorem 6.2.17]{timmermann08qgbook}:
%$$h(\gamma_n^*\gamma_n)=\dfrac{1}{1+q^2_n},\quad h(\gamma_n^*\gamma_n\gamma_n^*\gamma_n)=\dfrac{1-q^2_n}{1-q^6_n}.$$
%Let $\mathrm{tr}$ be the trace introduced in Proposition \ref{density op haagerup Lp} associated to the Haar state $h$ on $ \mathbb G$. It is well-known and easy to see from \eqref{modular group on cqg} that $\sigma_{\frac{\mathrm i}{2}}(\alpha_n^*\alpha_n)=\alpha_n^*\alpha_n$. Also recall $\alpha_n^*\alpha_n+\gamma_n^*\gamma_n=1$ by the unitarity of $u_n$. Then
%\begin{align*}
% \|\alpha_n\|_4^4&=\mathrm{tr}(|\alpha_n D^\frac{1}{4}|^4)
% =\mathrm{tr}(D^\frac{1}{4}\alpha_n^*\alpha_n D^\frac{1}{2} \alpha_n^*\alpha_n D^\frac{1}{4} )
% =\mathrm{tr}(D^\frac{3}{4}\sigma_{\frac{\mathrm i}{2}}(\alpha_n^*\alpha_n) \alpha_n^*\alpha_n D^\frac{1}{4} )\\
% &=\mathrm{tr}(\alpha_n^*\alpha_n \alpha_n^*\alpha_n D)=h(\alpha_n^*\alpha_n \alpha_n^*\alpha_n)=h((1-\gamma_n^*\gamma_n)^2)\\
% &=1-2h(\gamma_n^*\gamma_n)+h(\gamma_n^*\gamma_n\gamma_n^*\gamma_n)
% =1-\dfrac{2}{1+q^2_n}+\dfrac{1-q^2_n}{1-q^6_n}.
%\end{align*} 
%But 
%$$\|\alpha_n\|_2^2=h(\alpha_n^*\alpha_n)=\frac{q_n^2}{1+q_n^2}.$$
%So
%$$\dfrac{\|\alpha_n\|_4^4}{\|\alpha_n\|_2^4}=(1+q_n^{-2})^2\left(1-\dfrac{2}{1+q^2_n}+\dfrac{1-q^2_n}{1-q^6_n}\right)\to \infty,\quad n\to\infty\  (\text{i.e., }q_n\to 0),$$
%which implies that $\mathbf E$ is not a $\Lambda(4)$-set.
\end{example}

\subsection{Independence of the interpolation parameters for $L^p_{(\theta)}(\mathbb G)$}
In this subsection we would like to show that our definition of $\Lambda(p)$-sets does not depend on different interpolation parameters of $L^p$-spaces associated to a compact quantum group $\mathbb G$. Recall that in Section \ref{subs:nc Lp}, we introduce the space $L^p(\mathbb G)$ as the complex interpolation space $(L^\infty(\mathbb G),L^\infty (\mathbb G)_*)_{1/p}$ associated to the compatible couple $(L^\infty(\mathbb G),L^\infty (\mathbb G)_*)$ given by the embedding
$$L^\infty(\mathbb G)\hookrightarrow L^\infty (\mathbb G)_*,\quad x\mapsto h(\cdot \,x),\quad x\in L^\infty(\mathbb G).$$
However in \cite{kosaki84interpolation} Kosaki provides some other possibilities to define the complex interpolation scale $(L^p(\mathbb G))_{1\leq p\leq \infty}$. More precisely, fix a parameter $0\leq \theta \leq 1$, we may consider the compatible couple $(L^\infty(\mathbb G),L^\infty (\mathbb G)_*)^{(\theta)}$ given by the embedding
$$L^\infty(\mathbb G)\hookrightarrow L^\infty (\mathbb G)_*,\quad x\mapsto h(\cdot \,\sigma_{-\theta\mathrm i}(x)),\quad x\in L^\infty(\mathbb G),$$
and we define the complex interpolation space
$$L^p_{(\theta)}(\mathbb G)=(L^\infty(\mathbb G),L^\infty (\mathbb G)_*)_{1/p}^{(\theta)}.$$
Note that $L^p_{(0)}(\mathbb G)$ coincides with the space $L^p(\mathbb G)$ defined before. In the language of Haagerup's $L^p$-spaces, we have considered the embedding
$$L^\infty(\mathbb G)\hookrightarrow L^{1,\mathsf{H}} (\mathbb G),\quad x\mapsto  D^{\theta}x D^{1-\theta}=\sigma_{-\theta\mathrm i}(x)D,\quad x\in L^\infty(\mathbb G),$$
and for each $1\leq p\leq \infty$ we have the isometric isomorphism
$$L^p_{(\theta)}(\mathbb G)\to L^{p,\mathsf H}(\mathbb G),\quad x\mapsto D^{\frac{\theta}{p}}xD^{\frac{1-\theta}{p}},\quad x\in L^\infty (\mathbb G).$$
In particular,
\begin{equation}\label{eq:p norm diff theta}
\|x\|_{L^p_{(\theta)}(\mathbb G)}=\|\sigma_{-\mathrm i \frac{\theta}{p}}(x)D^{\frac{1}{p}}\|_{L^{p,\mathsf H}(\mathbb G)}=\|\sigma_{-\mathrm i \frac{\theta}{p}}(x)\|_{L^p_{(0)}(\mathbb G)},\quad
x\in \pol.
\end{equation}
The spaces $L^p_{(\theta)}(\mathbb G)$ and $L^p_{(\theta')}(\mathbb G)$ for different parameters $\theta,\theta'$ are isometric as Banach spaces, but one needs to be careful with the parameter $\theta$ when doing the Fourier analysis on $\mathbb G$, for which we refer to \cite[Section 7]{caspers2013fourier} for some related discussions. So returning back to the topic on $\Lambda(p)$-sets, it is natural to 
ask if the notion of $\Lambda(p)$-sets is independent of the choice of the parameter $\theta$. In the following we give an affirmative answer.

\begin{prop}
	\label{prop:indep lambda parameter}
	Let $\mathbb G$ be a compact quantum group and let $2<p<\infty$, $0\leq \theta\leq 1$. Then $\mathbf E \subset \irr$ is a $\Lambda(p)$-set for $L^p_{(\theta)}(\mathbb G)$, that is, there exists a constant $K>0$ with
	$$\|x\|_{L^p_{(\theta)}(\mathbb G)} \leq K \|x\|_{L^2_{(\theta)}(\mathbb G)},\quad x\in \mathrm{Pol}_{\mathbf E}(\mathbb G),$$
	if and only if it is a $\Lambda(p)$-set for $L^p_{(0)}(\mathbb G)$ in the sense of Definition \ref{def:lambda p}.
\end{prop}
\begin{proof}
	Assume that $\mathbf E \subset \irr$ is a $\Lambda(p)$-set for $L^p_{(\theta)}(\mathbb G)$ with the constant $K$ given above. Then for all $a\in\ell^\infty (\mathbf E)$, together with the equality \eqref{eq:p norm diff theta} we have for $x\in \mathrm{Pol}_{\mathbf E}(\mathbb G)$,
	\begin{align}\label{eq:diff para pf 1}
		\|m_a^R x\|_{L^p_{(\theta)}(\mathbb G)} 
		& \leq K \|m_a^R x\|_{L^2_{(\theta)}(\mathbb G)}
		= K \big\|\sigma_{-\mathrm i \frac{\theta}{2}}(m_a^R x)\big\|_{L^2_{(0)}(\mathbb G)}.
	\end{align}
	By Lemma \ref{lem:fourier series with sigma}, we know that 
	$$\mathcal F (\sigma_{-\mathrm i \frac{\theta}{2}}(m_a^R x)) = Q^{\frac{\theta}{2}} a\hat x Q^{\frac{\theta}{2}}.$$
	Hence by Proposition \ref{prop:Plancherel},
	\begin{equation}\label{eq:diff para pf 2}
	\big\|\sigma_{-\mathrm i \frac{\theta}{2}}(m_a^R x)\big\|_{L^2_{(0)}(\mathbb G)}^2
	=\sum_{\pi\in\irr}d_\pi \mathrm{Tr}(|Q^{\frac{\theta}{2}}_\pi a \hat x(\pi) Q^{\frac{\theta}{2}}_\pi|^2Q_\pi).
	\end{equation}
	Let $\pi\in\mathbf E$ and without loss of generality we choose an appropriate basis of $H_\pi$ so that $Q_\pi$ is diagonal under this basis. Assume that $a\in c_c (\mathbf E)$ and take $x\in \mathrm{Pol}_{\mathbf E}(\mathbb G)$, then the above inequalities gives
	\begin{align*}
	\| m_{a}^R x\|_{L^p_{(\theta)}(\mathbb G)} ^2 
	& \leq K^2 \|Q^{\frac{\theta}{2}} a Q^{-\frac{\theta}{2}} \|_\infty ^2 \sum_{\pi\in\mathbf E}d_\pi \mathrm{Tr}(|Q^{\frac{\theta}{2}}_\pi  \hat x(\pi) Q^{\frac{\theta}{2}}_\pi|^2Q_\pi)\\
	& = K^2 \|Q^{\frac{\theta}{2}} a Q^{-\frac{\theta}{2}} \|_\infty ^2\big\|\sigma_{-\mathrm i \frac{\theta}{2}}( x)\big\|_{L^2_{(0)}(\mathbb G)}^2
	=K^2 \|Q^{\frac{\theta}{2}} a Q^{-\frac{\theta}{2}} \|_\infty ^2 \| x\|_{L^2_{(\theta)}(\mathbb G)}^2\\
	& \leq K^2 \|Q^{\frac{\theta}{2}} a Q^{-\frac{\theta}{2}} \|_\infty ^2 \| x\|_{L^p_{(\theta)}(\mathbb G)}^2.
	\end{align*}
	And a similar inequality can be proved for the map $m_{a}^L$. Then following the idea in the proof of Proposition \ref{prop: unif bdd of mod gp for sidon}, we may find a constant $K'>0$ such that  
		$$\max\{\|Q_\pi\|,\|Q_\pi^{-1}\|\}\leq K',\quad \pi\in \mathbf E.$$
	Therefore together with \eqref{eq:p norm diff theta}, Lemma \ref{lem:fourier series with sigma} and Lemma \ref{lem:multiplier on ltwo}, we get for all $x\in \mathrm{Pol}_{\mathbf E}(\mathbb G)$,
	\begin{align*}
		\|x\|_{L^p_{(0)}(\mathbb G)}
		& = \|\sigma_{\mathrm i \frac{\theta}{p}}(x) \|_{L^p_{(\theta)}(\mathbb G)}
		\leq K \|\sigma_{\mathrm i \frac{\theta}{p}}(x) \|_{L^2_{(\theta)}(\mathbb G)} 
		= K \|\sigma_{\mathrm i (\frac{1}{p}-\frac{1}{2})\theta}(x) \|_{L^2_{(0)}(\mathbb G)}\\
		& = K \|\mathcal F(\sigma_{\mathrm i (\frac{1}{p}-\frac{1}{2})\theta}(x)) \|_{\ell^2(\hat{\mathbb G})} 
		= K \|Q^{\frac{1}{2}-\frac{1}{p}}\hat x Q^{\frac{1}{2}-\frac{1}{p}} \|_{\ell^2(\hat{\mathbb G})}\\
		&= K \|m^R_{Q^{\frac{1}{2}-\frac{1}{p}}}m^L_{Q^{\frac{1}{2}-\frac{1}{p}}}x\|_{L^2_{(0)}(\mathbb G)}
		\leq K (K')^{1-\frac{2}{p}} \|x\|_{L^2_{(0)}(\mathbb G)}.
	\end{align*}
	So $\mathbf E$ is a $\Lambda(p)$-sets for $L^p_{(0)}(\mathbb G)$.
	
	Note that if conversely $\mathbf E$ is a $\Lambda(p)$-set for $L^p_{(0)}(\mathbb G)$, then we have already shown in Proposition \ref{prop: unif bdd of mod gp for sidon} that there exists a constant $K'>0$ such that  
	$$\max\{\|Q_\pi\|,\|Q_\pi^{-1}\|\}\leq K',\quad \pi\in \mathbf E.$$
	So a similar estimation as above yields that $\mathbf E$ must be a $\Lambda(p)$-sets for $L^p_{(\theta)}(\mathbb G)$. Therefore the proof is complete.
\end{proof}

Finally we remark that all the discussions in the previous sections can be in fact reproduced for $L^p_{(\theta)}(\mathbb G)$ with the similar idea, and we omit the details. 

\subsection{Existence of $\Lambda(p)$-sets}

In this short subsection we discuss the existence of $\Lambda(p)$-sets for compact quantum groups. We refer to the following result proved in the appendix: let $M$ be a von Neumann algebra with a normal faithful state $\varphi$ and $B=\{x_k:k\geq 1\}\subset M$ be an orthonormal system with respect to $\varphi$ such that $\sup_k\|x_k\|_\infty<\infty$, then there exists an infinite subset $Y\subset B$ and a constant $C_p$ such that $\|x\|_p\leq C_p \|x\|_2$ for all $x\in \mathrm{span}\,(Y)$. 
Immediately we deduce the existence of $\Lambda(p)$-sets with uniform dimension assumptions.

\begin{thm}\label{thm: exist lambda p}
Let $\mathbb G$ be a compact quantum group. Let $\mathbf E\subset \irr$ be an infinite subset with $\sup_{\pi\in \mathbf E}d_\pi<\infty$. Then for each $2< p <\infty$, there exists an infinite subset $\mathbf F \subset \mathbf E$ which is a $\Lambda(p)$-set for $\mathbb G$.
\end{thm}
\begin{proof}
Denote $D_0=\sup_{\pi\in \mathbf E}d_\pi<\infty$ and fix $2\leq p<\infty$. Then also $n_\pi=\dim(H_\pi)\leq d_\pi \leq D_0$. Choose an appropriate basis of $H_\pi$ such that the matrix $Q_\pi$ is diagonal under this basis. For each $\pi\in \mathbf E$, $1\leq i,j \leq n_\pi$, write $v_{ij}^{(\pi)}=\|u_{ij}^{(\pi)}\|_2^{-1}u_{ij}^{(\pi)}$. By \eqref{eq:haar state def}, $\|v_{ij}^{(\pi)}\|_\infty\leq \|u_{ij}^{(\pi)}\|_2^{-1}\leq d_\pi^2 \leq D_0^2$. Consider $B_0^{(\pi)}=\{v_{ij}^{(\pi)}:1\leq i,j \leq n_\pi\}$ for $\pi\in \mathbf E$, and $B_0=\cup_{\pi\in\mathbf E}B_0^{(\pi)}$. Then $B_0$ is an orthogonal system since $Q_\pi$ is chosen diagonal. Write $\mathbf E_0=\mathbf E$. According to the  theorem in the appendix, we can find an infinite subset $B_1\subset B_0$ with constant $C_1>0$ such that for all finitely many  $c_1,\ldots,c_n\subset \mathbb C$, $x_1,\ldots,x_n\subset B_1$,
$$\|\sum_{l=1}^nc_lx_l\|_p\leq C_1 \|\sum_{l=1}^nc_lx_l\|_2.$$
Set $\mathbf E_1=\{\pi\in\mathbf E_0:\exists v_{ij}^{(\pi)}\in B_1\}$ and let $B^{(\pi)}_1=\{v_{ij}^{(\pi)}\in B_1:1\leq i,j \leq n_\pi\}$ for $\pi\in \mathbf E_1$. The last set is non-empty. Then $$B_1=\cup_{\pi\in \mathbf E_1}B^{(\pi)}_1,\qquad \mathrm{Card}(B^{(\pi)}_0\setminus B^{(\pi)}_1)\leq n_\pi^2-1\leq D_0^2-1,\ \pi\in \mathbf E_1.$$
Repeating inductively the above procedures, if $k \geq 1$ and if $\exists \pi\in \mathbf E_k$, $B^{(\pi)}_0\setminus (\cup_{l=1}^kB^{(\pi)}_l) \neq \emptyset$, we construct the proper subsets $\mathbf E _{k+1}\subset\mathbf  E_k$, $B^{(\pi)}_{k+1}\subset B^{(\pi)}_0\setminus (\cup_{l=1}^kB^{(\pi)}_l)$,  $B_{k+1}=\cup_{\pi\in \mathbf E _{k+1}}B^{(\pi)}_{k+1}\subset B_0\setminus (\cup_{l=1}^kB_l) $ and a constant $C_{k+1}>0$ such that for all finitely many  $c_1,\ldots,c_n\subset \mathbb C$, $x_1,\ldots,x_n\subset B_{k+1}$,
$$\|\sum_{l=1}^nc_lx_l\|_p\leq C_{k+1} \|\sum_{l=1}^nc_lx_l\|_2$$
and such that
$$\mathrm{Card}(\mathbf E _{k+1})=\infty,\qquad
\mathrm{Card}(B^{(\pi)}_0\setminus (\cup_{l=1}^{k+1}B^{(\pi)}_l)\leq D_0^2-(k+1),\ \pi\in \mathbf E_{k+1}.$$
Since $D_0$ is finite, the above inequality shows that there exists $\tilde{k}\leq D_0^2$ such that $B^{(\pi)}_0= \cup_{l=1}^{\tilde{k}}B^{(\pi)}_l$ for $\pi\in \mathbf E _{\tilde{k}}$. Let $\mathbf F =\mathbf E_{\tilde{k}}$ and $\tilde{B}_l=\cup_{\pi\in\mathbf F}B^{(\pi)}_l\subset B_l$. Then for each $1\leq l \leq \tilde{k}$, 
$$ \|x\|_p \leq C_l\|x\|_2,\quad x\in \mathrm{span}(\tilde{B}_l).$$
Since $\cup_{l=1}^{\tilde{k}}\tilde{B}_l\subset B_0$ is an orthonormal system and 
$$\mathrm{Pol}_{\mathbf F}(\mathbb{G})=\mathrm{span}(\cup_{\pi\in\mathbf F}B_0^{(\pi)})=\mathrm{span}(\cup_{\pi\in\mathbf F}\cup_{l=1}^{\tilde{k}}B^{(\pi)}_l)
=\mathrm{span}(\cup_{l=1}^{\tilde{k}}\tilde{B}_l),$$
we obtain
$$\|x\|_p\leq D_0^2\max \{C_l:1\leq l\leq \tilde{k}\} \|x\|_2,\quad x\in  \mathrm{Pol}_{\mathbf F}(\mathbb{G}).$$
Hence $\mathbf F\subset \irr$ is the desired infinite $\Lambda(p)$-set.
\end{proof}

The existence of $\Lambda(p)$-sets without the assumption $\sup_{\pi\in \mathbf E}d_\pi<\infty$ in the above theorem is in general not true, which can be seen from the non-existence of central $\Lambda(4)$-sets for the classical $\mathrm{SU}(2)$, as well as from the quantum non-tracial example \ref{ex: lambda p}.(2).

\subsection{Remarks on the lacunarity for $SU_q(2)$}
In \cite{hewittross1970abstract}, a classical version of Theorem \ref{sidon lambda} was used to prove the fact that the special unitary group $\mathrm{SU}(2)$ does not admit any infinite Sidon set. The key observation therein is that $\mathrm{SU}(2)$ does not admit any infinite central $\Lambda(4)$-set. Here we want to show that, in strong contrast, the $q$-deformed quantum group $\mathrm{SU}_q(2)$ with $0<q<1$ \emph{does} admit an infinite central $\Lambda(4)$-set. The non-traciality of the Haar state on $\mathrm{SU}_q(2)$ plays an essential role for this result.
We recall that $\mathrm{Irr}(\mathrm{SU}_q(2))$ can be identified with $\mathbb N \cup\{0\}$ and follow the notation in the preliminary part. 
\begin{prop}\label{central lambda four su}
	Let $0<q<1$ and let $\mathbf{E}=\{n_{k}\in\mathbb{N}\cup \{0\}:k\geq0\}\subset \mathrm{Irr}(\mathrm{SU}_q(2))$ be such that $n_{k}=n_{k-1}+k$
	for $k\geq1$. Then $\mathbf E$ is a central $\Lambda(4)$-set for $\mathrm{SU}_q(2)$. More precisely, there exists $K_{q}\geq0$ such that for any finitely
	supported sequence $(c_{n})_{n\in\mathbf{E}}\in\mathbb{C}$,
	\begin{equation}
	\|\sum_{n\in\mathbf{E}}c_{n}\chi_{n}\|_{4}\leq K_{q}\|\sum_{n\in\mathbf{E}}c_{n}\chi_{n}\|_{2}.\label{eq:central lampda 4 for quantum su2}
	\end{equation}
\end{prop}
\begin{proof}
	Recall the formulae \eqref{modular group on cqg} and \eqref{qn for sutwo}. Then for each $m\in\mathbb{N}\cup \{0\}$, 
	\begin{align}
	h(\chi_{m}\sigma_{-\frac{\mathrm{i}}{2}}(\chi_{m})) & =\sum_{i,j=1}^{m+1}h((u_{ii}^{(m)})^{*}\sigma_{-\frac{\mathrm{i}}{2}}(u_{jj}^{(m)}))=\sum_{i,j=1}^{m+1}h((u_{ii}^{(m)})^{*}(Q_{m}^{1/2})_{jj}u_{jj}^{(m)}(Q_{m}^{1/2})_{jj})\label{eq:cal chi m}\\
	& =\sum_{i=1}^{m+1}(Q_{m})_{ii}h((u_{ii}^{(m)})^{*}u_{ii}^{(m)})=\sum_{i=1}^{m+1}(Q_{m})_{ii}(Q_{m}^{-1})_{ii}/\mathrm{Tr}(Q_{m})\nonumber \\
	& =\frac{m+1}{q^{-m}+q^{-m+2}+\cdots+q^{m-2}+q^{m}}.\nonumber 
	\end{align}
	On the other hand, it is easy to see 
	\begin{equation}
	\|\chi_{m}\|_{2}^{2}=h(\chi_{m}^{2})=1,\quad\|\sum_{n\in\mathbf{E}}c_{n}\chi_{n}\|_{2}^{2}=\sum_{n\in\mathbf{E}}|c_{n}|^{2}.\label{eq:chi 2 norm}
	\end{equation}
	By \eqref{eq:haar state def} and \eqref{modular group on cqg}, we also have
	\begin{equation}
	\label{ortho char su}
	h(\chi_m\sigma_{-\frac{\mathrm{i}}{2}}\chi_n)=0,\quad m\neq n.
	\end{equation}
	Let $\mathrm{tr}$ be the trace on the
	Haagerup $L^{1}$-space $L^{1,\mathsf H}(\mathrm{SU}_{q}(2))$. Recall Proposition \ref{density op haagerup Lp} - we have
	for a finitely supported sequence $(c_{n})_{n\in\mathbf{E}}$
	with $\max_{n}|c_{n}|=1$ and $f=\sum_{n\in\mathbf{E}}c_{n}\chi_{n}$,
	\begin{align*}
	\|f\|_{4}^{4} & =\mathrm{tr}(|fD^{1/4}|^{4})=\mathrm{tr}(D^{1/4}f^{*}fD^{1/2}f^{*}fD^{1/4})\\
	& =\mathrm{tr}(f^{*}f(D^{1/2}f^{*}fD^{-1/2})D)=h(f^{*}f\sigma_{-\frac{\mathrm{i}}{2}}(f^{*}f))\\
	& =\sum_{i,j,r,s\in \mathbf E}\bar{c}_{i}c_{j}\bar{c}_{r}c_{s}h(\chi_{i}\chi_{j}\sigma_{-\frac{\mathrm{i}}{2}}(\chi_{r}\chi_{s})).
	\end{align*}
	Then using \eqref{eq:cal chi m}, \eqref{ortho char su} and the fact $\chi_{m}\chi_{m'}=\chi_{|m-m'|}+\chi_{|m-m'|+1}+\cdots+\chi_{m+m'}$
	for any $m,m'\in\mathbb{N}\cup \{0\}$, we get
	\begin{align*}
	\|f\|_{4}^{4} & \leq\sum_{i,j,r,s\in\mathbf{E}}|\bar{c}_{i}||c_{j}||\bar{c}_{r}||c_{s}|\sum_{m=\max\{|i-j|,|r-s|\}}^{\min\{i+j,r+s\}}|h(\chi_{m}\sigma_{-\frac{\mathrm{i}}{2}}(\chi_{m}))|\\
	& =\sum_{i,j,r,s\in\mathbf{E}}|\bar{c}_{i}||c_{j}||\bar{c}_{r}||c_{s}|\sum_{m=\max\{|i-j|,|r-s|\}}^{\min\{i+j,r+s\}}\frac{m+1}{q^{-m}+q^{-m+2}+\cdots+q^{m-2}+q^{m}}\\
	& \leq\sum_{i,j,r,s\in\mathbf{E}}|\bar{c}_{i}||c_{j}||\bar{c}_{r}||c_{s}|\sum_{m=\max\{|i-j|,|r-s|\}}^{\min\{i+j,r+s\}}\frac{m+1}{q^{-m}},\end{align*}
	and therefore
	\begin{align*}
	\|f\|_{4}^{4} 
	& \leq K\sum_{i,j,r,s\in\mathbf{E}}^{n}|\bar{c}_{i}||c_{j}||\bar{c}_{r}||c_{s}|\sum_{m=\max\{|i-j|,|r-s|\}}^{\min\{i+j,r+s\}}q^{m/2}
	\leq K\sum_{i,j,r,s\in\mathbf{E}}|\bar{c}_{i}||c_{j}||\bar{c}_{r}||c_{s}|\frac{q^{\frac{\max\{|i-j|,|r-s|\}}{2}}}{1-q^{1/2}}\\
	& \leq\frac{K}{1-q^{1/2}}\sum_{i,j,r,s\in\mathbf{E}}|\bar{c}_{i}||c_{j}||\bar{c}_{r}||c_{s}|q^{\frac{|i-j|+|r-s|}{4}}
	=\frac{K}{1-q^{1/2}}\Big(\sum_{i,j\in\mathbf{E}}|\bar{c}_{i}||c_{j}|q^{\frac{|i-j|}{4}}\Big)^{2}\\
	&
	\leq\frac{2K}{1-q^{1/2}}\left(\Big(\sum_{i\in\mathbf{E},i=j}|\bar{c}_{i}|^{2}\Big)^2+\Big(2\sum_{i,j\in\mathbf{E},i<j}|c_{i}||c_{j}|q^{\frac{j-i}{4}}\Big)^{2}\right)
	\end{align*}
	where $K \geq 1$ is the constant such that $x+1\leq Kq^{-x/2}$ for $x\geq 1$.
	Recall the assumption on $\mathbf{E}=\{n_{k}\in\mathbb{N}\cup \{0\}:k\geq0\}$
	that $n_{k}=n_{k-1}+k$ for $k\geq1$ and also the assumption $\max_{n}|c_{n}|=1$, so that $1\leq \|f\|_2$.
	We have
	\begin{align*}
	\sum_{i,j\in\mathbf{E},i<j}|c_{i}||c_{j}|q^{\frac{j-i}{4}} & =\sum_{k\geq0}|c_{n_{k}}|\sum_{l\geq1}|c_{n_{k}+(k+1)+\cdots+(k+l)}|q^{\frac{(k+1)+\cdots+(k+l)}{4}}\\
	& \leq\sum_{k\geq0}\sum_{l\geq1}q^{\frac{(2k+l+1)l}{8}}\leq\sum_{k\geq0}\sum_{l\geq1}q^{\frac{2k+l+1}{8}}\\
	& =\sum_{k\geq0}q^{\frac{2k+1}{8}}\sum_{l\geq1}q^{\frac{l}{8}}\leq q^{1/8}\cdot\frac{1}{1-q^{1/4}}\cdot\frac{q^{1/8}}{1-q^{1/8}}\coloneqq K'.
	\end{align*}
	Therefore together with \eqref{eq:chi 2 norm} we get
	\[
	\|f\|_{4}^{4}\leq\frac{2K}{1-q^{1/2}}\left(\Big(\sum_{i\in\mathbf{E},i=j}|\bar{c}_{i}|^{2}\Big)^2+4K'^{2}\right)\leq\frac{2K(4K'^{2}+1)}{1-q^{1/2}}\|f\|_{2}^{4}.
	\]
	Take $K_{q}=(\frac{2K(4K'^{2}+1)}{1-q^{1/2}})^{1/4}$ and we get the desired
	inequality \eqref{eq:central lampda 4 for quantum su2}.
\end{proof}

In the end we remark that for the case  $0<q<1$, we may also state below the non-existence of infinite $\Lambda(p)$-sets for  $\mathrm{SU}_q(2)$, which directly follows from Proposition \ref{prop: unif bdd of mod gp for sidon}.

\begin{prop}
	\label{prop:no sidon on su}
	Let $0<q<1$ and $1<p<\infty$. There exist no infinite $\Lambda(p)$-sets or infinite Sidon sets for $\mathrm{SU}_q(2)$.
\end{prop}
\begin{proof}
	Recall that the irreducible representations of $\mathrm{Irr}(\mathrm{SU}_q(2))$ indexed by $\mathbb N \cup \{0\}$, and for each $n\in \mathbb N \cup \{0\}$ we have $\|Q_n\|=\|Q_{n}^{-1}\|=q^{-n}$. So for any infinite subset $\mathbf E \subset \mathrm{Irr}(\mathrm{SU}_q(2))$, we have
	$$\sup_{n\in\mathbf E}\{\|Q_n\|, \|Q_{n}^{-1}\| \} =\infty.$$
	So according to Proposition \ref{prop: unif bdd of mod gp for sidon}, $\mathrm{SU}_q(2)$ does not admit any infinite interpolation set for $\mathrm M(L^p(\mathbb G))$, or equivalently, it does not admit any infinite $\Lambda(p)$-set by Theorem \ref{thm:lambda inter}. And by Corollary \ref{sidon lambda}, $\mathrm{SU}_q(2)$ does not admit any infinite Sidon set.
\end{proof}
In the next section we will use another method to show in Corollary \ref{cor: Gq no sidon} that a large class of quantum deformations of semi-simple Lie groups do not admit any infinite Sidon set.

\section{Central Sidon sets with examples}\label{sect: central Sidon sets with examples}

In this final section we briefly discuss some properties of central Sidon sets. 
\begin{defn}
(1)	We say that a subset $\mathbf{E}\subset\mathrm{Irr}(\mathbb{G})$
is a \emph{central Sidon set} if there exists $K>0$ such that for
any finite sequence $(c_{\pi})\subset\mathbb{C}$ and $x=\sum_{\pi}c_{\pi}\chi_{\pi}$,
we have $\|\hat{x}\|_{1}\leq K\|x\|_{\infty}$.

(2) A linear functional on $\mathrm{Pol}(\mathbb{G})$ is said to be \emph{central
}if there exists numbers $(\omega_{\pi}:\pi\in\mathrm{Irr}(\mathbb{G}))$
such that 
\[
\hat{\omega}(\pi)=\omega_{\pi}\mathrm{Id}_{\pi},\quad\pi\in\mathrm{Irr}(\mathbb{G}).
\]
Denote by $\mathrm{Pol}^{z}(\mathbb{G})=\{x=\sum_{\pi\in\mathrm{Irr}(\mathbb{G})}c_{\pi}\chi_{\pi}\in\mathrm{Pol}(\mathbb{G}):c_{\pi}\in\mathbb{C}\}$
the subspace of central polynomials and let $C_r^z(\mathbb G)$ be the norm closure of $\mathrm{Pol}^{z}(\mathbb{G})$ in $C_r(\mathbb G)$.
\end{defn} 
\begin{rem}
Any Sidon set $\mathbf{E}\subset\mathrm{Irr}(\mathbb{G})$ is a central Sidon set for $\mathbb G$. In fact, note that for $x=\sum_{\pi}c_{\pi}\chi_{\pi}$ as in the above definition, we have $\hat{x}(\pi)=\dim(\pi)^{-1}c_{\pi}Q_{\pi}^{-1}$
and then by definition, 
\begin{equation}
\|\hat{x}\|_{1}=\sum_{\pi\in\irr}d_{\pi}\mathrm{Tr}(|p_{\pi}\hat{x}(\pi)Q_{\pi}|)=\sum_{\pi\in\irr}\dim(\pi)|c_{\pi}|.\label{eq:l1 norm for central}
\end{equation}
Hence $\mathbf{E}$ is a central Sidon set.
\end{rem}

\begin{lem}
Let $\mathbb{G}$ be a compact quantum group. The following assertions
are equivalent: 

\emph{(1)} $\mathbb{G}$ is of Kac type; 

\emph{(2)} any central functional $\omega$ is bounded on $\mathrm{Pol}(\mathbb{G})$ (with respect to $\|\,\|_\infty$)
if and only if it is bounded on $\mathrm{Pol}^{z}(\mathbb{G})$ with
the same norm;

\emph{(3) }there exists a conditional expectation $\mathcal{E}$ from
$C_{r}(\mathbb{G})$ onto $C_r^z(\mathbb G)$ such that $h\circ\mathcal{E}=h$.\end{lem}
\begin{proof}
(1) $\Rightarrow$ (2). Assume that $\mathbb{G}$ is of Kac type. Let
$\omega$ be a central functional which is bounded on $\mathrm{Pol}^{z}(\mathbb{G})$.
Let $\tilde{\omega}\in L^{\infty}(\mathbb{G})^{*}$ be its Hahn-Banach
extension to $L^{\infty}(\mathbb{G})$. Denote by $\mathcal{E}$
the $h$-preserving conditional expectation from $\mathcal{M}=L^{\infty}(\mathbb{G})$ onto
the von Neumann subalgebra $\mathcal{N}$ generated by $\mathrm{Pol}^{z}(\mathbb{G})$
in $L^{\infty}(\mathbb{G})$. Recall that for $\pi,\beta\in\mathrm{Irr}(\mathbb{G})$,
we have $\chi_{\pi}^{*}=\chi_{\bar{\pi}}$ and $\chi_{\pi}\chi_{\pi'}=\chi_{\pi\otimes\pi'}$,
so the subspace $\mathrm{Pol}^{z}(\mathbb{G})$ spanned by characters
is ultraweakly dense in $\mathcal{N}$. Note also that $\mathcal{E}$ is the
adjoint map of the embedding $\iota:L^{1}(\mathcal{N})\to L^{1}(\mathcal{M})$, so for
any $\pi,\pi'\in\mathrm{Irr}(\mathbb{G})$ and any $i,j$,
\begin{align*}
\langle\mathcal{E}(u_{ij}^{(\pi)}),\,\chi_{\pi'}\rangle_{L^{1}(\mathcal{N})^{*},L^{1}(\mathcal{N})} & =\langle u_{ij}^{(\pi)},\,\iota(\chi_{\pi'})\rangle_{L^{1}(\mathcal{M})^{*},L^{1}(\mathcal{M})}=h((u_{ij}^{(\pi)})^{*}\chi_{\pi'})=\dim(\pi)^{-1}\delta_{ij}\delta_{\pi\pi'}\\
 & =\dim(\pi)^{-1}\delta_{ij}h(\chi_{\pi}^{*}\chi_{\pi'})=\langle \dim(\pi)^{-1}\delta_{ij}\chi_{\pi},\,\chi_{\pi'}\rangle_{L^{1}(\mathcal{N})^{*},L^{1}(\mathcal{N})},
\end{align*}
which means that $E(u_{ij}^{(\pi)})=\delta_{ij}\dim(\pi)^{-1}\chi_{\pi}$.
Consequently, $\omega=\tilde{\omega}\circ E$ on $\mathrm{Pol}(\mathbb{G})$,
and therefore $\omega$ is bounded on $\mathrm{Pol}(\mathbb{G})$
with the same norm.

(2) $\Rightarrow$ (3). Define the linear map $\mathcal{E}:\mathrm{Pol}(\mathbb{G})\to \mathrm{Pol}^{z}(\mathbb{G})$
by 
\[
\mathcal{E}(u_{ij}^{(\pi)})=\delta_{ij}\dim(\pi)^{-1}\chi_{\pi},\quad\pi\in\mathrm{Irr}(\mathbb{G}).
\]
It is easy to see that $h\circ\mathcal{E}=h$ on $\mathrm{Pol}(\mathbb{G})$.
Also for any central functional $\omega$ on $\mathrm{Pol}(\mathbb{G})$,
we have $\omega\circ\mathcal{E}=\omega$. Now for any $x\in\mathrm{Pol}(\mathbb{G})$,
by the assertion (2), we have
\[
\|\mathcal{E}(x)\|=\sup_{\omega\in \mathrm{Pol}^{z}(\mathbb{G})^{*},\|\omega\|=1}|\omega(\mathcal{E}(x))|=\sup_{\omega\in\mathrm{Pol}(\mathbb{G})^{*},\|\omega\|=1}|\omega(x)|=\|x\|.
\]
So $\mathcal{E}$ is contractive on $\mathrm{Pol}(\mathbb{G})$ and
can be extended to a conditional expectation from $C_{r}(\mathbb{G})$
onto $C_r^z(\mathbb G)$ preserving the Haar state.

(3) $\Rightarrow$ (1). Assume (3) holds. It is a standard argument
that $\sigma_{t}(\mathrm{Pol}^{z}(\mathbb{G}))\subset \mathrm{Pol}^{z}(\mathbb{G})$,
see e.g. the proof of Theorem 4.2 in \cite[Chap.IX]{takesaki2003oa2}. In fact,
let $(H,\Lambda,\pi)$ be the faithful GNS construction of $\mathrm{Pol}(\mathbb{G})$
with respect to the Haar state $h$ and denote by $H_{0}$ the completion
of $\Lambda(\mathrm{Pol}^{z}(\mathbb{G})$ in $H$. Let $E$ be the
orthogonal projection from $H$ onto $H_{0}$. Then for $x\in\mathrm{Pol}(\mathbb{G}),y\in \mathrm{Pol}^{z}(\mathbb{G})$,
we have
\[
\langle\Lambda(x),\Lambda(y)\rangle=h(x^{*}y)=h(\mathcal{E}(x^{*}y))=h(\mathcal{E}(x)^{*}y)=\langle\Lambda(\mathcal{E}(x)),\Lambda(y)\rangle,
\]
so 
\[
E(\Lambda(x))=\Lambda(\mathcal{E}(x)),\quad x\in\mathrm{Pol}(\mathbb{G}).
\]
As usual denote by $\mathscr S:\Lambda(x)\mapsto \Lambda(x^*)$ the operator on $H$ induced by involution, and also denote by $\mathfrak \Delta$ the modular operator on $\mathrm{Pol}(\mathbb G)$ associated to the Haar state $h$. Since $\mathcal{E}$ preserves the $*$-operation, we see that $E\mathscr S=\mathscr S E$
on the subspace $\Lambda(\mathrm{Pol}(\mathbb{G}))$. Taking adjoint
we see that $\mathscr S^{*}$ also commutes with $E$ on $\Lambda(\mathrm{Pol}(\mathbb{G}))$,
hence so it is for $\mathfrak{\Delta}= \mathscr S ^{*}\mathscr{S}$. Therefore $\mathfrak\Delta^{\mathrm{i}t}$ leaves $\Lambda(\mathrm{Pol}^{z}(\mathbb{G})$
invariant for all $t\in\mathbb{R}$, which yields that $\sigma_{t}(\mathrm{Pol}^{z}(\mathbb{G}))\subset \mathrm{Pol}^{z}(\mathbb{G})$.
However we recall that the modular automorphism group acts on $\mathrm{Pol}(\mathbb{G})$
as 
\[
\sigma_{t}(u_{ij}^{(\pi)})=\sum_{k,l}(Q_{\pi}^{\mathrm{i}t})_{ik}u_{kl}^{(\pi)}(Q_{\pi}^{\mathrm{i}t})_{lj},\quad\pi\in\mathrm{Irr}(\mathbb{G}),\ t\in\mathbb{R}, 1\leq i,j \leq n_\pi
\]
and hence
\[
\sigma_{t}(\chi_{\pi})=\sigma_{t}\Big(\sum_{i}u_{ii}^{(\pi)}\Big)=\sum_{k,l}\Big(\sum_{i}(Q_{\pi}^{\mathrm{i}t})_{li}(Q_{\pi}^{\mathrm{i}t})_{ik}\Big)u_{kl}^{(\pi)}=\sum_{k,l}(Q_{\pi}^{2\mathrm{i}t})_{lk}u_{kl}^{(\pi)},\quad\pi\in\mathrm{Irr}(\mathbb{G})\ t\in\mathbb{R}.
\]
So the invariance $\sigma_{t}(\mathrm{Pol}^{z}(\mathbb{G}))\subset \mathrm{Pol}^{z}(\mathbb{G})$
yields that $Q_{\pi}=\mathrm{Id}_{\pi}$ for all $\pi\in\mathrm{Irr}(\mathbb{G})$,
that is, $\mathbb{G}$ is of Kac type.\end{proof}
\begin{prop}
\label{prop:central sidon_interpolation}Let $\mathbb{G}$ be a compact
quantum group of Kac type. Then $\mathbf{E}\subset\mathrm{Irr}(\mathbb{G})$
is a central Sidon set if and only if for all bounded sequences $(a_{\pi})_{\pi\in\mathbf{E}}\subset\mathbb{C}$,
there exists a bounded central functional $\varphi\in C_r(\mathbb{G})^{*}$
such that $\hat{\varphi}(\pi)=a_{\pi}\mathrm{Id}_{\pi}$
for $\pi\in\mathbf{E}$.\end{prop}
\begin{proof}
The proof is an analogue of that of Theorem \ref{thm:sidon_l1_interpolation state}
(1) $\Leftrightarrow$ (2). We replace $C_r(\mathbb{G})$ by $C_r^z(\mathbb G)$
and $\ell^{\infty}(\mathbf{E})$ by $(\ell^{\infty}(\mathbf{E}))^z=\{(a_{\pi})\in\ell^{\infty}(\mathbf{E}):a_{\pi}\in\mathbb{C}\mathrm{Id}_{\pi},\pi\in\mathrm{Irr}(\mathbb{G})\}$,
etc., and argue as before, thanks to the previous lemma.\end{proof}

\begin{rem}
There have been some suggestions on the definition of unconditional Sidon sets for a compact (quantum) group which would be different from that in Definition \ref{def:interpolation set infty}. More precisely, for a compact quantum group $\mathbb G$ and a subset $\mathbf{E}\subset\mathrm{Irr}(\mathbb{G})$, we may consider the following lacunary condition: there exists a constant $K>0$ such that
\begin{equation}\label{eq:another uncond sidon}
\forall (\varepsilon_\pi)_{\pi\in\mathbf E}\subset \{-1,1\},\quad \|\sum_{\pi\in \mathbf E}d_\pi\varepsilon_\pi (\iota\otimes\mathrm{Tr})((1 \otimes \hat x (\pi)Q_\pi)u^{(\pi)})\|_\infty \leq K \|x\|_\infty,\quad x\in \mathrm{Pol}_{\mathbf E}(\mathbb G).
\end{equation}
And it was ever unclear on the relations between  \eqref{eq:another uncond sidon} and the Sidon sets even if $\mathbb G$ is a compact group $G$. Here we remark that the above two notions of lacunarity are in fact totally different if $G$ is non-abelian. Indeed, more generally, if the compact quantum group $\mathbb G$ is coamenable and of Kac type, we may follow the same idea as in Theorem \ref{sidon and u sidon} and the above proposition to see that  \eqref{eq:another uncond sidon} holds if and only if $\mathbf{E}\subset\mathrm{Irr}(\mathbb{G})$ is a central Sidon set for $\mathbb G$, which is well-known to be different from being a Sidon set. 
\end{rem}
\begin{rem}
For a central functional $\omega$ on $\mathrm{Pol}(\mathbb{G})$,
the associated multiplier map $T_{\omega}=(\omega\otimes\mathrm{id})\circ\Delta$
acts as $u_{ij}^{(\pi)}\mapsto\omega_{\pi}u_{ij}^{(\pi)}$ on
$\mathrm{Pol}(\mathbb{G})$. In \cite[Sect.2]{decommerfry2014accap}
it is proved that if $\Phi:\mathrm{Irr}(\mathbb{G}_{1})\to\mathrm{Irr}(\mathbb{G}_{2})$
is a monoidal equivalence between two compact quantum groups $\mathbb{G}_{1}$
and $\mathbb{G}_{2}$ and if $\omega^{(1)},\omega^{(2)}$ are the
central functionals on $\mathrm{Pol}(\mathbb{G}_{1}),\mathrm{Pol}(\mathbb{G}_{2})$
respectively, such that $\omega_{\pi}^{(1)}=\omega_{\Phi(\pi)}^{(2)}$,
then $T_{\omega^{(1)}}$ and $T_{\omega^{(2)}}$ have the same complete
bounded norms. Note that if additionally $\mathbb{G}_{1}$ and $\mathbb{G}_{2}$
are coamenable, then $\|\omega^{(1)}\|=\|T_{\omega^{(1)}}\|=\|T_{\omega^{(2)}}\|=\|\omega^{(2)}\|$
since $\omega=\omega\star\epsilon=(\mathrm{id}\otimes\epsilon)\circ T_{\omega}$
for any functional $\omega$ on $\mathrm{Pol}(\mathbb{G})$. Then
according to the previous proposition, any two coamenable compact
quantum groups of Kac type which are monoidally equivalent, have a one-to-one
correspondence of their central Sidon sets via the monoidal equivalence
map. The following result generalizes this fact.\end{rem}
\begin{prop}\label{same central sidon}
Let $\mathbb{G}_{1},\mathbb{G}_{2}$ be two compact quantum groups.
Assume that $\Phi:\mathrm{Rep}(\mathbb{G}_{1})\to\mathrm{Rep}(\mathbb{G}_{2})$
is an injective map preserving the fusion rules, that
is, for all $\pi,\pi'\in\mathrm{Rep}(\mathbb{G}_{1})$ we have
$$\Phi(\pi\otimes\pi')=\Phi(\pi)\otimes \Phi(\pi'),\quad \Phi(\oplus_{i=1}^n\pi_i)=\oplus_{i=1}^n\Phi(\pi_i), \quad \pi,\pi',\pi_i\in \mathrm{Rep}(G),n\geq 1.$$
Then for any finite sequence $(c_{\pi})\subset\mathbb{C}$ and
$x=\sum_{\pi\in\mathrm{Irr}(\mathbb{G}_{1})}c_{\pi}\chi_{\pi}\in\mathrm{Pol}(\mathbb{G}_{1})$,
we have
\begin{equation}
\|\sum_{\pi\in\mathrm{Irr}(\mathbb{G}_{1})}c_{\pi}\chi_{\pi}\|_{L^{\infty}(\mathbb{G}_{1})}=\|\sum_{\pi\in\mathrm{Irr}(\mathbb{G}_{1})}c_{\pi}\chi_{\Phi(\pi)}\|_{L^{\infty}(\mathbb{G}_{2})}.\label{eq:isometry for central functions}
\end{equation}
Consequently, for any central Sidon set $\mathbf{E}\subset\mathrm{Irr}(\mathbb{G}_{1})$,
if additionally there exists $C>0$ satisfying $\dim(\Phi(\pi))\leq C\dim(\pi)$
for all $\pi\in\mathbf{E}$, then $\Phi(\mathbf{E})\subset\mathrm{Irr}(\mathbb{G}_{2})$
is a central Sidon set for $\mathbb{G}_{2}$. \end{prop}
\begin{proof}
The isometry \eqref{eq:isometry for central functions} has been mentioned in \cite{banica99fusion}. In fact, as remarked in \cite{banica99fusion}, the map $\Phi$ automatically satisfies $\Phi(\bar\pi)=\overline{\Phi(\pi)}$ for all $\pi\in\mathrm{Rep}(\mathbb G)$ once we note that for an irreducible representation $\pi$ its adjoint $\bar{\pi}$ is the unique irreducible representation such that $1$ is the subrepresentation of $\pi\otimes\bar{\pi}$. Denote by $A$ the $*$-algebra generated by $\{\chi_{\pi}:\pi\in\mathrm{Irr}(\mathbb{G}_{1})\}$.
By the linear independence of $\{\chi_{\pi}:\pi\in\mathrm{Irr}(\mathbb{G}_{1})\}$, the injection $\Phi$ induces an injective $*$-homomorphism 
\[
\tilde{\Phi}:A\to\mathrm{Pol}(\mathbb{G}_{2}),\quad\chi_{\pi}\mapsto\chi_{\Phi(\pi)},\ \pi\in\mathrm{Rep}(\mathbb{G}_{1}).
\]
To see that $\tilde{\Phi}$ well defines a $*$-homomorphism, it
suffices to notice that for $\pi\in\mathrm{Irr}(\mathbb{G}_{1})$,
we have $\chi_{\pi}^{*}=\chi_{\bar{\pi}}$ and 
\[
\tilde{\Phi}(\chi_{\pi}^{*})=\tilde{\Phi}(\chi_{\bar{\pi}})=\chi_{\Phi(\bar{\pi})}=\chi_{\overline{\Phi(\pi)}}=\tilde{\Phi}(\chi_{\pi})^{*}
\]
and for $\pi,\pi'\in\mathrm{Rep}(\mathbb{G}_{1})$ satisfying the decomposition formula $\pi\otimes\pi'=\sum_k\pi_k$ with each $\pi_k\in\mathrm{Irr}(\mathbb G _1 )$, we have
\[
\tilde{\Phi}(\chi_{\pi}\chi_{\pi'})=\sum_{k}\tilde{\Phi}(\chi_{\gamma_{k}})=\sum_{k}\chi_{\Phi(\gamma_{k})}=\chi_{\Phi(\pi)}\chi_{\Phi(\pi')}=\tilde{\Phi}(\chi_{\pi})\tilde{\Phi}(\chi_{\pi'}).
\]
Note that $\tilde{\Phi}$ preserves the restriction of Haar states on $\mathrm{Pol}^{z}
(\mathbb{G}_1)$ and
$\mathrm{Pol}^{z}
(\mathbb{G}_2)$. So $\tilde{\Phi}$ gives rise to an equivalence between the faithful sub-GNS-representations of $\mathrm{Pol}^z(\mathbb
G_1)\subset \mathrm{Pol}
(\mathbb G_1)$ and
$\mathrm{Pol}^z(\mathbb
G_2)\subset \mathrm{Pol}
(\mathbb G_2)$ with respect to the Haar states.  As a result $\tilde{\Phi}$ is an isometry, which gives \eqref{eq:isometry for central functions}.
The assertion regarding the central Sidon sets then follows directly
from the definition and \eqref{eq:l1 norm for central}.\end{proof}
\begin{example}
Let $q\in[-1,1]\backslash\{0\}$ and consider the quantum groups $\mathrm{SU}_{q}(n)$,
$n\geq2$. $\mathrm{SU}_{q}(n)$ is a compact matrix quantum group in the sense of \cite{woronowicz87matrix} and we denote by $\pi_{n}^{(q)}$ the fundamental representation
of $\mathrm{SU}_{q}(n)$. It is easy to see that for the classical
case $q=1$, the subset $\{\pi_{n}^{(1)}:n\geq2\}$ is a Sidon set
for the compact group $\prod_{n\geq2}\mathrm{SU}(n)$ (see \cite[p.308-310]{cm81existsidon}).
Now by Proposition \ref{deformation fusion rule} for $0<q <1$
we may find a map $\Phi:\mathrm{Irr}(\mathrm{SU}_{q }(n))\to\mathrm{Irr}(\mathrm{SU}(n))$
satisfying the assumptions of the above proposition. As a result,
for any $0<q <1$  the subset $\{\pi_{n}^{(q)}:n\geq2\}$
is a central Sidon set for the compact quantum group $\prod_{n\geq2}\mathrm{SU}_{q}(n)$.
\end{example}
In \cite{rider72central} it is shown that for a connected compact
group $G$, $G$ has an infinite central Sidon set if and only if
$G$ is not a semi-simple Lie group. Combined with Proposition \ref{deformation fusion rule} and Proposition \ref{same central sidon} we get the following observation, which in particular shows that the quantum $\mathrm{SU}_{q}(2)$ ($0<q<1$) does not admit any infinite Sidon set, as mentioned in the last section after Proposition \ref{central lambda four su}.
\begin{cor}\label{cor: Gq no sidon}
For all $0<q<1$ and all simply connected compact semi-simple Lie group $G$, the compact quantum group $G_{q}$ given by the Drinfeld-Jimbo deformation
does not admit any infinite central Sidon set.
\end{cor}

\appendix
\section*{Appendix: Existence of $\Lambda(p)$-sets in orthogonal systems for general noncommutative $L^p$-spaces}

\renewcommand{\thesection}{A}
\setcounter{thm}{0}
\setcounter{equation}{0}

In this appendix we present a method of constructing $\Lambda(p)$-sets in orthogonal systems for noncommutative $L^p$-spaces. The main result is the following theorem. It is due to  Marek Bo\.{z}ejko in the tracial case.

\begin{thm}\label{thm: exist lambda a1}
Let $\mathcal{M}$ be a von Neumann algebra equipped with a normal faithful state $\varphi$ and consider the associated $L^p$-spaces $L^p(\mathcal{M},\varphi)$. Let $B=\{x_i\in \mathcal{M}:i\geq 1\} $ be an orthogonal system with respect to $\varphi$ (i.e., $\varphi(x_i^*x_j)=0$ for all $i\neq j$) such that $\sup_i\|x_i\|_\infty<\infty$. Then for each $2< p <\infty$, there exists an infinite subset $\{x_{i_k}:k\geq 1\}\subset B$ and a constant $C>0$ such that for all finitely supported sequences $(c_k)\subset \mathbb C$ we have
$$\|\sum_{k\geq 1}c_kx_{i_k}\|_p \leq C \Big(\sum_{k\geq 1}|c_k|^2\Big)^{\frac{1}{2}}.$$
\end{thm}   

This result was first proved by \cite{kaczmarzsteinhaus36ortho} in the commutative case and then by \cite{bozejko79lambdap} in the case where $\varphi$ is tracial. The same idea also applies to the general setting, and for the sake of completeness we include a detailed proof below. We refer to \cite{picardello73lacunary,bozejko73lambdapgroup,bozejko75sqrtp} for different approaches.

To establish the theorem, let us show the following slightly stronger result.
\begin{thm}
Let $\mathcal{M}$ and $\varphi$ be given as in Theorem \ref{thm: exist lambda a1} and $n\geq 2$. Let $B=\{x_i\in L^{2n}(\mathcal{M},\varphi):i\geq 1\} $ be an orthogonal system with respect to $\varphi$ such that $\sup_i\|x_i\|_{2n}<\infty$. Then  there exists an infinite subset $\{x_{i_k}:k\geq 1\}\subset B$ and a constant $C>0$ such that for all finitely supported sequences $(c_k)\subset \mathbb C$ we have
\begin{equation}
\label{eq: lambda 2n ineq}
\|\sum_{k\geq 1}c_kx_{i_k}\|_{2n} \leq C \Big(\sum_{k\geq 1}|c_k|^2\Big)^{\frac{1}{2}}.
\end{equation}
\end{thm}
\begin{proof}
Write $K=\max\{1,\sup_i\|x_i\|_{2n}\}<\infty$. Denote by $\mathcal{M}_a$ the subalgebra of all analytic elements in $\mathcal{M}$. Recall that $\mathcal{M}_a$ a ultraweakly dense subspace of $\mathcal{M}$, and for all $x\in \mathcal{M}_a$,  the analytic extension $\sigma_z(x)\in \mathcal{M}$ for $z\in\mathbb C$ is well-defined (see \cite{takesaki2003oa2}). In particular $\mathcal{M}_a$ is also a dense subspace of $L^p(\mathcal{M},\varphi)$ for $1\leq p <\infty$ according to Lemma \ref{prop: poly dense in L_1}.

(1) Firstly, assume that $B\subset \mathcal{M}_a$. Note that for any $y\in \mathcal{M}_a$,
$$\varphi(yx_k)\to 0,\quad \varphi(x_k y)\to 0,\quad k\to\infty.$$ 
We choose a subset $\{x_{i_k}:k\geq 1\}\subset B$ inductively as follows. Let $x_{i_1}=x_1$ and if $\{x_{i_j}:1\leq j\leq k\}$ for $k\geq 1$ is chosen, we take an $x_{i_{k+1}}\in B$ such that for all $1\leq k_0,k_1,\ldots,k_{n-1},l_1,\ldots,l_{n-1}\leq k$, 
$$\left|\varphi\left( \sigma_{\frac{(n-1)\mathrm i}{n}}(x_{i_{k_1}}^*x_{i_{l_1}})
\sigma_{\frac{(n-2)\mathrm i}{n}}(x_{i_{k_2}}^*x_{i_{l_2}})
\cdots
\sigma_{\frac{\mathrm i}{n}}(x_{i_{k_{n-1}}}^*x_{i_{l_{n-1}}})
x_{i_{k_0}}^*x_{i_{k+1}}\right)\right|\leq 
\dfrac{1}{2k^{2n-1}(k+1)}
$$
and
$$\left|\varphi\left(\sigma_{\mathrm i}(x_{i_{k_0}}) \sigma_{\frac{(n-1)\mathrm i}{n}}(x_{i_{k_1}}^*x_{i_{l_1}})
\sigma_{\frac{(n-2)\mathrm i}{n}}(x_{i_{k_2}}^*x_{i_{l_2}})
\cdots
\sigma_{\frac{\mathrm i}{n}}(x_{i_{k_{n-1}}}^*x_{i_{l_{n-1}}})
x_{i_{k+1}}^*\right)\right|\leq 
\dfrac{1}{2k^{2n-1}(k+1)}.
$$
This can always be done by the orthogonality of $B$.  In fact, since the elements in $B$ are assumed to be analytic, we have 
$$\sigma_{\frac{(n-1)\mathrm i}{n}}(x_{i_{k_1}}^*x_{i_{l_1}})
\sigma_{\frac{(n-2)\mathrm i}{n}}(x_{i_{k_2}}^*x_{i_{l_2}})
\cdots
\sigma_{\frac{\mathrm i}{n}}(x_{i_{k_{n-1}}}^*x_{i_{l_{n-1}}})
x_{i_{k_0}}^*\in \mathcal{M} (\subset L^2(\mathcal{M},\varphi)),$$
$$\sigma_{\mathrm i}(x_{i_{k_0}}) \sigma_{\frac{(n-1)\mathrm i}{n}}(x_{i_{k_1}}^*x_{i_{l_1}})
\sigma_{\frac{(n-2)\mathrm i}{n}}(x_{i_{k_2}}^*x_{i_{l_2}})
\cdots
\sigma_{\frac{\mathrm i}{n}}(x_{i_{k_{n-1}}}^*x_{i_{l_{n-1}}})\in \mathcal{M} (\subset L^2(\mathcal{M},\varphi)).$$
Note that there exist only finitely many elements of the above forms for each given $k\geq 1$. So the element $x_{i_{k+1}}$ can be well chosen. With the same notation we deduce the following inequalities,
\begin{align*}
& \quad\,\left|\mathrm{tr} \left( D^{\frac{1}{2n}}x_{i_{k_1}}^*x_{i_{l_1}}D^{\frac{1}{2n}}D^{\frac{1}{2n}}x_{i_{k_2}}^*x_{i_{l_2}}D^{\frac{1}{2n}}
\cdots
D^{\frac{1}{2n}}x_{i_{k_{n-1}}}^*x_{i_{l_{n-1}}}D^{\frac{1}{2n}}D^{\frac{1}{2n}}
x_{i_{k_0}}^*x_{i_{k+1}}D^{\frac{1}{2n}}\right)\right|\\
& = \left|\mathrm{tr}\left(D \sigma_{\frac{(n-1)\mathrm i}{n}}(x_{i_{k_1}}^*x_{i_{l_1}})
\sigma_{\frac{(n-2)\mathrm i}{n}}(x_{i_{k_2}}^*x_{i_{l_2}})
\cdots
\sigma_{\frac{\mathrm i}{n}}(x_{i_{k_{n-1}}}^*x_{i_{l_{n-1}}})
x_{i_{k_0}}^*x_{i_{k+1}}\right)\right|\\
& =\left|\varphi\left( \sigma_{\frac{(n-1)\mathrm i}{n}}(x_{i_{k_1}}^*x_{i_{l_1}})
\sigma_{\frac{(n-2)\mathrm i}{n}}(x_{i_{k_2}}^*x_{i_{l_2}})
\cdots
\sigma_{\frac{\mathrm i}{n}}(x_{i_{k_{n-1}}}^*x_{i_{l_{n-1}}})
x_{i_{k_0}}^*x_{i_{k+1}}\right)\right|
\leq 
\dfrac{1}{2k^{2n-1}(k+1)},
\end{align*}
and similarly
$$\left|\mathrm{tr} \left( D^{\frac{1}{2n}}x_{i_{k_1}}^*x_{i_{l_1}}D^{\frac{1}{2n}}
\cdots
D^{\frac{1}{2n}}x_{i_{k_{n-1}}}^*x_{i_{l_{n-1}}}D^{\frac{1}{2n}}D^{\frac{1}{2n}}
x_{i_{k+1}}^*x_{i_{k_0}}D^{\frac{1}{2n}}\right)\right|
\leq 
\dfrac{1}{2k^{2n-1}(k+1)}.$$
Fix a sequence $(c_k) \subset  \mathbb C$ such that 
$\sum_{k\geq 1}|c_k|^2=1 $.
Take $r\geq 1$ and write
$$g_r=\sum_{k=1}^{r}c_k x_{i_k}.$$
Since then each $|c_k|\leq 1$ for $k\geq 1$, the above inequalities yield that
\begin{equation}\label{eq: lambda two n estimate single x 1}
\left|\mathrm{tr} \left( (D^{\frac{1}{2n}}g_r^*g_rD^{\frac{1}{2n}})^{n-1}D^{\frac{1}{2n}}g_r^*x_{i_{r+1}}D^{\frac{1}{2n}}
\right)\right|\leq r^{2n-1}\cdot \dfrac{1}{2(r+1)r^{2n-1}}= \dfrac{1}{2(r+1)},
\end{equation}
and similarly
\begin{equation}\label{eq: lambda two n estimate single x 2}
\left|\mathrm{tr} \left( (D^{\frac{1}{2n}}g_r^*g_rD^{\frac{1}{2n}})^{n-1}D^{\frac{1}{2n}}x_{i_{r+1}}^*g_rD^{\frac{1}{2n}}
\right)\right|\leq  \dfrac{1}{2(r+1)}.
\end{equation}
Now for each $r\geq 1$ write $I_r = \|\sum_{k=1}^{r}c_kx_{i_k}\|_{2n}^{2n}$. Then
$$I_{r+1}=\|\sum_{k=1}^{r+1}c_kx_{i_k}\|_{2n}^{2n}=\|g_r+c_{r+1}x_{i_{r+1}}\|_{2n}^{2n}=\mathrm{tr}\left( (D^{\frac{1}{2n}}(g_r+c_{r+1}x_{i_{r+1}})^*(g_r+c_{r+1}x_{i_{r+1}})D^{\frac{1}{2n}})^n\right).$$
We may write
\begin{equation}\label{eq: lambda p existence decomp}
I_{r+1}=J_1+J_2+J_3+J_4,
\end{equation}
where
$$J_1=\mathrm{tr}\left( (D^{\frac{1}{2n}}g_r^*g_rD^{\frac{1}{2n}})^n\right)=I_r,\quad J_2=\mathrm{tr}\left( (D^{\frac{1}{2n}}(c_{r+1}x_{i_{r+1}})^*(c_{r+1}x_{i_{r+1}})D^{\frac{1}{2n}})^n \right),
%=|c_{r+1}|^{2n}\|x_{i_{r+1}}\|_{2n}^{2n}
$$
$$J_3=\sum_{\underline{y}\in A_1}
\mathrm{tr} \left( D^{\frac{1}{2n}}y_{1}^*y_{1}'D^{\frac{1}{2n}}
D^{\frac{1}{2n}}y_{2}^*y_{2}'D^{\frac{1}{2n}}
\cdots
D^{\frac{1}{2n}}y_{n}^*y_{n}'D^{\frac{1}{2n}}\right),$$
$$J_4=\sum_{j=2}^{n-1}\sum_{\underline{y}\in A_j}
\mathrm{tr} \left( D^{\frac{1}{2n}}y_{1}^*y_{1}'D^{\frac{1}{2n}}
D^{\frac{1}{2n}}y_{2}^*y_{2}'D^{\frac{1}{2n}}
\cdots
D^{\frac{1}{2n}}y_{n}^*y_{n}'D^{\frac{1}{2n}}\right),$$
with
\begin{align*}
A_j =\{&
\underline{y}
=(y_1,\ldots,y_n,
y_1',\ldots,y_n'):
y_k,
y_k'\in\{g_r,c_{r+1}x_{i_{r+1}}\}, 1\leq k\leq n,\\
& \mathrm{Card}\{k: y_k=x_{i_{r+1}}\}+\mathrm{Card}\{k: y_k'=x_{i_{r+1}}\}=j \}.
\end{align*}
We see that
\begin{equation}\label{eq: lambda p exsit j2}
|J_2|=|c_{r+1}|^2\|x_{i_{r+1}}\|_{2n}^{2n} \leq K^{2n} |c_{r+1}|^2.
\end{equation}
By traciality of $\mathrm{tr}$ and by \eqref{eq: lambda two n estimate single x 1}-\eqref{eq: lambda two n estimate single x 2} we have
\begin{align}
\label{eq: lambda p exsit j3}
|J_3| &=\Big|n c_{r+1} \mathrm{tr} \left( (D^{\frac{1}{2n}}g_r^*g_rD^{\frac{1}{2n}})^{n-1}D^{\frac{1}{2n}}g_r^*x_{i_{r+1}}D^{\frac{1}{2n}}
\right) \\
& \qquad + n c_{r+1} 
\mathrm{tr} \left( (D^{\frac{1}{2n}}g_r^*g_rD^{\frac{1}{2n}})^{n-1}D^{\frac{1}{2n}}x_{i_{r+1}}^*g_rD^{\frac{1}{2n}}
\right)
\Big| \nonumber\\ 
&\leq \frac{n|c_{r+1}|}{r+1}. \nonumber
\end{align}
By the H\"older inequality
$$\left|
\mathrm{tr} \left( D^{\frac{1}{2n}}y_{1}^*y_{1}'D^{\frac{1}{2n}}
D^{\frac{1}{2n}}y_{2}^*y_{2}'D^{\frac{1}{2n}}
\cdots
D^{\frac{1}{2n}}y_{n}^*y_{n}'D^{\frac{1}{2n}}\right)
\right|
\leq 
\prod_{k=1}^{n}\|y_k\|_{2n}\|y_k'\|_{2n},$$
so
\begin{equation}
\label{eq: lambda p exist j4}
|J_4|\leq \sum_{j=2}^{n}2^{2n}\|g_r\|_{2n}^{n-j}\|c_{r+1}x_{i_{r+1}}\|_{2n}^j\leq K^n 4^nn|c_{r+1}|^2(I_r+1).
\end{equation}
Combining \eqref{eq: lambda p existence decomp}-\eqref{eq: lambda p exist j4} we have
$$I_{r+1}\leq 
(1+K^n 4^nn|c_{r+1}|^2)I_r + (K^{2n}+K^n 4^nn)|c_{r+1}|^2+\frac{n|c_{r+1}|}{r+1}.$$
Iterating the above inequality and denoting $a_k=K^n 4^nn|c_{k}|^2$, $b_k=(K^{2n}+K^n 4^nn)|c_{k}|^2+\frac{n|c_{k}|}{k}$ for each $k\geq 1$, we obtain for $r\geq 2$
\begin{align*}
I_r &\leq \prod_{k=2}^{r}(1+a_k)I_1+
\sum_{l=2}^{r-1}b_l\prod_{j=l+1}^{r}(1+a_j)+b_r\\
& \leq e^{\sum_{k\geq 1}a_k} I_1 + e^{\sum_{k\geq 1}a_k} \sum_{l\geq 1}b_l 
\leq (2K^{2n}+K^n 4^nn+n\sum_{l\geq 1}l^{-2})e^{K^n4^nn}.
\end{align*}
The above right hand side is a constant only depending on $K$ and $n$, so the desired inequality \eqref{eq: lambda 2n ineq} is proved.

(2) Now choose an arbitrary family $B=\{x_i\in L^{2n}(\mathcal{M},\varphi):i\geq 1\} $ satisfying the assumption. Write $p=2n$. Without loss of generality we take $K\geq 1$ in the sequel. By the density of $\mathcal{M}_a$ in $L^p(\mathcal{M},\varphi)$, we choose the sequences $(x_k'),(y_k)\subset \mathcal{M}_a$ and $(j_k)_{k\geq 1}\subset \mathbb N$  inductively as follows. Let $j_1=1$ and $x_1'=y_1\in \mathcal{M}_a$ such that
$$\|x_1-x_1'\|_p\leq 2^{-1}.$$
If $x_l',y_l,j_l$ for $1\leq l\leq k$ are chosen, take by orthogonality a $j_{k+1}\geq j_k$ such that for all $1\leq l\leq k$,
$$|\varphi(y_l^*x_{j_{k+1}})|\leq 2^{-k-l-2}\|y_l\|_2^2\|y_l\|_p^{-1} $$ 
and by the density property choose $x_{k+1}'\in \mathcal{M}_a$ such that 
$$ \|x_{j_{k+1}}-x_{k+1}'\|_p\leq 2^{-k-2}\Big(1+\sum_{l=1}^{k}\frac{\|y_l\|_p }{\|y_l\|_2}\Big)^{-1},$$
where we always make the standard convention in this proof that the terms such as $\|y_l\|_2^2/\|y_l\|_p^{-1}$ and $\|y_l\|_p/\|y_l\|_2^{-1}$ are regarded as zero if $y_l=0$.
Also set
$$y_{k+1}=x_{k+1}'-\sum_{l=1}^{k}\varphi(y_l^*y_l)^{-1}\varphi(y_l^*x_{k+1}')y_l.$$
We see immediately that $(y_k)$ is an orthogonal system. Moreover, when we recall the orthogonality of $(x_k)$, by the choice of the above sequences we have
\begin{align}\label{eq: lambda exist non analytic estimation}
\|x_{j_{k+1}}-y_{k+1}\|_p
& \leq \|x_{j_{k+1}}-x_{k+1}'\|_p
+\sum_{l=1}^{k}\frac{\|y_l\|_p }{\|y_l\|_2^{2}}|\varphi(y_l^*x_{k+1}')|\\
& = \|x_{j_{k+1}}-x_{k+1}'\|_p
+\sum_{l=1}^{k}\frac{\|y_l\|_p }{\|y_l\|_2^{2}}|\varphi(y_l^*(x_{k+1}'-x_{j_{k+1}}))
+\varphi(y_l^*x_{j_{k+1}})|\nonumber\\
& \leq \|x_{j_{k+1}}-x_{k+1}'\|_p
+\sum_{l=1}^{k}\frac{\|y_l\|_p }{\|y_l\|_2^{2}}
\|y_l\|_2\|x_{j_{k+1}}-x_{k+1}'\|_2+\sum_{l=1}^{k}2^{-k-l-2} \nonumber\\
& \leq \|x_{j_{k+1}}-x_{k+1}'\|_p
\Big(1+\sum_{l=1}^{k}\frac{\|y_l\|_p }{\|y_l\|_2}\Big)+
2^{-k-2}
\leq 2^{-k-1}.\nonumber
\end{align}
By the argument in (1), the orthogonal system $\{y_k:k\geq 1\}\subset \mathcal{M}_a$ admits a subset $\{y_{k_r}:r\geq 1\}$ satisfying the estimation of the type \eqref{eq: lambda 2n ineq}. Now it is easy to observe that $\{x_{j_{k_r}}:r\geq 1\}\subset B$ is the desired subset. To see this, it suffices to note that, there exists a constant $C>0$ such that for all sequences $(c_r)_{r\geq 1}\subset \mathbb C$ with $\sum_r|c_r|^2=1$ and for all $s\geq 1$,
$$\|\sum_{r=1}^{s}c_ry_{k_r}\|_p \leq C$$
and hence by the estimates \eqref{eq: lambda exist non analytic estimation},
$$ \|\sum_{r=1}^{s}c_r x_{j_{k_r}}\|_p \leq \|\sum_{r=1}^{s}c_ry_{k_r}\|_p +
\sum_{r=1}^{s}|c_r|\|x_{j_{k_r}}-y_{k_r}\|_p
\leq C+\sum_{r=1}^{s}|c_r|2^{-r} \leq C+1.$$
So the proof is complete.
\end{proof}
\begin{remn}
In general, the noncommutative $L^p$-spaces introduced by \cite{kosaki84interpolation} depend on different interpolation parameters. Following the same line of the proof one can see that the above result still holds for those types of $L^p$-spaces.
\end{remn}

\subsection*{Acknowledgment}
The author is indebted to his advisors Quanhua Xu and Adam Skalski for their helpful discussions and constant encouragement. He would also like to thank
Professors Marek Bożejko
and Marius Junge for
fruitful communications
and suggestions, and thank
Professor Gilles Pisier for
his careful reading and
pointing out a mistake in
the preprint version. The research was partially supported by  the ANR Project ANR-11-BS01-0008 and the NCN (National Centre of Science) grant 2014/14/E/ST1/00525.

\end{document}